\documentclass[11pt,reqno]{NumPDEsArticle}

\usepackage[basicdelimiters]{NumPDEsMacros}

\usepackage[english]{babel}

\usepackage{csquotes}
\usepackage{enumitem}
\usepackage{amsmath, amssymb}

\usepackage{pdftexcmds}

\usepackage{microtype}
\usepackage{graphicx}
\usepackage[labelformat=simple]{subcaption}

\usepackage{tikz}
\usetikzlibrary{calc,matrix}
\usepackage{pgfplots}
\pgfplotsset{compat=newest}
\usepackage{pgfplotstable}
\usepackage[export]{adjustbox}%

\definecolor{TUblue}{rgb}{0,0.4,0.6}
\definecolor{TUgray}{rgb}{0.3922,0.3882,0.3882}
\definecolor{TUgreen}{rgb}{0,0.4941,0.4431}
\definecolor{TUmagenta}{rgb}{0.7294,0.2745,0.5098}
\definecolor{TUyellow}{rgb}{0.8824,0.5373,0.1333}

\newcommand{\Eta}{\mathsf{H}}

\newcommand{\coarse}{H}
\newcommand{\fine}{h}
\newcommand{\exact}{\star}
\newcommand{\elll}{{\underline{\ell}}}

\newcommand{\kk}{{\underline{k}}}

\newcommand{\rr}{{\overline{r}}}

\newcommand{\kmax}{{\underline{k}}}
\newcommand{\kmin}{{k_{\textup{min}}}}

\newcommand{\lmax}{{\underline{\ell}}}

\newcommand{\thetamark}{\const{\theta}{mark}}

\newcommand{\lambdalin}{\const{\lambda}{lin}}

\newcommand\f{\boldsymbol{f}}
\newcommand\n{\boldsymbol{n}}

\newcommand\Bup{{\textup{B}}}

\newcommand\Lup{{\textup{L}}}
\newcommand\Mup{{\textup{M}}}

\newcommand\hook{\hookrightarrow}

\newcommand\drawslopetriangle[4][ST]{
  \pgfplotsextra
  {
    \pgfkeys{/pgf/fpu=true}
    \pgfmathsetmacro\leftcoord{#3}
    \pgfmathsetmacro\rightcoord{10*#3}
    \pgfmathsetmacro\bottomcoord{#4}
    \pgfmathsetmacro\topcoord{10^(#2)*#4}
    \pgfkeys{/pgf/fpu=false}

    \coordinate (#1-BL) at (axis cs:\leftcoord,\bottomcoord);
    \coordinate (#1-BR) at (axis cs:\rightcoord,\bottomcoord);
    \coordinate (#1-TL) at (axis cs:\leftcoord,\topcoord);

    \shadedraw[%
      bottom color = black!20,%
      middle color = black!5,%
      top color    = white,%
      draw         = black%
    ]
      (#1-TL) -- (#1-BL) node[midway, left=-2pt] {\scriptsize\(#2\)}
      -- (#1-BR) node[midway, below=-2pt] {\scriptsize\(1\)} -- (#1-TL);
  }
}
\newcommand\drawswappedslopetriangle[4][SST]{
  \pgfplotsextra
  {
    \pgfkeys{/pgf/fpu=true}
    \pgfmathsetmacro\leftcoord{#3/10}
    \pgfmathsetmacro\rightcoord{#3}
    \pgfmathsetmacro\topcoord{#4}
    \pgfmathsetmacro\bottomcoord{10^(-#2)*#4}
    \pgfkeys{/pgf/fpu=false}

    \coordinate (#1-TR) at (axis cs:\rightcoord,\topcoord);
    \coordinate (#1-BR) at (axis cs:\rightcoord,\bottomcoord);
    \coordinate (#1-TL) at (axis cs:\leftcoord,\topcoord);

    \shadedraw[%
      bottom color = black!20,%
      middle color = black!5,%
      top color    = white,%
      draw         = black%
    ]
      (#1-BR) -- (#1-TR) node[midway, right=-2pt] {\scriptsize\(#2\)}
      -- (#1-TL) node[midway, above=-2pt] {\scriptsize\(1\)} -- (#1-BR);
  }
}

\makeatletter
\newcommand{\strequal}[2]{\pdf@strcmp{#1}{#2}==0}
\makeatother

\def\A{\boldsymbol{A}}

\makeatletter
\newcommand{\labeltext}[2]{%
    \@bsphack
    \csname phantomsection\endcsname 
    \def\@currentlabel{#1}{\label{#2}}%
    \@esphack
}
\makeatother

\title[Newton AILFEM]{Newton's method in adaptive iteratively linearized FEM}

\usepackage{orcidlink}

\address{TU Wien, Institute of Analysis and Scientific Computing, Wiedner Hauptstr. 8--10/E101/4, 1040 Vienna, Austria}

\author{Philipp Bringmann~\orcidlink{0000-0002-4546-5165}}
\email{philipp.bringmann@asc.tuwien.ac.at}

\author{Maximilian Brunner~\orcidlink{0000-0003-0636-1491}}
\email{maximilian.brunner@asc.tuwien.ac.at \quad \rm (corresponding author)}

\author{Dirk Praetorius~\orcidlink{0000-0002-1977-9830}}
\email{dirk.praetorius@asc.tuwien.ac.at}

\thanks{%
    This research was funded in whole or in part by the Austrian Science Fund (FWF)
    [\href{https://www.fwf.ac.at/en/research-radar/10.55776/F65}{10.55776/F65},
    \href{https://www.fwf.ac.at/en/research-radar/10.55776/I6802}{10.55776/I6802}, 
    \href{https://www.fwf.ac.at/en/research-radar/10.55776/PAT3699424}{10.55776/PAT3699424},
and
    \href{https://www.fwf.ac.at/en/research-radar/10.55776/PAT3446525}{10.55776/PAT3446525}].
    For open access purposes, the author has applied a CC BY public copyright license
    to any author accepted manuscript version arising from this submission.
}

\definecolor{psychedelicpurple}{rgb}{0.87, 0.0, 1.0}

\begin{document}


\maketitle
\thispagestyle{fancy}

\begin{abstract}
    This paper concerns the inclusion of Newton's method into an adaptive finite element method (FEM)
    for the solution of nonlinear partial differential equations (PDEs).
    It features an adaptive choice of the damping parameter in the Newton iteration for the discretized nonlinear problems on each level
    ensuring both global linear and local quadratic convergence.
    In contrast to energy-based arguments in the literature, a novel approach in the analysis considers the discrete dual norm of the residual as a computable measure for the linearization error.
    As a consequence, this paper provides the first convergence analysis with optimal rates
    of an adaptive iteratively linearized FEM beyond energy-minimization problems.
    The presented theory applies to strongly monotone operators with locally Lipschitz continuous Fréchet derivative.
    We present a class of semilinear PDEs fitting into this framework and provide numerical experiments to underline the theoretical results.
\end{abstract}


\section{Introduction}

\subsection{State of the art}

Newton's method is the prime selection for the iterative solution of nonlinear problems due to its striking local quadratic convergence.
It aims at finding the root of the residual of the equation.
In this paper, we discretize the nonlinear problem first (by means of an adaptive finite element method) and
then apply Newton's method to the resulting finite-dimensional discrete problem.
The refinement of the underlying mesh leads to an improved discrete approximation of the unknown exact solution.
In optimization, linearization is performed on the continuous level and these continuous problems are subsequently discretized.
The coupling of Newton's method with uniform mesh refinement is apparently clear due to \emph{mesh-independence},
backed by a quote from~\cite{dp1992}:
\begin{flushright}
    \itshape
    \enquote{Mesh independence of Newton's method means that Newton's
    method applied to a family of finite-dimensional discretizations of a Banach space
    nonlinear operator equation behaves essentially the same for all sufficiently fine discretizations.}
\end{flushright}
The asymptotic mesh independence is a striking theoretical feature of Newton's method of nonlinear operator equations;
see, e.g., \cite{abpr1986, dp1992, al1996, wsd2005}.
Available results are, however, restricted to the case of (sufficiently fine) quasi-uniform triangulations.

Adaptive finite element methods (AFEMs) are powerful schemes to discretize partial differential equations.
Without a priori knowledge of the exact solution, they automatically refine the underlying mesh.
This allows to impressively reduce the computational cost by successively generating meshes
that are only resolved where the error contribution is indicated to be the most significant.
The convergence of AFEMs for nonlinear PDEs with an \emph{exact} solution of the discrete nonlinear equation is well-investigated in the literature.
Some pioneering contributions on plain convergence include
\cite{v2002,dk2008} for the \(p\)-Laplace equation and
\cite{c2008,bc2008} for degenerate convex minimization problems.
Optimal convergence rates for adaptive nonlinear FEMs, though under the assumption that the arising nonlinear discrete system of equations are solved exactly, have been established
in \cite{gmz2012} for quasilinear PDEs with strongly monotone nonlinearity,
in \cite{bdk2012} for the \(p\)-Laplace equation,
and in \cite{bbimp2022} for goal-oriented AFEM solving strongly monotone and locally Lipschitz continuous semilinear PDEs.

Practical realizations of AFEMs include an iterative inexact solution of the discrete nonlinear equation
such as fixed-point iterations~\cite{gmz2011,cw2017,hw2020:ailfem} or Newton's method~\cite{ev2013,aw2015,amw2017,h2023}.
A crucial aspect in the design of such combined adaptive schemes is the automatic balance
between the discretization and the linearization error \cite{aev2011}.
Several works are devoted to the joint convergence analysis of adaptive iteratively linearized finite element methods (AILFEMs)
\cite{ghps2018,hw2020:convergence,hpw2021}
and to the combination with the iterative algebraic solution of the linearized discrete equations
for optimal complexity, i.e., optimal convergence rates with respect to the overall computational cost \cite{ghps2021,hpsv2021,bps2024}.
Recent contributions focus on robust convergence results independent of the choice of adaptivity or stopping parameters for the linearization
\cite{bfmps2025,mps2024}.

The convergence analysis in all mentioned publications is based on
the contraction of an energy difference and is, therefore, restricted to the case of symmetric nonlinear PDEs.
Moreover, it requires a sufficiently small damping parameter in the Newton scheme to ensure a contraction of the energy difference that prevents the local quadratic convergence of Newton's method.
Based on intrinsic energy estimates, the paper~\cite{h2023} proposes an adaptive damping strategy and empirically observes convergence to an undamped Newton's method, i.e., $\delta \nearrow 1$. 

This paper aims to overcome the limitations of energy-based arguments by introducing a fully residual-based adaptive algorithm. Its convergence analysis even proves that the proposed adaptive damping strategy eventually ensures an undamped Newton's method for locally Lipschitz continuous problems.

\subsection{Nonlinear model problem}\label{section:nonlinearModelProblem}

Let $\XX$ be a Hilbert space equipped with the inner product $\edual{\,\cdot\,}{\,\cdot\,}$ and
its induced norm $\enorm{\,\cdot\,}$. The topological dual space of $\XX$ is denoted by $\XX'$ and equipped with the dual norm $\norm{\, \cdot \,}_{\XX'}$.
Given a nonlinear operator $\AA \colon \XX \to \XX'$ and a right-hand side $F \in \XX'$,
this paper investigates the nonlinear operator equation seeking \(u^\exact \in \XX\) such that
\begin{equation}
    \label{eq:weakform}
    \dual{\AA u^\exact}{v} = \dual{F}{v} \quad \text{ for all } v \in \XX,
\end{equation}
where \(\dual{\,\cdot\,}{\,\cdot\,}\) denotes the duality pairing on \(\XX' \times \XX\).
The existence and uniqueness of the solution $u^\exact \in \XX$
follows from the Browder--Minty theorem on monotone operators~\cite[Theorem~26.A]{zeidler} under the two following assumptions:
\begin{description}
    \item[\bfseries (N1) strong monotonicity of $\AA$]
        \labeltext{N1}{assump:stronglyMonotone}
        There exists a constant $\alpha > 0$ such that
        \[
            \alpha \, \enorm{v - w}^2
            \leq
            \dual{\AA v - \AA w}{v - w}%
            \quad \text{for all } v, w \in \XX.
        \]
    \item[\bfseries (N2) local Lipschitz continuity of $\AA$]
        \labeltext{N2}{assump:locallyLipschitz}
        For all $\vartheta > 0$, there exists $\Lup[\vartheta] > 0$ such that
        \[
            \dual{\AA v - \AA w}{\varphi}
            \leq
            \Lup[\vartheta] \enorm{v - w} \enorm{\varphi}
            \text{ for all } v, w, \varphi \in \XX
            \text{ with } \max\big\{ \enorm{v}, \enorm{v - w}\big\} \leq \vartheta.
        \]
\end{description}
The strong monotonicity~\eqref{assump:stronglyMonotone} ensures Lipschitz continuity of the solution operator, i.e.,
\begin{equation}
    \label{eq:inverse_Lipschitz}
    \alpha \, \enorm{v - w}
    \leq
    \norm{\AA v - \AA w}_{\XX'}
    \quad \text{ for all } v, w \in \XX.
\end{equation}
Moreover, the well-posedness transfers to any closed subspace $\XX_\coarse \subseteq \XX$
with $u_\coarse^\exact \in \XX_\coarse$ solving
\begin{equation}\label{eq:weakformDiscrete}
    \dual{\AA u_\coarse^\exact}{v_\coarse} = \dual{F}{v_\coarse} \quad \text{ for all } v_\coarse \in \XX_\coarse.
\end{equation}

Without loss of generality, suppose that $\AA0 \neq F \in \XX'$.
Due to~\cite[Proposition~2]{bbimp2022cost},
the (discrete) exact solutions $u^\exact \in \XX$ and $u_\coarse^\exact \in \XX_\coarse \subseteq \XX$
satisfy a uniform bound in the energy norm
\begin{equation}\label{eq:exact:bounded}
    \max\big\{ \enorm{u^\exact}, \enorm{u^\exact_\coarse} \big\}
    \le
    \Bup_0
    \coloneqq
    \frac{1}{\alpha} \, \norm{F - \AA 0}_{\XX'} \neq 0
\end{equation}
as well as a C\'ea-type estimate
\begin{equation}
    \label{eq:cea}
    \enorm{u^\exact - u_\coarse^\exact}
    \le
    \Ccea \min_{v_\coarse \in \XX_\coarse}  \enorm{u^\exact -v_\coarse}
    \quad \text{with } \quad \Ccea = \frac{\Lup[2 \Bup_0]}{\alpha}.
\end{equation}

\subsection{Main results}

In section~\ref{section:linearization},
Newton's method is employed for finding the root of the residual $v_\coarse \mapsto F- \AA v_\coarse$
for some approximation $v_\coarse \approx u^\exact$ in a fixed discrete space $\XX_\coarse$.
We measure the linearization error in terms of a discrete dual norm \(\Vert F - \AA v_\coarse \Vert_{\XX_\coarse'}\).
This term is computable at the expense of solving a discrete linear problem and is used to steer the adaptive damping strategy of Newton's method.
The discrete residual term is the quantity of interest in the established analysis for Newton's method and, thus,
the convergence proof may follow the global convergence analysis of Newton's method \cite[Theorem~2.12]{d2004}.
Essential modifications have to be made to account for the fact that $\AA$ is only local Lipschitz continuous in the situation at hand.

The second part of the paper in section~\ref{section:adaptive_mesh_refinement} and later proposes an adaptively damped Newton's method that is coupled with an adaptive mesh-refining algorithm,
where adaptivity is also understood in the sense that the algorithm decides the stopping of the Newton linearization based on the discretization error.
We prove full R-linear convergence irrespective of the involved adaptivity parameters, i.e.,
contraction of a quasi-error quantity regardless whether mesh refinement or a linearization step is performed.
The only limitation is on the choice of the minimal number of Newton steps on each mesh level. This limitation, however, is expected to be of less significance in practice, in particular when the algorithm is in the quadratic convergence regime.
Our analytical findings culminate in the proof of optimal convergence rates
with respect to an idealized cost measure without exploiting any energy structure.

\subsection{Outline}

This paper is organized as follows.
Section~\ref{section:linearization} introduces the Newton linearization on a fixed discrete subspace
and derives a key a posteriori estimate in Lemma~\ref{lemma:aposteriori}.
It presents a global convergence analysis for the Newton method with adaptive damping strategy in Algorithm~\ref{algorithm:ADN}.
Section~\ref{section:adaptive_mesh_refinement} introduces the coupled Algorithm~\ref{algorithm:AILFEM}
with adaptive mesh refinement and the adaptively damped Newton method for the solution of the discrete nonlinear problem on each level.
Moreover, it discusses the increase of the discrete dual norm under mesh refinement and how it is controlled in the algorithm.
Section~\ref{section:mainResults} presents the main results of full R-linear convergence in Theorem~\ref{theorem:fullRLinearConvergence}
and of optimal convergence rates understood with respect to computational complexity in Theorem~\ref{theorem:optimalRates}.
Finally, section~\ref{section:applications} is devoted to the application of the presented theory
to a semilinear, locally Lipschitz continuous model problem and concludes the paper with numerical experiments.

Throughout, the notation \(A \lesssim B\) is used to shorten the estimate \(A \leq C\,B\) with a generic constant \(C > 0\)
which may depend on the problem data, the shape regularity of underlying triangulations, or the polynomial degree of the discretization,
but not on the mesh size and thus not on the dimension \(\dim(\XX_H)\) of the discrete space.
The equivalence $A \eqsim B$ abbreviates \(A \lesssim B\) and \(B \lesssim A\).



\section{Newton linearization}
\label{section:linearization}

This section is devoted to the iterative solution of the nonlinear discrete problem~\eqref{eq:weakformDiscrete}
on a fixed discrete space $\XX_\coarse \subseteq \XX$.
It introduces the Newton method with an adaptive damping strategy and
extends its convergence analysis to locally Lipschitz operator equations.

\subsection{Iterative procedure}
\label{subsection:Newton_definition}

The Newton iteration requires the operator $\AA\colon \XX \to \XX'$ to be Fr\'echet differen\-tiable at any $v \in \XX$,
i.e., there exists a bounded linear operator $\d{\AA}[v] \in \mathcal{L}(\XX, \XX')$ such that

\(\AA(v + w) = \AA(v) + \d{\AA}[v] w + o(\enorm{w})\) for all \(v, w \in \XX\) or, more precisely,
\begin{equation}
    \label{eq:frechet}
    \lim_{w \in \XX, \, \enorm{w} \to 0}
    \frac{\norm{\AA(v+w) - \AA(v) -\d\AA[v]w}_{\XX'}}{\enorm{w}} 
    = 0
    \quad \text{ for all } v \in \XX.
\end{equation}

Recall the operator norm on the space $\mathcal{L}(\XX, \XX')$ of  bounded linear operators
\begin{equation*}
    \norm{\d{\AA}[v]}_{\mathcal{L}(\XX, \XX')}
    \coloneqq
    \sup_{w \in \XX, \,\enorm{w} = 1}
    \norm{\d{\AA}[v] w}_{\XX'}
    \coloneqq
    \sup_{w \in \XX, \,\enorm{w} = 1}
    \sup_{\varphi \in \XX, \,\enorm{\varphi} = 1}
    \dual{\d{\AA}[v] w}{\varphi}.
\end{equation*}

Throughout this paper, we apply the Newton iteration to solve the nonlinear discrete problem~\eqref{eq:weakformDiscrete}
on any closed subspaces $\XX_\coarse \subseteq \XX$.
Given some approximation $v_\coarse \in \XX_\coarse$ to the unknown exact solution \(u_\coarse^\exact\),
the Newton update $\varrho_\coarse \in \XX_\coarse$ is determined by the \emph{linear} residual problem
\begin{equation}
    \label{eq:Newton_update}
    \dual{\d{\AA}[v_\coarse]\varrho_\coarse}{w_\coarse}
    =
    \dual{F - \AA v_\coarse}{w_\coarse}
    \quad\text{ for all } w_\coarse \in \XX_\coarse.
\end{equation}
For a damping parameter $0 < \delta \leq 1$,
the Newton iteration is formally described by the process function $\Phi(\delta; \, \cdot \,)\colon \XX_\coarse \to \XX_\coarse$ with
\begin{equation}
    \label{eq:update:structure}
    \Phi(\delta; v_\coarse)
    \coloneqq
    v_\coarse + \delta\,\varrho_\coarse.
\end{equation}

The existence and uniqueness of $\varrho_\coarse \in \XX$ solving~\eqref{eq:Newton_update}
and, thus, the well-definedness of $\Phi(\delta; \, \cdot)$
is guaranteed by the Lax--Milgram lemma~\cite{lm1954} under the following assumptions:
\begin{description}
    \item[\bfseries (N3) strong monotonicity of $\d{\AA}$]\labeltext{N3}{assump:derivativeMonotone}
        There exists a constant $\beta > 0$ such that
        \[
            \beta \, \enorm{w}^2
            \leq
            \dual{\d{\AA}[v] w}{w}
            \quad\text{ for all } v, w \in \XX,
        \]
    \item[\bfseries (N4) local Lipschitz continuity of $\d{\AA}$]\labeltext{N4}{assump:derivativeLocallyLipschitz}
        For all $\vartheta > 0$, there exists $\Mup[\vartheta] > 0$ such that
        \begin{align*}
            \norm{\d{\AA}[v]- \d{\AA}[w]}_{\LL(\XX, \XX')}
            &\leq
            \Mup[\vartheta] \, \enorm{v - w}
            \\
            &\phantom{{}={}}\quad
            \text{ for all } v, w \in \XX
            \text{ with }
            \max\{\enorm{v}, \enorm{v - w}\} \leq \vartheta.
        \end{align*}
        Note that $\Mup[\,\cdot\,]$ is uniformly bounded from below and monotone, i.e.,
        $0\le \vartheta_1 \le \vartheta_2$ implies $0 <\Mup[0] \le \Mup[\vartheta_1] \le \Mup[\vartheta_2]$.
\end{description}
These assumptions ensure that the derivative $\d{\AA}[v_\coarse]$ is an invertible linear operator
from $\XX_\coarse$ to its topological dual space $\XX_\coarse'$ for every \(v_\coarse \in \XX_\coarse\)
and that the inverse is bounded in the sense of
\begin{equation}\label{eq:Newton:boundedInverse}
    \beta \, \enorm{w_\coarse}
    \leq
    \norm{\d{\AA}[v_\coarse] w_\coarse}_{\XX_\coarse'} \coloneqq \sup_{\varphi_\coarse \in \XX_\coarse, \enorm{\varphi_\coarse}=1} \dual{\d{\AA}[v_\coarse] w_\coarse}{\varphi_\coarse}
    \quad\text{ for all } v_\coarse, w_\coarse \in \XX_\coarse.
\end{equation}

\begin{remark}\label{remark:assumptions}
    \emph{(i)}
    For $v, w \in \XX$ and $t \in \R$, it holds
    \begin{equation*}
        \alpha \, t^2 \, \enorm{w}^2
        \eqreff{assump:stronglyMonotone}{\le}
        \dual{\AA(v+ tw) - \AA(v) - \d\AA[v](tw)}{tw} + t^2 \, \dual{\d\AA[v]w}{w}.
    \end{equation*}
    The division by \(t^2\), the limit \(t \to 0\), and the definition~\eqref{eq:frechet} of the Fréchet derivative imply the coercivity \eqref{assump:derivativeMonotone} for some $\beta$ with \(\alpha \leq \beta\).

    \emph{(ii)}
    For $u, v, w \in \XX$, a Taylor expansion with integral remainder yields
\begin{equation*}
    \begin{aligned}
        \dual{\AA u - \AA v}{w}
        &=
        \int_0^1 \dual{\d\AA[v+\lambda(u-v)](u - v)}{w} \, \d \lambda \\
&         \le
        \int_0^1 \norm{\d\AA[v+\lambda(u-v)]}_{\mathcal{L}(\XX, \XX')}\, \d \lambda\, \enorm{u - v} \, \enorm{w}.
    \end{aligned}
\end{equation*}
    Since~\eqref{assump:derivativeLocallyLipschitz} implies local boundedness of \(\d\AA[\cdot]\), i.e., $\norm{\d\AA[w]}_{\mathcal{L}(\XX, \XX')} \leq \Mup[\vartheta]$ for all $w \in \XX$ with $\enorm{w} \leq \vartheta$,
    this shows local Lipschitz continuity~\eqref{assump:locallyLipschitz} of \(\AA\).
\end{remark}

\subsection{A posteriori linearization error estimate}
\label{section:aposteriori}

The local Lipschitz continuity~\eqref{assump:locallyLipschitz} of \(\AA\)
and the Lipschitz continuity~\eqref{eq:inverse_Lipschitz} of its inverse $\AA^{-1}$
ensure that the discrete residual \(\norm{F - \AA v_\coarse}_{\XX_\coarse'}\)
provides an a~posteriori error estimator for the linearization error
stemming from the inexact solution of the discrete nonlinear problem~\eqref{eq:weakformDiscrete}.
\begin{lemma}[a~posteriori linearization error estimate]
    \label{lemma:linearization_estimator}
    For any function $v_\coarse \in \XX_\coarse$ and any \(\vartheta > 0\)
    with \(\max\{\enorm{u_\coarse^\exact}, \enorm{u_\coarse^\exact - v_\coarse} \} \leq \vartheta\), there holds the equivalence
    \begin{equation}
        \label{eq:linearization_estimator}
        \alpha\, \enorm{u_\coarse^\exact - v_\coarse}
        \leq
        \norm{F - \AA v_\coarse}_{\XX_\coarse'}
        \coloneqq
        \sup_{\varphi_\coarse \in \XX_\coarse, \,\enorm{\varphi_\coarse} = 1}
        \dual{F - \AA v_\coarse}{\varphi_\coarse}
        \leq
        \Lup[\vartheta] \,
        \enorm{u_\coarse^\exact - v_\coarse}.
        \hfill \qed
    \end{equation}
\end{lemma}

The discrete residual is not only crucial as a measure for the linearization error
in the convergence analysis of the Newton iteration.
The fact that the dual norm on a finite-dimensional subspace $\XX_\coarse$ is computable
at the expense of solving a linear equation motivates its use as a practical error estimator in iterative algorithms.
Recall that the Riesz representative \(r_\coarse \in \XX_\coarse\) determined by the linear problem
\begin{equation}
    \label{eq:residual_Riesz}
    \edual{r_\coarse}{w_\coarse}
    =
    \dual{F - \AA v_\coarse}{w_\coarse}
    \quad \text{ for all } w_\coarse \in \XX_\coarse
\end{equation}
satisfies the equality $\enorm{r_\coarse}=\norm{F - \AA v_\coarse}_{\XX_\coarse'}$.

\subsection{Global convergence analysis}
\label{subsection:Newton_convergence}

The convergence analysis of the Newton scheme employs the notion of level sets
basing on the initial iterate \(u_\coarse^0 \in \XX_\coarse\)
\begin{equation}
    \label{eq:levelset}
    \LL_\coarse[u_\coarse^0]
    \coloneqq
    \set[\big]{v_\coarse \in \XX_\coarse :
    \norm{F - \AA v_\coarse}_{\XX_\coarse'} < \norm{F - \AA u_\coarse^0}_{\XX_\coarse'}}.
\end{equation}
The linearization error equivalence~\eqref{eq:linearization_estimator} leads to uniform boundedness of \(\LL_\coarse[u_\coarse^0]\) as a next result.
\begin{lemma}[boundedness of level set]
    \label{lemma:level_set_boundedness}
    Suppose that the initial iterate \(u_\coarse^0 \in \XX_\coarse\) is bounded
    by \(\enorm{u_\coarse^0} \le c_0 \Bup_0 = c_0\norm{F - \AA 0}_{\XX'}\) for some constant \(c_0 > 1\) and $\Bup_0$ from~\eqref{eq:exact:bounded}. Then, the assumptions~\eqref{assump:stronglyMonotone}--\eqref{assump:locallyLipschitz} ensure that
    \begin{equation}
        \label{eq:Newton:uniformBoundResidual}
            \norm{F - \AA v_\coarse}_{\XX_\coarse'}
            \leq \,
            \Lup[(1 + c_0) \Bup_0] \, \enorm{u_\coarse^\exact - u_\coarse^0}
            \\
            \leq
            \Lup[(1 + c_0) \Bup_0] \, (1 + c_0) \Bup_0
                    \text{ for all } v_\coarse \in \LL_\coarse[u_\coarse^0]
    \end{equation}
    and, thus,
    \begin{equation}
        \label{eq:level_set_boundedness}
        \!\enorm{v_\coarse}
        \leq \!
            \frac{
        \norm{F - \AA v_\coarse}_{\XX_\coarse'}}{\alpha}
        +
        \enorm{u_\coarse^\exact}
        \leq
        \frac{\Lup[(1 + c_0) \Bup_0] \, (1 + c_0) \Bup_0}{\alpha} + \Bup_0
        \eqqcolon
        \overline\Bup_0
        \text{ for all } v_\coarse \in \LL_\coarse[u_\coarse^0].
    \end{equation}
\end{lemma}

\begin{proof}
    For \(v_\coarse \in \LL[u_\coarse^0]\),
    the bounds \(\enorm{u_\coarse^0} \le c_0 \Bup_0\) and~\eqref{eq:exact:bounded}
    show that $\max\{ \enorm{u_\coarse^\exact-u_\coarse^0}, \enorm{u_\coarse^\exact}\} \le (1 + c_0) \Bup_0$ and
    the upper bound in~\eqref{eq:linearization_estimator} verifies estimate~\eqref{eq:Newton:uniformBoundResidual} by
    \begin{equation*}
            \norm{F - \AA v_\coarse}_{\XX_\coarse'}
            \eqreff*{eq:levelset}<
            \, \norm{F - \AA u_\coarse^0}_{\XX_\coarse'}
            \, \eqreff*{eq:linearization_estimator}\leq \,
            \Lup[(1 + c_0) \Bup_0] \, \enorm{u_\coarse^\exact - u_\coarse^0}
            \leq
            \Lup[(1 + c_0) \Bup_0] \, (1 + c_0) \Bup_0.
    \end{equation*}
    The residual bounds the norm of any $v_\coarse \in \XX_\coarse$
    using the estimate~\eqref{eq:linearization_estimator} in
    \begin{equation*}
        \label{eq:Newton:boundIterates}
        \enorm{v_\coarse}
        \leq
        \enorm{u_\coarse^\exact - v_\coarse}
        +
        \enorm{u_\coarse^\exact}
        \eqreff{eq:linearization_estimator}\leq
        \alpha^{-1} \,
        \norm{F - \AA v_\coarse}_{\XX_\coarse'}
        +
        \enorm{u_\coarse^\exact}.
    \end{equation*}
    The combination of last two formulas
    concludes the proof of~\eqref{eq:level_set_boundedness}.
\end{proof}

The global convergence of the Newton iteration bases on
the contraction of the residual for sufficiently small damping parameters \(\delta\)
as specified in the following theorem.
\begin{theorem}[global convergence]
    \label{theorem:newton}
    Suppose  the assumptions~\eqref{assump:stronglyMonotone}--\eqref{assump:derivativeLocallyLipschitz}
    and that the initial iterate \(u_\coarse^0 \in \XX_\coarse\) is bounded
    by \(\enorm{u_\coarse^0} \le c_0 \Bup_0 = c_0\norm{F - \AA 0}_{\XX'}\) for some constant \(c_0 > 1\).
    For the constants \(\overline\Bup_0 >0\) from~\eqref{eq:level_set_boundedness}
    and \(\Mup[\,\cdot\,]>0\) from~\eqref{assump:derivativeLocallyLipschitz},
    abbreviate
    \begin{equation}
        \label{eq:uniform_bound:C_0}
        \textup{C}_0
        \coloneqq
        \frac{\Mup[\overline\Bup_0] \,\Lup[(1 + c_0) \Bup_0]\, (1 + c_0) \Bup_0}{2\beta^2} \le \frac{\alpha \, \Mup[\overline\Bup_0] \,\overline\Bup_0}{2\beta^2}.
    \end{equation}
    For all $0< \delta \le 1$
    and $v_\coarse\in \LL_\coarse[u_\coarse^0]$,
    the following estimate holds
    \begin{equation}
        \label{eq:levelsetContraction}
        \norm{F - \AA(\Phi(\delta; v_\coarse))}_{\XX_\coarse'}
        <
        \big[ 1 - \delta +\textup{C}_0 \, \delta^2\big]
        \,\norm{F - \AA v_\coarse}_{\XX_\coarse'}.
    \end{equation}
    The choice of the optimal damping parameter
    \(\delta = \delta_{\textup{opt}} \coloneqq \min\{ (2 \textup{C}_0)^{-1}, 1\}\)
    for the quadratic function \(r \colon (0, 1] \to \R\) with \(r[\delta] \coloneqq 1 - \delta + \textup{C}_0 \, \delta^2\)
    ensures the optimal reduction constant $0 < r_\textup{opt}[\delta_\textup{opt}]< 1$ such that there holds
    \begin{equation}
        \norm{F - \AA(\Phi(\delta; v_\coarse))}_{\XX_\coarse'}
        <
        r_\textup{opt}[\delta_{\textup{opt}}]
         \,
        \norm{F - \AA v_\coarse}_{\XX_\coarse'}
         \text{ with } r_\textup{opt}[\delta_\textup{opt}]
        =
        \begin{cases}
            1- \frac{\delta_{\textup{opt}}}{2}  \mkern-12mu &\text{if } \delta_{\textup{opt}} < 1,
            \\ \textup{C}_0 &\text{else}.
        \end{cases}
    \end{equation}
    Moreover, any \(0 < \delta \leq \delta_{\textup{opt}}\) leads to a reduction factor $r[\delta]$ satisfying
    \(1 - \delta_{\textup{opt}} / 2 \leq r[\delta] < 1\) such that
    the Newton iterate $\Phi(\delta; v_\coarse) \in \LL[u_\coarse^0]$ remains within the same level set.
\end{theorem}

\begin{proof}
    The proof follows the lines of~\cite[Theorem~2.12]{d2004} and consists of three steps.

    \emph{Step~1.}
    The boundedness~\eqref{eq:Newton:boundedInverse} of the inverse of \(\d\AA\) and the Newton update~\eqref{eq:Newton_update} imply
    \begin{equation}
        \label{eq:Newton:residualDominatesDifference}
        \beta\,\enorm{\varrho_\coarse}
        \, \eqreff*{eq:Newton:boundedInverse}\leq \,
         \,
        \norm{\d{\AA}[v_\coarse] \varrho_\coarse}_{\XX_\coarse'}
        \eqreff*
        {eq:Newton_update}=
        \norm{F - \AA v_\coarse}_{\XX_\coarse'}.
    \end{equation}
    The combination of this with estimate~\eqref{eq:Newton:uniformBoundResidual}
    and the bound \(\alpha \leq \beta\) from Remark~\ref{remark:assumptions}{\rm(i)} yields
    \[
        \alpha \, \enorm{\varrho_\coarse}
        \eqreff{eq:Newton:uniformBoundResidual}\le
         \Lup[(1 + c_0) \Bup_0] \, (1 + c_0) \Bup_0
         \eqreff{eq:level_set_boundedness}\le
        \alpha \, \overline\Bup_0.
    \]
    For any \(0 < \tau \leq \delta \leq 1\),
    this and Lemma~\ref{lemma:level_set_boundedness} ensure that \(\enorm{v_\coarse}\) and
    \(\tau \, \enorm{\varrho_\coarse} \leq \enorm{\varrho_\coarse}\)
    are bounded by \(\overline\Bup_0\).
    Hence,~\eqref{assump:derivativeLocallyLipschitz}
    with $v = \Phi(\tau; v_\coarse) =v_\coarse + \tau \varrho_\coarse$ and $w = v_\coarse$ gives
    \begin{equation}
        \label{eq:Newton:residualDominatesDifferenceTau}
        \begin{aligned}
            \norm{(\d{\AA}[v_\coarse + \tau \varrho_\coarse] - \d{\AA}[v_\coarse]) \varrho_\coarse}_{\XX_\coarse'}
            &\le
            \norm{\d{\AA}[\Phi(\tau; v_\coarse)] - \d{\AA}[v_\coarse]}_{\mathcal{L}(\XX, \XX')} \, \enorm{\varrho_\coarse}
            \\
            &\eqreff*{assump:derivativeLocallyLipschitz}\leq\; \Mup[\overline\Bup_0] \, \enorm{\Phi(\tau; v_\coarse) - v_\coarse} \, \enorm{\varrho_\coarse}
            \eqreff{eq:update:structure}=
            \Mup[\overline\Bup_0]  \, \tau \, \enorm{\varrho_\coarse}^2.
        \end{aligned}
    \end{equation}

    \emph{Step~2.}
    Let $w_\coarse \in \XX_\coarse$. The fundamental theorem of calculus for the operator \(\AA\) yields
    \begin{equation*}
    \begin{aligned}
        \dual{F - \AA(v_\coarse + \delta \varrho_\coarse)}{w_\coarse}
        &=
        \dual{F - \AA v_\coarse  - \int_0^\delta \d{\AA}[v_\coarse + \tau \varrho_\coarse] \varrho_\coarse \d{\tau}}{w_\coarse}
        \\
        & \mkern-100mu=
        (1 - \delta)\, \dual{F - \AA v_\coarse}{w_\coarse} +
\delta\, \dual{F - \AA v_\coarse}{w_\coarse} - \dual{\int_0^\delta \d{\AA}[v_\coarse + \tau \varrho_\coarse] \varrho_\coarse \d{\tau}}{w_\coarse}
            \\
        &\mkern-100mu \eqreff*{eq:Newton_update}=
        (1 - \delta)\, \dual{F - \AA v_\coarse}{w_\coarse}
- \dual{\int_0^\delta (\d{\AA}[v_\coarse + \tau \varrho_\coarse] - \d{\AA}[v_\coarse]) \varrho_\coarse \d{\tau}}{w_\coarse}
    \end{aligned}
\end{equation*}
    Taking the supremum over $w_\coarse \in \XX_\coarse$ with $\enorm{w_\coarse}=1$ and the monotonicity of the integral prove
    \begin{equation}\label{eq:newtonTheorem}
        \begin{aligned}
        \norm{F -& \AA(v_\coarse +\delta \varrho_\coarse)}_{\XX_\coarse'}
        \\
        &\leq
(1-\delta) \, \norm{F - \AA v_\coarse}_{\XX_\coarse'} +         \int_0^\delta \norm{(\d{\AA}[v_\coarse + \tau \varrho_\coarse] - \d{\AA}[v_\coarse] ) \varrho_\coarse}_{\XX_\coarse'} \d{\tau}.
    \end{aligned}
\end{equation}
    The local Lipschitz continuity of the derivative in the form of~\eqref{eq:Newton:residualDominatesDifferenceTau}
    (using~\eqref{assump:derivativeLocallyLipschitz})
    and the bound~\eqref{eq:Newton:residualDominatesDifference} of the Newton update
    (using~\eqref{assump:derivativeMonotone})
    result in the key estimate
    \begin{equation}
        \label{eq:quadraticEstimate}
        \begin{split}
            \norm{F - \AA(v_\coarse + \delta \varrho_\coarse)}_{\XX_\coarse'}
            &\eqreff*{eq:Newton:residualDominatesDifferenceTau}\le \,
            (1-\delta) \,\norm{F - \AA v_\coarse}_{\XX_\coarse'} +
            \Mup[\overline\Bup_0] \,\enorm{\varrho_\coarse}^2 \,
            \int_0^\delta \tau \d{\tau}
            \\
            &\eqreff*{eq:Newton:residualDominatesDifference}\leq
            \bigg[
                1- \delta +
                \frac{\Mup[\overline\Bup_0]}{2\, \beta^2} \,
                \norm{F - \AA v_\coarse}_{\XX_\coarse'}\,\delta^2
            \bigg]
            \, \norm{F - \AA v_\coarse}_{\XX_\coarse'}.
        \end{split}
    \end{equation}

    \emph{Step~3.}
    The estimates~\eqref{eq:Newton:uniformBoundResidual} and~\eqref{eq:quadraticEstimate} show that
    any $0 < \delta \le 1$ and $\textup{C}_0 >0$ from~\eqref{eq:uniform_bound:C_0} satisfy
    \begin{equation*}
        \norm{F - \AA(v_\coarse + \delta \varrho_\coarse)}_{\XX_\coarse'}
        <
        \big[1 - \delta + \textup{C}_0 \,\delta^2\big]
        \,\norm{F - \AA v_\coarse}_{\XX_\coarse'}.
    \end{equation*}
    Elementary calculus for the quadratic function \(r \colon (0, 1] \to \R\)
    with \(r[\delta] = 1 - \delta + \textup{C}_0 \, \delta^2\)
    verifies the remaining assertions and concludes the proof.
\end{proof}

\subsection{Adaptive damping strategy}
\label{section:automatic_damping}

The optimal damping parameter $\delta_{\textup{opt}}$ from Theorem~\ref{theorem:newton} is unknown in general.
However, the following automatic selection of \(\delta\) ensures uniform contraction of
the Newton iteration on \emph{any} fixed subspace \(\XX_\coarse \subseteq \XX\).
\begin{algorithm}[adaptively damped Newton iteration]\label{algorithm:ADN}
    ~\smallskip

    \noindent
    {\textup{\textbf{Input:}}}
    Initial guess $u_\coarse^0 \in \XX_\coarse$ and $\delta_{\textup{min}} \coloneqq 1/2$.
    \smallskip

    \noindent\textup{\textbf{for}} \; $k = 0, 1,\ldots$ \; \textup{\textbf{repeat}} \;
    the steps {\textup{\ref{algorithm:ADN:linearizationFirst}}--\textup{\ref{algorithm:ADN:termination}}}: \smallskip
    \begin{enumerate}[label=\textup{(\alph*)}]
        \item\label{algorithm:ADN:linearizationFirst}
            \emph{(Newton update)}
            Compute $\varrho_\coarse^k \in \XX_\coarse$ by solving the linear equation
            \begin{equation}
                \label{algorithm:ADN:NewtonUpdate}%
                \dual{\d{\AA}[u_\coarse^{k}]\varrho_\coarse^{k}}{v_\coarse}
                =
                \dual{F-\AA u_\coarse^{k}}{v_\coarse}
                \quad\text{ for all } v_\coarse \in \XX_\coarse.
            \end{equation}
        \item
            \label{algorithm:ADN:adaptiveDampingProcedure}
            \emph{(determine damping parameter)} Initialize $\delta \coloneqq 1$. Check if
            \begin{equation}
                \label{algorithm:ADN:dampingCriterion}
                \norm{F - \AA(u_\coarse^{k} + \delta \varrho_\coarse^{k})}_{\XX_{\coarse}'}
                \leq
                (1-\delta_{\textup{min}}^{3/2} / 2) \,
                \norm{F - \AA u_\coarse^{k}}_{\XX_{\coarse}'}.
            \end{equation}
            \smallskip
            \noindent\textup{\textbf{repeat}}
            $\delta \coloneqq \delta/2$ and update the overall minimum
            $\delta_{\textup{min}} \coloneqq \min\{\delta, \delta_{\textup{min}}\}$ \textup{\textbf{until}~\eqref{algorithm:ADN:dampingCriterion} is met.}
        \item\label{algorithm:ADN:termination}
            \emph{(update iterate)}
            Set $\delta_\coarse^k \coloneqq \delta$
            and $u_\coarse^{k+1} \coloneqq u_\coarse^k + \delta_\coarse^k \, \varrho_\coarse^k$.
            \smallskip
    \end{enumerate}
    {\textup{\textbf{Output:}}} Sequence of iterates $u_\coarse^k$, minimal damping parameter \(\delta_{\textup{min}}\).
\end{algorithm}

The damping condition~\eqref{algorithm:ADN:dampingCriterion} is motivated by the observation that
the reduction factor \(r[\delta] = 1-\delta + \textup{C}_0\, \delta^2\) in Theorem~\ref{theorem:newton}
satisfies \(r[\delta] \in [1 - \delta_{\textup{opt}}/2, 1)\) for all \(\delta \in (0, \delta_{\textup{opt}}]\).
Due to the dominating linear term in \(r[\,\cdot\,]\),
the reduction factor \(\rr\colon (0,1] \to \R\) with \(\rr[\delta] = 1 - \delta^{3/2} / 2\)
in~\eqref{algorithm:ADN:dampingCriterion} converges faster to one than \(r[\,\cdot\,]\).
Thus, the successive reduction of \(\delta_{\textup{min}}\) allows us to ultimately find a damping parameter \(\delta_{\textup{min}} \in (0, \delta_{\textup{opt}}]\) such that the reduction factor \(r[\delta_{\textup{min}}] \in [1 - \delta_{\textup{opt}}/2, 1)\) is guaranteed.
As long as \(\delta_{\textup{min}}\) is uniformly bounded from below,
the contraction in each Newton iteration is uniformly bounded away from $1$.

For the practical performance, it is crucial that the reduction condition~\eqref{algorithm:ADN:dampingCriterion}
is rather mild and will only pose an essential restriction during the non-quadratic convergence phase with linearization error
\[
    0 \ll \enorm{u_\coarse^\exact-u_\coarse^k} \eqsim \norm{F - \AA u_\coarse^k}_{\XX_{\coarse}'}.
\]
If $u_\coarse^k$ is sufficiently close to $u_\coarse^\exact$ allowing for $\delta = 1$ and $\Mup[\overline\Bup_0]\, \beta^{-2} \, 2^{-1} \, \norm{F - \AA v_\coarse}_{\XX_\coarse'}  < 1$,
the contraction factor \(r[\delta_{\textup{min}}] = 1 - \delta_{\textup{min}}^{3/2} / 2\)
will be compensated by the additional residual factor in the estimate~\eqref{eq:quadraticEstimate}
of the proof of Theorem~\ref{theorem:newton}, yielding that
\[
    \norm{F - \AA(u_\coarse^{k} + \delta \varrho_\coarse^{k})}_{\XX_{\coarse}'}
    \lesssim
    \norm{F - \AA u_\coarse^{k}}_{\XX_{\coarse}'}^2.
\]
Since \(\delta\) is reset to one in each iteration,
the algorithm will allow for quadratic convergence with the classical Newton iteration,
but also detects when the quadratic convergence is lost and switches to the non-quadratic damped iteration.

The well-definedness of Algorithm~\ref{algorithm:ADN} is guaranteed by the following lemma.
\begin{lemma}\label{lemma:adaptiveDamping}
    Under the assumptions of Theorem~\ref{theorem:newton},
    the output sequence \((u_\coarse^k)_{k \in \N_0}\) and \(\delta_{\textup{min}} > 0\)
    of Algorithm~\ref{algorithm:ADN} satisfy the following properties:
    \begin{description}
        \item[\bfseries (i) Bounded iterates]
            For all \(k \in \N_0\) and $\overline\Bup_0>0$ from~\eqref{eq:level_set_boundedness}, there holds that \(\enorm{u_\coarse^k} \leq \overline\Bup_0\).
        \item[\bfseries (ii) Guaranteed reduction]
            There exists a uniform lower bound \(0 < \delta_0 \leq \delta_{\textup{min}}\) such that
            \(\rr_0 \coloneqq \rr[\delta_0] = 1 - \delta_0^{3/2}/2 < 1\) satisfies
            \begin{equation}
                \label{eq:damping:guaranteedContraction}
                \norm{F - \AA u_\coarse^{k+1}}_{\XX_{\coarse}'}
                \leq
                \rr_0 \,
                \norm{F - \AA u_\coarse^{k}}_{\XX_{\coarse}'}
                \quad \text{ for all } k \in \N_0.
            \end{equation}
        \item[\bfseries (iii) Quadratic convergence]
            There exists a uniform $k_0 \in \N_0$ such that with \(\delta = 1\)
            \[
                \norm{F - \AA(u_\coarse^k + \varrho_\coarse^k)}_{\XX_\coarse'}
                \leq
                \min\Big\{
                    1 - 2^{-5/2}, \:
                    \frac{\Mup[\overline\Bup_0]}{2\beta^2}\,
                    \norm{F - \AA u_\coarse^k}_{\XX_\coarse'}
                \Big\}\,
                \norm{F - \AA u_\coarse^k}_{\XX_\coarse'}
                \text{ for all } k \ge k_0.
            \]
    \end{description}
    Moreover, the constants $\delta_0$, $\rr_0$, and $k_0$ are independent of the subspace \(\XX_\coarse\).
\end{lemma}

\begin{proof}
    \emph{Step~1 (bounded iterates).}\;
    It is an immediate consequence of the successive reduction~\eqref{algorithm:ADN:dampingCriterion}
    that each iterate \(u_\coarse^k\) belongs to the level set $\LL_\coarse[u_\coarse^0]$ from~\eqref{eq:levelset}.
    Thus, Lemma~\ref{lemma:level_set_boundedness} provides uniform boundedness of the output of Algorithm~\ref{algorithm:ADN}.

    \emph{Step~2 (guaranteed reduction).}\;
    Asymptotically, the reduction factor \(r[\delta] = 1 - \delta + \textup{C}_0 \delta^{2} \to 1\)
    increases linearly for \(\delta \to 0\)
    whereas \(\rr[\delta] = 1 - \delta^{3/2}/2\) from~\eqref{algorithm:ADN:dampingCriterion}
    converges faster to one.
    Precisely, it holds
    \begin{equation}\label{eq:sketchDelta}
        r[\delta]
        \leq
        \rr[\delta]
        \quad \text{for any }
        0
        <
        \delta
        \leq
        \delta^\star
        \coloneqq
        \min\bigg\{ \frac{\big[(16\textup{C}_0+1)^{1/2}-1\big]^2}{16 \, \textup{C}_0^2},\: 1 \bigg\},
    \end{equation}
    where we note that $\textup{C}_0$ is independent of the iteration index \(k\) and the discrete space $\XX_\coarse$.
    Hence, any \(\delta \leq \delta^\star\) with the contraction~\eqref{eq:levelsetContraction}
    fulfills the acceptance condition~\eqref{algorithm:ADN:dampingCriterion}.
    Consequently, there exists a minimal $j_0 \in \N_0$ such that $\delta = \delta_0 \coloneqq 1/2^{j_0} \leq \delta^\star$ will be accepted and the condition~\eqref{algorithm:ADN:dampingCriterion}
    is satisfied after at most $j_0$ iterations of the inner loop in step~\ref{algorithm:ADN:adaptiveDampingProcedure}
    of Algorithm~\ref{algorithm:ADN} with reduction factor \(\rr_0 = 1 - \delta_0^{3/2}/2\). In particular,
    \(\delta_{\textup{min}} \geq 1/2^{j_0} = \delta_0\) is uniformly bounded from below.
   
    \emph{Step~3 (quadratic convergence).}\;
    The successive application of the guaranteed reduction~\eqref{eq:damping:guaranteedContraction} yields convergence
    \(\norm{F - \AA u_\coarse^{k+1}}_{\XX_{\coarse}'} \to 0\) as \(k \to \infty\)
    and allows to choose a uniform index $k_0 \in \N_0$ such that
    \[
        \norm{F - \AA u_\coarse^k}_{\XX_{\coarse}'}
        \leq
        \frac{8\, \beta^2}{5\,\Mup[\overline\Bup_0]}
        \quad \text{ for all } k \ge k_0.
    \]
    For \(\delta = 1\) and \(k \ge k_0\), the key estimate~\eqref{eq:quadraticEstimate} yields
    \[
        \norm{F - \AA(u_\coarse^k + \varrho_\coarse^k)}_{\XX_\coarse'}
        \eqreff{eq:quadraticEstimate}\le
        \frac{\Mup[\overline\Bup_0]}{2\beta^2}\,
        \norm{F - \AA u_\coarse^k}_{\XX_\coarse'}^2
        \leq
        \frac{4}{5} \,
        \norm{F - \AA u_\coarse^k}_{\XX_\coarse'}.
    \]
    Since \(\delta_{\textup{min}} \leq 1/2\) implies \(4/5 <  1 - 2^{-5/2} \leq 1 - \delta_{\textup{min}}^{3/2}/2\),
    the previous displayed formula confirms that the parameter \(\delta = 1\) is always accepted for \(k \ge k_0\).
    Hence, estimate~\eqref{eq:quadraticEstimate} ensures the quadratic convergence and thus concludes the proof.
\end{proof}


\section{Adaptive mesh-refining algorithm}
\label{section:algorithm}\label{section:adaptive_mesh_refinement}

Throughout the remaining part of the paper,
we consider Hilbert spaces \(\XX\) associated with
a polyhedral bounded Lipschitz domain $\Omega \subset \R^d$ with $d \in \N$.
For a conforming triangulation \(\TT_\coarse\) of \(\Omega\) and a polynomial degree $p \in \N$,
choose the conforming finite element space
\[
    \XX_\coarse
    \coloneqq
    \mathcal{S}^p_0(\TT_\coarse)
    \coloneqq
    \set{v_\coarse \in H_0^1(\Omega)\given \forall T \in \TT_\coarse\colon v_\coarse|_T \text{ is a piecewise polynomial of degree} \le p}    \subset
    \XX.
\]
This section presents an adaptive mesh-refining algorithm for the solution of the nonlinear problem~\eqref{eq:weakform}
which iteratively solves the discrete problem~\eqref{eq:weakformDiscrete} by Algorithm~\ref{algorithm:ADN}.

\subsection{Mesh refinement}

Let \(\TT_0\) be an initial triangulation of \(\Omega\) into compact simplices.
The local mesh refinement employs a newest-vertex bisection (NVB) algorithm such as~\cite{s2008}
for $d \ge 2$ with admissible $\TT_0$ as well as~\cite{kpp2013} for \(d = 2\)
and~\cite{dgs2023} for $d\ge 2$ with non-admissible $\TT_0$. For $d=1$, we refer to~\cite{affkp2013}.
For each triangulation \(\TT_\coarse\) and marked elements \(\MM_\coarse \subseteq \TT_\coarse\),
let \(\TT_\fine \coloneqq \texttt{refine}(\TT_\coarse, \MM_\coarse)\) be the coarsest conforming refinement of \(\TT_\coarse\)
such that at least all elements \(T \in \MM_\coarse\) have been refined, i.e.,
\(\MM_\coarse  \subseteq \TT_\coarse \setminus \TT_\fine\).
We write \(\TT_\fine \in \T(\TT_\coarse)\) if \(\TT_\fine\) can be obtained from \(\TT_\coarse\) by finitely many steps of NVB,
and abbreviate \(\T \coloneqq \T(\TT_0)\).
Moreover, we observe that the nestedness of meshes \(\TT_\fine \in \T(\TT_\coarse)\)
induces nestedness of the corresponding discrete spaces \(\XX_\coarse \subseteq \XX_\fine\).

\subsection{Axioms of adaptivity and a~posteriori error estimator}\label{subsection:axioms}
For $\TT_\coarse \in \T$, $T \in \TT_\coarse$, and $v_\coarse \in \XX_\coarse$,
let~$\eta_\coarse(T, v_\coarse) \in \R_{\ge 0}$ be the local contributions of an a~posteriori error estimator. We abbreviate
\[
    \eta_\coarse(\UU_\coarse, v_\coarse)
    \coloneqq
    \Big(\! \! \sum_{T \in \,\UU_\coarse} \!\! \eta_\coarse(T, v_\coarse)^2 \Big)^{1/2}
    \text{for all } \UU_\coarse \subseteq \TT_\coarse \quad \text{ and } \quad \eta_\coarse(v_\coarse) \coloneqq \eta_\coarse(\TT_\coarse, v_\coarse).
\]

The axioms of adaptivity from~\cite{axioms} are key properties in the convergence analysis of adaptive mesh-refining algorithms typically driven by residual-based a~posteriori error estimators.
The analysis in Section~\ref{section:mainResults} employs a modification from~\cite[Proposition~15]{bbimp2022} extending the stability~\eqref{axiom:stability} to locally Lipschitz continuous problems.
For the abstract analysis, we suppose that the following properties hold for all $\TT_\coarse \in \T$ and $\TT_\fine \in
\T(\TT_\coarse)$ with corresponding exact discrete solutions $u_\coarse^\exact$ and $u_\fine^\exact$, respectively,
solving~\eqref{eq:weakformDiscrete}.
\begin{description}
    \item[(A1) stability]\labeltext{A1}{axiom:stability}
        For all $\vartheta > 0$ and all $\UU_\coarse \subseteq \TT_\fine \cap \TT_\coarse$,
        there exists $\Cstab[\vartheta] >0$ such that
        \begin{equation*}
        \begin{aligned}
            \big| \eta_\fine(\UU_\coarse, v_\fine) - \eta_\coarse(\UU_\coarse, v_\coarse) \big|
            &\le
            \Cstab[\vartheta]\, \enorm{v_\fine - v_\coarse}  \\
			&  \text{ for all }v_\fine \in \XX_\fine \text{ and all } v_\coarse \in \XX_\coarse \text{ with } \max\big\{\enorm{v_\fine}, \enorm{v_\fine-v_\coarse}\big\} \le \vartheta.
        \end{aligned}
    \end{equation*}
    \item[(A2) reduction]\labeltext{A2}{axiom:reduction}
        There exists $0 < \qred <1$ such that
        \begin{align*}
            \eta_\fine(\TT_\fine \backslash \TT_\coarse, v_\coarse)
            &\le
            \qred \, \eta_\coarse(\TT_\coarse \backslash \TT_\fine, v_\coarse) \quad \text{ for all } v_\coarse \in \XX_\coarse.
        \end{align*}
    \item[(A3) reliability]\labeltext{A3}{axiom:reliability}
        There exists $\Crel >0$ such that
        \begin{align*}
            \enorm{u^\exact - u_\coarse^\exact}
            &\le
            \Crel \, \eta_\coarse(u_\coarse^\exact).
        \end{align*}
    \item[(A4) discrete reliability]\labeltext{A4}{axiom:discreteReliability}
        There exists $\Cdrel >0$ such that
        \begin{align*}
            \enorm{u_\fine^{\exact} - u_\coarse^{\exact}}
            &\le
            \Cdrel \, \eta_\coarse(\TT_\coarse \backslash \TT_\fine, u_\coarse^{\exact})    .
        \end{align*}
\end{description}
As proven in~\cite[Lemma~3.6]{axioms},
an immediate consequence of the axioms~\eqref{axiom:stability}--\eqref{axiom:reliability}, and the C\'ea-type estimate~\eqref{eq:cea} is quasi-monotonicity of the error estimator
with generic constant \(\Cmon \coloneqq 1 + \Cstab[2\Bup_0]\, (1 + \Ccea) \,\Crel\)
\begin{equation}
    \label{eq:quasi-monotonicity}
    \eta_\fine(u_\fine^\exact)
    \leq
    \Cmon \,
    \eta_\coarse(u_\coarse^\exact)
    \quad \text{for all } \TT_\fine \in \T(\TT_\coarse).
\end{equation}
Using discrete reliability~\eqref{axiom:discreteReliability}, the $\Cmon \coloneqq 1 + \Cstab[2\Bup_0]\,\Cdrel$ also provides the monotonicity estimate~\eqref{eq:quasi-monotonicity}. To avoid the use of discrete reliability for full R-linear convergence, we use the first variant when tracking the dependency of the constants.

The remaining axiom of quasi-orthogonality applies to a nested sequence
of finite element spaces \(\XX_\ell \subseteq \XX_{\ell+1}\) for all \(\ell \in \N_0\)
with discrete limit space $\XX_\infty \coloneqq \textup{closure}\big(\bigcup_{\ell=0}^\lmax \XX_\ell\big)$ where the closure is taken with respect to the norm of $\XX$.
Let $u_\infty^\exact \in \XX_\infty$ denote the corresponding Galerkin solution to~\eqref{eq:weakformDiscrete}.
\begin{description}
    \item[(QO) quasi-orthogonality]\labeltext{QO}{axiom:quasiorthogonality}
        For any \(0 < \varepsilon < 1\),
        there exists a constant \(C_{\textup{orth}}(\varepsilon) > 0\) such that
        the discrete solutions \(u_{\ell'}^\exact \in \XX_{\ell'}\) to~\eqref{eq:weakformDiscrete}
        for \(\ell' \in \N_0\) satisfy
        \begin{equation}
            \label{eq:quasiorthogonality}\tag{QO}
            \sum_{\ell' = \ell}^{\ell+n}
            \bigl[ \enorm{u_{\ell'+1}^{\exact} - u_{\ell'}^{\exact}}^2
            - \varepsilon \, \eta_{\ell'}(u_{\ell'}^{\exact})^2 \bigr]
            \le C_{\textup{orth}}(\varepsilon) \, \eta_\ell(u_\ell^\exact)^2
            \quad\text{ for all } \ell, n \in \N_0.
        \end{equation}
\end{description}

The overall error of the Newton iterates $u_\coarse^k \in \XX_\coarse$ of the discrete nonlinear problem decomposes
into a discretization error controlled by the a~posteriori error estimator $\eta_\coarse(\cdot)$
and the linearization error estimate from Lemma~\ref{lemma:linearization_estimator}.
\begin{lemma}[a~posteriori estimate]
    \label{lemma:aposteriori}
    Suppose the assumptions of Theorem~\ref{theorem:newton} and the axioms~\eqref{axiom:stability} and~\eqref{axiom:reliability}.
    Then, the iterates $u_\coarse^k \in \XX_\coarse$ of Algorithm~\ref{algorithm:ADN} with $\overline\Bup_0>0$ from~\eqref{eq:level_set_boundedness} satisfy
    \begin{equation}
        \label{eq:aposteriori}
        \enorm{u^\exact -u_\coarse^k}
        \le
		C_{\textup{rel}} \, \eta_\coarse(u_\coarse^k) + (1+C_{\textup{rel}} \,C_{\textup{stab}}[\overline\Bup_0]) \,\alpha^{-1} \, \norm{F - \AA u_\coarse^k}_{\XX_\coarse'}.
    \end{equation}
\end{lemma}

\begin{proof}
    The uniform boundedness from Lemma~\ref{lemma:adaptiveDamping}\textup{(i)} and
    the bound \(\enorm{u_\coarse^\exact} \leq \Bup_0\) from~\eqref{eq:exact:bounded} allow to
    apply stability~\eqref{axiom:stability} with constant \(\vartheta = \overline\Bup_0 \geq \Bup_0\).
    This and reliability~\eqref{axiom:reliability} prove
    \[
        \enorm{u^\exact - u_\coarse^\exact}
        \eqreff{axiom:reliability}\leq
        \Crel \,
        \eta_\coarse(u_\coarse^\exact)
        \eqreff{axiom:stability}\leq
        \Crel \,
        \big(
            \eta_\coarse(u_\coarse^k) + \Cstab\big[\overline\Bup_0\big]\, \enorm{u_\coarse^\exact- u_\coarse^k}
        \big).
    \]
    A triangle inequality leads us to
    \[
        \enorm{u^\exact -u_\coarse^k}
        \leq
        \enorm{u^\exact - u_\coarse^\exact} + \enorm{u_\coarse^\exact - u_\coarse^k}
        \le \Crel \, \eta_\coarse(u_\coarse^k) + (1+C_{\textup{rel}} \,C_{\textup{stab}}[\overline\Bup_0]) \, \enorm{u_\coarse^\exact - u_\coarse^k}.
    \]
    Applying the a~posteriori estimate~\eqref{eq:linearization_estimator} for the linearization error
    completes the proof.
\end{proof}

\subsection{Newton-based adaptive iteratively linearized FEM (NAILFEM)}
\label{subsection:AINFEM}

The following algorithm combines an established adaptive mesh-refinement procedure with
the adaptively damped Newton's method from Section~\ref{section:linearization}
for the solution of the discrete nonlinear problem~\eqref{eq:weakformDiscrete}
on each mesh level.
\begin{algorithm}[NAILFEM]\label{algorithm:AILFEM}
    {\textup{\textbf{Input:}}} Initial guess $u_{-1}^\kmax \coloneqq u_0^0 \coloneqq 0 \in \XX_0$,
    mesh-adaptivity parameters $0 < \theta  \le 1$ and $\Cmark \ge 1$,
    stopping parameter $\lambdalin > 0$,
    fixed minimal number \(\kmin \geq 1\) of Newton iterations,
    initial value for overall minimal damping $\delta_{\textup{min}} \coloneqq 1/2$.
    \smallskip

    \noindent
    \textup{\textbf{for}} \; $\ell=0,1,\ldots$ \;
    \textup{\textbf{repeat}} \; the steps {\ref{algorithm:meshFirst}}--{\ref{algorithm:meshLast}}:
    \begin{enumerate}[label=\textup{(\roman*)}]
        \item\label{algorithm:meshFirst}
        \textup{{\texttt{SOLVE \& ESTIMATE.}}}

        \noindent\textup{\textbf{for}} \; $k = 0, 1,\ldots$ \; \textup{\textbf{repeat}} \;
    the steps {\textup{\ref{algorithm:NAIL:linearizationFirstFull}}--\textup{\ref{algorithm:NAIL:stoppingCriterion}}}: \smallskip
    \begin{enumerate}[label=\textup{(\alph*)}]
        \item\label{algorithm:NAIL:linearizationFirstFull}
            \emph{(Newton update)}
            Compute $\varrho_\ell^k \in \XX_\ell$ by solving the linear equation
            \begin{equation*}
                \label{algorithm:NAIL:NewtonUpdate}%
                \dual{\d{\AA}[u_\ell^{k}]\varrho_\ell^{k}}{v_\ell}
                =
                \dual{F-\AA u_\ell^{k}}{v_\ell}
                \quad\text{ for all } v_\ell \in \XX_\ell.
            \end{equation*}
        \item
            \label{algorithm:NAIL:adaptiveDampingProcedure}
            \emph{(determine damping parameter)} Initialize $\delta \coloneqq 1$. Check if
            \begin{equation}
                \label{algorithm:NAIL:dampingCriterionFull}
                \norm{F - \AA(u_\ell^{k} + \delta \varrho_\ell^{k})}_{\XX_{\ell}'}
                \leq
                (1-\delta_{\textup{min}}^{3/2} / 2) \,
                \norm{F - \AA u_\ell^{k}}_{\XX_{\ell}'}.
            \end{equation}
            \smallskip
            \noindent\textup{\textbf{repeat}}
            $\delta \coloneqq \delta/2$ and update the minimum
            $\delta_{\textup{min}} \coloneqq \min\{\delta, \delta_{\textup{min}}\}$ \textup{\textbf{until}~\eqref{algorithm:NAIL:dampingCriterionFull} is met.}
        \item\label{algorithm:NAIL:termination}
            \emph{(update iterate and compute estimator)}
            Set $\delta_\ell^k \coloneqq \delta$
            and $u_\ell^{k+1} \coloneqq u_\ell^k + \delta_\ell^k \, \varrho_\ell^k$. Compute the local contributions $\eta_\ell(T,u_\ell^{k})$ for all $T \in \TT_\ell$. \smallskip
            \smallskip
            \item\label{algorithm:NAIL:stoppingCriterion}
            \emph{(stopping criterion)} Check the stopping criterion
            \begin{equation}
                \label{algorithm:AILFEM:stopping}
				k + 1 \geq \kmin
                \quad\text{and}\quad
				\norm{F - \AA u_\ell^{k+1}}_{\XX_{\ell}'}  \le \lambda_{\textup{lin}} \, \eta_\ell(u_\ell^{k+1}).
            \end{equation}
            If~\eqref{algorithm:AILFEM:stopping} evaluates to \textup{\texttt{TRUE}}, terminate the \(k\)-loop and set $\kmax[\ell] \coloneqq k+1$.
    \end{enumerate}
\smallskip    \item {\textup{\texttt{{MARK.}}}}
        Determine a set
        \(
            \MM_\ell \in \mathbb{M}_\ell[\theta, u_\ell^{\kmax}]
            \coloneqq
            \set{\UU_\ell \subseteq \TT_\ell \colon
                \theta \, \eta_\ell(u_\ell^{\kmax})^2 \le \eta_\ell(\UU_\ell, u_\ell^{\kmax})^2
            }
        \)
        such that
        \begin{equation}\label{eq:doerfler}
            \# \MM_\ell \le \Cmark  \min_{\UU_\ell \in \mathbb{M}_\ell[\theta, u_\ell^{\kmax}]} \# \UU_\ell.
        \end{equation}
    \item\label{algorithm:meshLast}
        {\textup{\texttt{{REFINE.}}}}
        Generate the new mesh $\TT_{\ell+1} \coloneqq  \mathtt{refine} (\TT_\ell, \MM_\ell)$ by employing NVB
        and define $u_{\ell+1}^{0} \coloneqq u_\ell^{\kmax}$ (nested iteration).
    \end{enumerate}
    {\textup{\textbf{Output:}}} Sequence of approximations $u_\ell^k$, error estimators $\eta_\ell(u_\ell^{k})$, and residuals $\norm{F-\AA u_\ell^k}_{\XX_\ell'}$.
\end{algorithm}

The output sequence of approximations $u_\ell^k$ of Algorithm~\ref{algorithm:AILFEM}
induces the extended index set
\begin{equation*}
    \overline\QQ
    \coloneqq
    \set{
        (\ell, k) \in \N_0^2 \;:\;
		u_\ell^k \text{ is computed by Algorithm~\ref{algorithm:AILFEM}}
    }.
\end{equation*}
Upon defining $\kmax[\ell] \coloneqq \sup\set{k \in \N_0 \;:\; (\ell, k) \in \overline\QQ}$, we define the index set
\[
    \QQ \coloneqq \overline\QQ \setminus \set{(\ell, \kmax[\ell]) \;:\; \ell \in \N_0 \text{ with } (\ell+1, 0) \in \QQ}.
\]
Note that $\overline{\QQ} \subset \N_0 \times \N_0$ is a countably infinite index set such that,
for all $(\ell,k) \in \N_0 \times \N_0$,
\begin{align*}
    (\ell+1,0) \in \overline{\QQ }
    &\quad \Longrightarrow \quad
    (\ell, \kmax[\ell]) \in \overline{\QQ} \text{ and }\kmax[\ell] = \max \set{ k \in \N_0 \;:\;(\ell,k) \in \overline{\QQ}},
    \\
    (\ell, k+1) \in \overline{\QQ}
    &\quad \Longrightarrow \quad (\ell, k) \in \QQ.
\end{align*}
With $\lmax \coloneqq \sup \set{ \ell \in \N_0 \;:\; (\ell,0) \in \QQ}$,
it then follows that either $\lmax = \infty$ or $\kmax[\lmax] = \infty$ and that $\kmax[\ell] \ge 1$ counts the number of Newton iterations for all mesh levels $\ell$ with $\ell < \lmax$.

The pair $(\ell, \kmax[\ell])$ is not included in $\QQ$ for the reason that
either $(\ell + 1, 0) \in \QQ$ and $u^0_{\ell + 1} = u^{\kmax[\ell]}_\ell$ by nested iteration
or even $\kmax[\ell] = \infty$ if the $k$-loop does not terminate after finitely many steps.
The sequential nature of Algorithm~\ref{algorithm:AILFEM} induces a lexicographical ordering on $\QQ$:
For $(\ell, k)$ and $(\ell^\prime, k^\prime) \in \QQ$,
we write $(\ell^\prime, k^\prime) < (\ell, k)$ if and only if
$(\ell^\prime, k^\prime)$ appears earlier in Algorithm~\ref{algorithm:AILFEM} than $(\ell, k)$.
Given this ordering, the \emph{total step counter} is defined by
\begin{align*}
    \abs{\ell, k}
    \coloneqq
    \#\set{(\ell^\prime, k^\prime) \in \QQ \;:\; (\ell^\prime, k^\prime) < (\ell, k)}
    =
    k + \sum_{\ell' = 0}^{\ell - 1} \kmax[\ell^{\prime}],
\end{align*}
which provides the total number of Newton iterations up to the computation of $u^k_\ell$.
Throughout the paper, we abbreviate $(\ell, \kmax) \coloneqq (\ell, \kmax[\ell])$ and $u^\kmax_{\ell} \coloneqq u^{\kmax[\ell]}_{\ell}$.

\subsection{Newton iteration under mesh refinement}

A crucial assumption in the analysis of Theorem~\ref{theorem:newton} is
the boundedness of the initial iterate $\enorm{u_\ell^0} \le c_0 \Bup_0$ for each mesh level \(\ell \in \N_0\).
On a fixed level, the discrete residual norm allows to control the norms of the iterates
as in Lemma~\ref{lemma:adaptiveDamping}~(i).
However, the discrete dual norm may increase under mesh refinement due to $\XX_{\ell-1} \subseteq \XX_\ell$
even in the case of nested iteration \(u_\ell^0 = u_{\ell-1}^\kmax\), hence
\[
    \norm{F - \AA u_{\ell-1}^\kmax}_{\XX_{\ell-1}'}
    \le
    \norm{F - \AA u_\ell^0}_{\XX_\ell'}
	\quad \text{ for any mesh level } 0 < \ell \text{ with } (\ell, 0) \in \overline\QQ.
\]
The following result relies on the fact that a minimal number of Newton iterations \(\kmin\) is required to compensate for the
factor \(\Lup[(1+c_0)\Bup_0]/\alpha\) of the residual norm that appears when changing mesh levels.
Given \(0 < \overline{q} < 1\) and $\rr_0$ from Lemma~\ref{lemma:adaptiveDamping},
define \(c_0 \coloneqq 2 \, (1 - \overline{q})^{-1}\) and choose the minimal \(\kmin \in \N\) such that
\begin{equation}
    \label{eq:Newton:condition_k0}
	q_0
    \coloneqq
    \frac{\Lup[(1 + c_0) \Bup_0]}{\alpha}\,
    \rr_0^\kmin
    < \overline{q}
    < 1.
\end{equation}
This choice of a minimal number of Newton iterations in Algorithm~\ref{algorithm:AILFEM}
guarantees uniform boundedness of all iterates and quasi-contraction of the residual.
\begin{lemma}
    \label{lemma:boundFinalIterate}
	Suppose that the assumptions~\eqref{assump:stronglyMonotone}--\eqref{assump:derivativeLocallyLipschitz} hold
    and that \(\kmin \in \N\) ensures~\eqref{eq:Newton:condition_k0}.
    Then, the output sequence \(u_\ell^k\) of Algorithm~\ref{algorithm:AILFEM} satisfies
    \begin{description}
        \begingroup\abovedisplayskip=-1.15\baselineskip
        \item[(i) Boundedness of final iterates]
            \begin{equation}
            \mkern240mu \enorm{u_\ell^\kmax} \leq c_0 \Bup_0 \quad \text{ for all } \ell < \lmax.
        \end{equation}
        \item[(ii) Boundedness of iterates]
            \begin{equation}\label{eq:uniformBound}
            \mkern260mu \enorm{u_\ell^k} \leq \overline\Bup_0 \quad \text{ for all } (\ell,k) \in \QQ.
        \end{equation}
    \endgroup
        \item[(iii) Contraction of residual] For all \((\ell, k) \in \QQ\) with \(k \geq \kmin\),
            \begin{equation}
                \label{eq:Newton:residual_contraction}
                \norm{F - \AA u_\ell^k}_{\XX_\ell'}
                \leq
                q_0\,
                \norm{F - \AA u_{\ell-1}^\kmax}_{\XX_{\ell-1}'}
                +
                \alpha q_0 \,
                \enorm{u_\ell^\exact - u_{\ell-1}^\exact}.
            \end{equation}
    \end{description}
\end{lemma}

\begin{proof}
    \emph{Step~1 (contraction of residual on a fixed level).}\;
    Suppose that $\enorm{u_\ell^0} \leq c_0 \Bup_0$ holds.
    Then, the linearization error estimate~\eqref{eq:linearization_estimator} from Lemma~\ref{lemma:linearization_estimator}
allows for an upper bound in terms of the difference of the exact discrete solutions
\begin{equation}
    \label{eq:dualNormIncrease}
    \begin{split}
        \norm{F - \AA u_\ell^0}_{\XX_\ell'} \;
        &\eqreff*{eq:linearization_estimator}\leq \;
        \Lup[(1+c_0)\Bup_0] \,\big(
            \enorm{u_{\ell-1}^\exact - u_{\ell-1}^\kmax}
            + \enorm{u_\ell^\exact - u_{\ell-1}^\exact}
            \big)
        \\
        &\eqreff*{eq:linearization_estimator}\leq \;
        \frac{\Lup[(1+c_0)\Bup_0]}{\alpha} \,
        \norm{F - \AA u_{\ell-1}^\kmax}_{\XX_{\ell-1}'}
        +
        \Lup[(1+c_0)\Bup_0] \, \enorm{u_\ell^\exact - u_{\ell-1}^\exact}.
    \end{split}
\end{equation}
    The reduction~\eqref{eq:damping:guaranteedContraction} with $\rr_0^\kmax \le \rr_0^\kmin$,
    the definition of $q_0$ from~\eqref{eq:Newton:condition_k0},
    and the estimate~\eqref{eq:dualNormIncrease} prove~\eqref{eq:Newton:residual_contraction}
    with $u_{\ell-1}^\kmax = u_\ell^0 \in \XX_{\ell-1}$ via
    \begin{equation}
        \label{eq:Newton:residual_contraction:one_level}
        \norm{F - \AA u_\ell^\kmax}_{\XX_\ell'}
        \eqreff{eq:damping:guaranteedContraction}\leq
        \, \rr_0^{\kmin} \,
        \norm{F - \AA u_\ell^0}_{\XX_\ell'}
        \eqreff{eq:dualNormIncrease}\leq
        q_0\,
        \norm{F - \AA u_{\ell-1}^\kmax}_{\XX_{\ell-1}'}
        +
        \alpha q_0 \,
        \enorm{u_\ell^\exact - u_{\ell-1}^\exact}.
    \end{equation}

    \emph{Step~2 (boundedness of final iterates).}
    First, note that the initial iterate \(u_0^0 = 0\) satisfies the upper bound \(\enorm{u_0^0} = 0 \leq c_0 \Bup_0\).
    We argue by induction on $\ell$ and suppose that \(\enorm{u_\ell^0} \leq c_0 \Bup_0\) for all \(\ell' \le \ell \in \N_0\).
    The estimate~\eqref{eq:level_set_boundedness} can be applied on the level \(\ell\) for \(v_\coarse = u_\ell^\kmax\)
    and, since \(\enorm{u_{\ell'}^0} \leq c_0 \Bup_0\) by induction hypothesis for all $\ell' \le \ell$, also estimate~\eqref{eq:Newton:residual_contraction:one_level} is applicable for all $\ell' \le \ell$. This proves
    \begin{equation*}
    \begin{aligned}
        \enorm{u_{\ell+1}^0} = \enorm{u_\ell^\kmax} \,
        &\eqreff*{eq:level_set_boundedness}\leq \,
        \alpha^{-1} \,
        \norm{F - \AA u_\ell^\kmax}_{\XX_\ell'}
        +
        \enorm{u_\ell^\exact}
        \eqreff*{eq:Newton:residual_contraction:one_level}\leq
        \alpha^{-1}
        q_0 \,
        \norm{F - \AA u_{\ell-1}^\kmax}_{\XX_{\ell-1}'}
        +
        q_0
        \enorm{u_\ell^\exact - u_{\ell-1}^\exact}
        +
        \enorm{u_\ell^\exact}
        \\
        &\leq
        \alpha^{-1}
        q_0 \,
        \norm{F - \AA u_{\ell-1}^\kmax}_{\XX_{\ell-1}'}
        +
        2 \Bup_0 \, q_0
        +
        \Bup_0.
    \end{aligned}
\end{equation*}
    The successive application of the estimate~\eqref{eq:Newton:residual_contraction:one_level},
    the definition of $\Bup_0 = \alpha^{-1} \norm{F - \AA u_0^0}_{\XX_0'}$ from~\eqref{eq:exact:bounded},
    and the choice of \(c_0 = 2(1 - \overline{q})^{-1}\) result in
    \[
        \enorm{u_\ell^\kmax}
        \leq
        \alpha^{-1}
        q_0^{\ell} \,
        \norm{F - \AA u_0^0}_{\XX_0'}
        +
        2 \Bup_0\,
        \Big( \sum_{\ell' = 1}^\ell q_0^{\ell'} \Big)
        +
        \Bup_0
        \le
        2 \Bup_0 \sum_{\ell' = 0}^\ell q_{0}^{\ell'}
        \leq
        \frac{2 \Bup_0}{1 - \overline{q}}
        = c_0 \Bup_0.
    \]
	The induction principle concludes the proof of assertion~(i).

    \emph{Step~3.}
    Due to nested iteration, the initial iterate \(u_\ell^0 = u_{\ell-1}^\kmax\) satisfies
	the assumption \(\enorm{u_\ell^0} \leq c_0 \Bup_0\) from Lemma~\ref{lemma:level_set_boundedness}
    for each level \(\ell \in \N_0\) if $\lmax = \infty$, or $\ell \le \lmax$ else.
    Thus, the upper bound~\eqref{eq:level_set_boundedness} holds and
    verifies assertion~(ii).
	This concludes the proof.
\end{proof}

\begin{remark}
    {\textup{(i)}}
    Note that the minimal number of Newton iterations \(\kmin\) in~\eqref{eq:Newton:condition_k0} is determined a~priori.
    Although a large guaranteed reduction in~\eqref{eq:damping:guaranteedContraction} close to $1$
    might lead to very large values of \(\kmin\) to ensure the condition~\eqref{eq:Newton:condition_k0},
    the quadratic convergence of the Newton iteration will provide the sufficient contraction much earlier.
    It is remarkable that the Newton contraction does not need to compensate the increase of the dual norm
    bounded by \(\enorm{u_\ell^\exact - u_{\ell-1}^\exact}\) as this term is driven to zero by the adaptive mesh refinement.

    {\textup{(ii)}}
    Previous works~\cite{bbimp2022cost,bps2024} on locally Lipschitz problems
    base on the Zarantonello iteration~\cite[Section~25.4]{zeidler} with a strict norm contraction
    and ensure this condition by a suitable adaptation of the stopping criterion.
    This allows to circumvent the need for a minimal number of linearization iterations \(\kmin\)
    while converging significantly slower in practice than the Newton iteration.

    {\textup{(iii)}}
    By virtue of the uniform bound of Lemma~\ref{lemma:boundFinalIterate}\textup{(ii)},
    all involved locally Lipschitz constants and stability constants are uniformly bounded
    by $\Lup[2\overline\Bup_0]$ and $\Cstab[2\overline\Bup_0]$, respectively.
    Throughout the remaining paper, we thus abbreviate
    \begin{equation}
        \label{eq:uniform_constants}
        \overline{\Lup} \coloneqq \Lup[2\overline\Bup_0]
        \quad \text{and} \quad
        \CCstab \coloneqq \Cstab[2\overline\Bup_0].
    \end{equation}
\end{remark}

Standard arguments in the spirit of~\cite[Proposition~4.5]{ghps2018} reveal that an infinite number of Newton iterations implies that
the sought solution $u^\exact = u_\lmax^\exact$ is already discrete.
\begin{lemma}
    \label{lemma:kLoop}
    Suppose the assumptions~\eqref{assump:stronglyMonotone}--\eqref{assump:derivativeLocallyLipschitz} as well as
    the axioms~\eqref{axiom:stability} and~\eqref{axiom:reliability}.
    Then, the case $\lmax < \infty$ with $\kmax[\lmax] = \infty$ in Algorithm~\ref{algorithm:AILFEM}
    implies that $u_\lmax^\exact = u^\exact$ and $\eta_\lmax(u_\lmax^\exact) = 0$.
\end{lemma}

\begin{proof}
	The boundedness of the initial iterate $u_\lmax^0$ from Lemma~\ref{lemma:boundFinalIterate}{\textrm{(i)}} enables the application of Lemma~\ref{lemma:adaptiveDamping}.
	Consequently, all iterates are bounded with \(\enorm{u_\lmax^k} \leq \Bup_0\) for all \(k \in \N_0\).
    By assumption, the linearization stopping criterion~\eqref{algorithm:AILFEM:stopping} is never met
    and the convergence~\eqref{eq:damping:guaranteedContraction} of the Newton iteration on the fixed level \(\lmax\) implies
    \[
        \eta_\lmax(u_\lmax^k)
        \eqreff{algorithm:AILFEM:stopping}<
        \lambdalin^{-1} \,
        \norm{F-\AA u_\lmax^k}_{\XX_\lmax'}
        \xrightarrow{\eqref{eq:damping:guaranteedContraction}} 0
        \quad \text{ as } k \to \infty.
    \]
	Reliability~\eqref{axiom:reliability} and stability~\eqref{axiom:stability} conclude the proof with
    \[
        \enorm{u^\exact - u_\lmax^\exact}
		\eqreff{axiom:reliability}\lesssim
        \eta_\lmax(u_\lmax^\exact)
        \eqreff{axiom:stability}\lesssim
		\eta_\lmax(u_\lmax^k) + \enorm{u_\lmax^\exact - u_\lmax^k}
        \stackrel{\eqref{eq:linearization_estimator}}\lesssim
        \eta_\lmax(u_\lmax^k) + \norm{F-\AA u_\lmax^k}_{\XX_\lmax'}
        \longrightarrow 0 \quad \text{ as } k \to \infty.
        \qedhere
    \]
\end{proof}


\section{Convergence analysis}\label{section:mainResults}

This section is devoted to the two main results of this work: Full R-linear convergence (Theorem~\ref{theorem:fullRLinearConvergence}) and optimal convergence rates (Theorem~\ref{theorem:optimalRates}) of the proposed NAILFEM (Algorithm~\ref{algorithm:AILFEM}).

\subsection{Full R-linear convergence}
\label{subsection:fullRLinearConvergence}

The first main result concerns the quasi-error
\begin{equation}
    \label{eq:quasiError}
    \Eta_\ell^k
    \coloneqq
    \norm{F-\AA u_\ell^k }_{\XX_\ell'}
    +
    \eta_\ell(u_\ell^k)
    \quad \text{ for } (\ell, k) \in \QQ.
\end{equation}
Its two components control the linearization and discretization error and
it holds that $\enorm{u^\exact - u_\ell^k} \lesssim \Eta_\ell^k$
by the a posteriori error estimate~\eqref{eq:aposteriori}.
\begin{theorem}[full R-linear convergence]
    \label{theorem:fullRLinearConvergence}
    Suppose that the assumptions~\eqref{assump:stronglyMonotone}--\eqref{assump:derivativeLocallyLipschitz},
    the axioms~\eqref{axiom:stability}--\eqref{axiom:reliability}, and quasi-orthogonality~\eqref{eq:quasiorthogonality} hold.
    Let the minimal number \(\kmin \in \N\) of Newton iterations be chosen such that~\eqref{eq:Newton:condition_k0} is satisfied.
    Then, for arbitrary $0< \theta \le 1$, $C_{\textup{mark}} \ge 1$, and $\lambda_{\textup{lin}} > 0$,
    there holds full R-linear convergence, i.e.,
    there exist a constant $C_{\textup{lin}}>0$ and a contraction factor $0 < q_{\textup{lin}} < 1$ such that
    \begin{equation}
        \label{eq:fullRLinearConvergence}
        \Eta_\ell^k
        \le
        C_{\textup{lin}}\, q_{\textup{lin}}^{|\ell, k|-|\ell',k'|} \,
        \Eta_{\ell'}^{k'}
        \quad\text{for all } (\ell, k), (\ell', k') \in \QQ \text{ with } |\ell', k'| \le |\ell, k|.
    \end{equation}
    The constants $C_{\textup{lin}}$ and $q_{\textup{lin}}$ are independent of the mesh size and the iteration number, but depend only on
    $\alpha, \delta_0, \kmin, q_0, \qred, \rr_0, \Bup_0, \CCstab, C_{\textup{Céa}}, \Crel,
    C_{\mathrm{orth}}(\varepsilon), \delta_{\textup{min}}, \theta, \lambdalin$
    with
    $C_{\mathrm{orth}}(\varepsilon)$ from~\eqref{eq:quasiorthogonality} for $\varepsilon>0$ chosen below in Step~3 of the proof.
  \end{theorem}

\begin{proof}
    The proof of Theorem~\ref{theorem:fullRLinearConvergence} employs
    the summability-based proof strategy from~\cite{bfmps2025} allowing
    to establish R-linear convergence of the quasi-error $\Eta_\ell^k$.
    It consists of eight steps.

    \emph{Step~1 (estimator contraction).}
    Recall that the axiom~\eqref{axiom:reduction} with $\qred$ leads to the Dörfler contraction
    \cite[Equation~(51)]{ghps2021}
    \begin{equation}
        \label{eq:Doerfler_contraction}
        \eta_{\ell+1}(u_{\ell}^\kmax)
        \le
        q_\theta \, \eta_{\ell}(u_{\ell}^\kmax)
        \quad\text{with}\quad
        0 < q_\theta \coloneqq [1 - (1-\qred^2) \, \theta]^{1/2} < 1.
    \end{equation}
    This and the stability~\eqref{axiom:stability} with uniform constant \(\CCstab\) from \eqref{eq:uniform_constants} prove
    \begin{equation*}
        \eta_{\ell+1}(u_{\ell+1}^{\kmax})
        \eqreff{axiom:stability}\le
        \eta_{\ell+1}(u_\ell^\kmax)
        + \CCstab \, \enorm{u_{\ell+1}^{\kmax} - u_\ell^\kmax}
        \eqreff{eq:Doerfler_contraction}\le
        q_\theta \, \eta_{\ell}(u_{\ell}^\kmax)
        + \CCstab \, \enorm{u_{\ell+1}^{\kmax} - u_\ell^\kmax}.
    \end{equation*}
    Choose \(\nu > 0\) sufficiently small such that \(q_\nu \coloneqq (1 + \nu) q_\theta^2 < 1\)
    and abbreviate \(C_\nu \coloneqq (1 + \nu^{-1}) \CCstab^2\).
    The Young inequality with this parameter \(\nu\) shows
    \begin{equation}
        \label{eq:convergence:Doerfler_contraction}
        \eta_{\ell+1}(u_{\ell+1}^{\kmax})^2
        \le
        q_\nu \, \eta_{\ell}(u_{\ell}^\kmax)^2
        + C_\nu \, \enorm{u_{\ell+1}^{\kmax} - u_\ell^\kmax}^2.
    \end{equation}

    \emph{Step~2 (residual contraction).}
    For the contraction constant \(0 < q_0 < 1\) from~\eqref{eq:Newton:residual_contraction},
    choose \(\mu > 0\) sufficiently small such that
    \(q_\mu \coloneqq (1 + \mu) q_0^2 < 1\) and set
    \(C_\mu \coloneqq (1 + \mu^{-1}) \alpha^2 q_0^2\).
    Then, the combination of the residual contraction~\eqref{eq:Newton:residual_contraction} and the Young inequality reads
    \begin{equation}
        \label{eq:convergence:residual_contraction}
        \norm{F - \AA u_{\ell+1}^{\kmax}}_{\XX_{\ell+1}'}^2
        \eqreff{eq:Newton:residual_contraction}\le
        q_\mu \, \norm{F - \AA u_\ell^{\kmax}}_{\XX_{\ell}'}^2
        +
        C_\mu \, \enorm{u_{\ell+1}^\exact - u_\ell^\exact}^2.
    \end{equation}
    The linearization error estimate~\eqref{eq:linearization_estimator}
    and the residual contraction~\eqref{eq:Newton:residual_contraction} yield
    \begin{align*}
        \enorm{u_{\ell+1}^{\kmax}- u_\ell^{\kmax}}
        &\le
        \enorm{u_{\ell+1}^\exact - u_{\ell+1}^{\kmax}}
        + \enorm{u_\ell^\exact - u_\ell^{\kmax}}
        + \enorm{u_{\ell+1}^\exact - u_\ell^\exact}
        \\
        &\eqreff*{eq:linearization_estimator}\le
        \alpha^{-1} \,\norm{F - \AA u_{\ell+1}^{\kmax}}_{\XX_{\ell+1}'}
        + \alpha^{-1} \,\norm{F - \AA u_\ell^{\kmax}}_{\XX_{\ell}'}
        + \enorm{u_{\ell+1}^\exact - u_\ell^\exact}
        \\
        &\eqreff*{eq:Newton:residual_contraction}\le
        (1 + q_0)\alpha^{-1} \, \norm{F - \AA u_\ell^{\kmax}}_{\XX_{\ell}'}
        + (1 + q_0)\, \enorm{u_{\ell+1}^\exact - u_\ell^\exact}.
    \end{align*}
    The Young inequality thus proves
    \begin{equation}
        \label{eq:convergence:final_iterates_difference}
        \enorm{u_{\ell+1}^{\kmax} - u_\ell^{\kmax}}^2
        \le
        2\, (1 + q_0)^2 \alpha^{-2} \, \norm{F - \AA u_\ell^{\kmax}}_{\XX_{\ell}'}^2
        + 2\, (1 + q_0)^2 \, \enorm{u_{\ell+1}^\exact - u_\ell^\exact}^2.
    \end{equation}

    \emph{Step~3 (contraction of quasi-error with respect to \(\ell\)).}
    The stability~\eqref{axiom:stability} and the linearization error estimate~\eqref{eq:linearization_estimator} imply
    \begin{equation}
        \label{eq:estimate_exact_estimator}
        \eta_\ell(u_\ell^\exact)^2 \, \,
        \eqreff*{axiom:stability}\leq \, \,
        2 \,\eta_\ell(u_\ell^\kmax)^2
        +
        2 \CCstab^2 \, \enorm{u_\ell^\exact - u_\ell^\kmax}^2
        \eqreff{eq:linearization_estimator}\le
        2 \, \eta_\ell(u_\ell^\kmax)^2
        +
        2 \CCstab^2 \alpha^{-2} \,
        \norm{F - \AA u_\ell^\kmax}_{\XX_\ell'}^2.
    \end{equation}
    As an intermediate result, we investigate the contraction of the weighted quasi-error
    \begin{equation*}
        \widetilde\Eta_\ell
        \coloneqq
        \big[
            \norm{F-\AA u_\ell^{\kmax}}_{\XX_\ell'}^2
            +
            \gamma \, \eta_\ell(u_\ell^{\kmax})^2
        \big]^{1/2}
        \quad \text{ for } (\ell, \kmax) \in \QQ
    \end{equation*}
    for any parameter \(0 < \gamma < 2^{-1}(1 + q_0)^{-2} \alpha^2 C_\gamma^{-1} (1 - q_\mu)\) ensuring that
    \[
        q_\gamma
        \coloneqq
        q_\mu
        +
        2 (1 + q_0)^2 \alpha^{-2} C_\nu \gamma < 1.
    \]
    Using \(C_\gamma \coloneqq C_\mu + 2 (1 + q_0)^2 C_\nu \gamma\),
    the combination of the
    estimates~\eqref{eq:convergence:Doerfler_contraction}--\eqref{eq:convergence:final_iterates_difference} results in
    \begin{align*}
        (\widetilde\Eta_{\ell+1})^2
        &=
        \norm{F - \AA u_{\ell+1}^{\kmax}}_{\XX_{\ell+1}'}^2
        +
        \gamma \, \eta_{\ell+1}(u_{\ell+1}^{\kmax})^2
        \\
        &\eqreff*{eq:convergence:Doerfler_contraction}\le \, \,
        \norm{F - \AA u_{\ell+1}^{\kmax}}_{\XX_{\ell+1}'}^2
        +
        q_\nu\,\gamma \,  \eta_\ell(u_{\ell}^\kmax)^2
        +
        C_\nu \gamma \, \enorm{u_{\ell+1}^{\kmax} - u_\ell^{\kmax}}^2
        \\
        &\eqreff*{eq:convergence:residual_contraction}\le \, \,
        q_\mu \norm{F - \AA u_{\ell}^{\kmax}}_{\XX_{\ell}'}^2
        +
        C_\mu \, \enorm{u_{\ell+1}^\exact - u_\ell^\exact}^2
        +
        q_\nu\gamma \,  \eta_\ell(u_{\ell}^\kmax)^2
        +
        C_\nu\gamma \,  \enorm{u_{\ell+1}^{\kmax} - u_\ell^{\kmax}}^2
        \\
        &\eqreff*{eq:convergence:final_iterates_difference}\le \, \,
        q_\gamma
        \norm{F - \AA u_{\ell}^{\kmax}}_{\XX_{\ell}'}^2
        +
        q_\nu \gamma \,
        \eta_{\ell}(u_{\ell}^\kmax)^2
        +
        C_\gamma \,
        \enorm{u_{\ell+1}^\exact - u_\ell^\exact}^2.
    \end{align*}
    From~\eqref{eq:estimate_exact_estimator}, we infer that
    \begin{equation*}
        \begin{split}
            0
            &\leq
            2 C_\gamma \varepsilon \,
            \eta_\ell(u_\ell^\kmax)^2
            +
            2 \overline C_{\textup{stab}}^2 \alpha^{-2} C_\gamma \varepsilon \,
            \norm{F - \AA u_\ell^\kmax}_{\XX_\ell'}^2
            -
            C_\gamma \varepsilon \,
            \eta_\ell(u_\ell^\exact)^2.
        \end{split}
    \end{equation*}
    The sum of the previous two formulas reads
    \begin{equation}
        \label{eq:linearConvergenceEstimate}
        \begin{split}
            (\widetilde\Eta_{\ell+1})^2
            &\leq
            \big[
                q_\gamma
                +
                2 \CCstab^2 \alpha^{-2}C_\gamma\varepsilon
            \big] \,
            \norm{F - \AA u_\ell^{\kmax}}_{\XX_{\ell}'}^2
            \\
            &\phantom{{}\leq{}}
            +
            \big[
                q_\nu
                +
                2 \gamma^{-1} C_\gamma \varepsilon
            \big] \,
            \gamma \,
            \eta_\ell(u_\ell^{\kmax})^2
            +
            C_\gamma \,
            \big[
                \enorm{u_{\ell+1}^\exact - u_\ell^\exact}^2
                -
                \varepsilon \, \eta_\ell(u_\ell^\exact)^2
            \big].
        \end{split}
    \end{equation}
    The choice of the parameter \(\varepsilon = \varepsilon[\gamma] > 0\) such that
    \[
        \varepsilon
        <
        \min\big\{
            2^{-1}\CCstab^{-2}\alpha^2 C_\gamma^{-1} (1 - q_\gamma),\:
            2^{-1} C_\gamma^{-1} \gamma (1 - q_\nu)
        \big\}
    \]
    ensures that
    \(
        0 <
        q[\gamma,\varepsilon]
        \coloneqq
        \max\{
            q_\gamma
            +
            2 \CCstab^2 \alpha^{-2} C_\gamma \varepsilon,
            q_\nu
            +
            2 C_\gamma \gamma^{-1} \varepsilon
        \}
        < 1
    \).
    Abbreviating the remainder
    $(\textup{R}_\ell)^2 \coloneqq \enorm{u_{\ell+1}^\exact - u_\ell^\exact}^2 - \varepsilon \, \eta_\ell(u_\ell^\exact)^2$,
    the estimate~\eqref{eq:linearConvergenceEstimate} reads
    \begin{equation}
        \label{eq:summability_criterion}
        (\widetilde\Eta_{\ell+1})^2
        \le
        q[\gamma, \varepsilon] \, (\widetilde\Eta_{\ell})^2
        +
        C_\gamma \, (\textup{R}_\ell)^2
        \quad \text{ for all } \ell \text{ with } (\ell+1,\kmax) \in \QQ
    \end{equation}

    \emph{Step~4 (tail-summability with respect to \(\ell\)).}
    Given \(0 \leq \ell' < \elll\), quasi-orthogonality~\eqref{eq:quasiorthogonality}, quasi-monotonicity~\eqref{eq:quasi-monotonicity} of the estimators, and~\eqref{eq:estimate_exact_estimator}
    establish the tail-summability of the remainder \(\textup{R}_\ell\)
    \begin{equation}
        \label{eq:convergence:remainder_summability}
        \sum_{\ell=\ell'+1}^{\lmax-1}
        (\textup{R}_{\ell})^2
        \eqreff{eq:quasiorthogonality}\lesssim
        \eta_{\ell'}(u_{\ell'}^\exact)^2
        \eqreff{eq:estimate_exact_estimator}\lesssim
        (\Eta_{\ell'}^\kmax)^2.
    \end{equation}
    For all \(N \in \N_0\) with \(\ell'+N < \elll\),
    reliability~\eqref{axiom:reliability}, estimate~\eqref{eq:estimate_exact_estimator},
    and the quasi-monotonicity~\eqref{eq:quasi-monotonicity} of the estimator show that
    \begin{equation}
        \label{eq:convergence:remainder_stability}
        \begin{split}
            (\textup{R}_{\ell'+N})^2
            &\le
            \enorm{u_{\ell'+N+1}^\exact - u_{\ell'+N}^\exact}^2
            +
            \eta_{\ell'+N}(u_{\ell'+N}^\exact)^{2}
            \\
            &\eqreff*{axiom:reliability}\lesssim \, \,
            \eta_{\ell'+N+1}(u_{\ell'+N+1}^\exact)^2
            +
            \eta_{\ell'+N}(u_{\ell'+N}^\exact)^2
            \eqreff{eq:quasi-monotonicity}\lesssim
            \eta_{\ell'}(u_{\ell'}^\exact)^2
            \eqreff{eq:estimate_exact_estimator}\lesssim
            (\Eta_{\ell'}^\kmax)^2.
        \end{split}
    \end{equation}
    The contraction~\eqref{eq:summability_criterion}, the tail-summability~\eqref{eq:convergence:remainder_summability}, and stability~\eqref{eq:convergence:remainder_stability}
    of the remainder (by exploiting the equivalence \(\widetilde{\Eta}_\ell \eqsim \Eta_\ell^{\kmax}\)) verify the assumptions of the tail-summability criterion~\cite[Lemma~1]{bfmps2025}.
    Hence,
    \begin{equation}
        \label{eq:contractionFinalIterates}
        \sum_{\ell = \ell'+1}^{\lmax-1} \Eta_{\ell}^{\kmax}
        \lesssim
        \Eta_{\ell'}^{\kmax}
        \quad\text{ for all } 0 \leq \ell' < \lmax.
    \end{equation}

    \emph{Step~5 (quasi-contraction with respect to \(k\)).}
    Let the mesh level $\ell \in \N_0$ with \((\ell,0) \in \QQ\) be arbitrary.
    For any linearization indices \(k',k \in \N_0\) with $0 \le k' \le k \le \kmax[\ell]$, the
    guaranteed reduction~\eqref{eq:damping:guaranteedContraction} of the residual from Lemma~\ref{lemma:adaptiveDamping}
    and the linearization error estimate~\eqref{eq:linearization_estimator} imply
    \begin{align*}
        \enorm{u_\ell^{k} - u_\ell^{k'}}
        &\leq
        \enorm{u_\ell^{\exact} - u_\ell^{k}}
        +
        \enorm{u_\ell^{\exact} - u_\ell^{k'}}
        \eqreff{eq:linearization_estimator}\leq
        \alpha^{-1} \,
        \big[
            \norm{F - \AA u_\ell^{k}}_{\XX_\ell'}
            +
            \norm{F - \AA u_\ell^{k'}}_{\XX_\ell'}
        \big]
        \\
        &\eqreff*{eq:damping:guaranteedContraction}\leq
        \alpha^{-1} (1+\rr_0^{k-k'}) \,
        \norm{F - \AA u_\ell^{k'}}_{\XX_\ell'}
        \le
        \alpha^{-1} (1+\rr_0) \,
        \norm{F - \AA u_\ell^{k'}}_{\XX_\ell'}.
    \end{align*}
    This and the stability~\eqref{axiom:stability} with constant \(\CCstab\) from~\eqref{eq:uniform_constants} prove
    \begin{equation}
        \label{eq:convergence:linearization_contraction}
        \eta_\ell(u_\ell^{k})
        \eqreff{axiom:stability}\le
        \eta_\ell(u_\ell^{k'})
        + \CCstab \,
        \enorm{u_\ell^{k}- u_\ell^{k'}}
        \le
        \eta_\ell(u_\ell^{k'})
        +
        \CCstab \alpha^{-1} (1 + \rr_0) \,
        \norm{F - \AA u_\ell^{k'}}_{\XX_\ell'}.
    \end{equation}

	Hence, we obtain the stability estimate, for any $0 \le k' \le k \le \kmax[\ell]$,
    \begin{equation}
        \label{eq:quasierror:contraction_in_linearization}
        \Eta_\ell^{k}
        =
        \norm{F - \AA u_\ell^{k}}_{\XX_\ell'}
        +
        \eta_\ell(u_\ell^k)
        \eqreff{eq:convergence:linearization_contraction}\lesssim
        \norm{F - \AA u_\ell^{k}}_{\XX_\ell'}
        +
        \norm{F - \AA u_\ell^{k'}}_{\XX_\ell'}
        +
        \eta_\ell(u_\ell^{k'})
        \eqreff{eq:damping:guaranteedContraction}\lesssim
        \Eta_\ell^{k'}.
    \end{equation}

    If we additionally suppose that \(0 < k_{\textup{min}} \leq k' \leq k < \kmax[\ell]\),
    then the second condition of the linearization stopping criterion~\eqref{algorithm:AILFEM:stopping}
    in Algorithm~\ref{algorithm:AILFEM}{(i)} is not met, i.e.,
    \begin{equation}\label{eq:notMetxFalse}
        \eta_\ell(u_\ell^k)
        <
        \lambdalin^{-1}\,\norm{F - \AA u_\ell^k}_{\XX_\ell'}.
    \end{equation}
    Hence,
    the estimate~\eqref{eq:convergence:linearization_contraction}
    and the guaranteed reduction~\eqref{eq:damping:guaranteedContraction} of the residual prove
    \begin{align*}
        \Eta_\ell^{k+1}
        &=
        \norm{F - \AA u_\ell^{k+1}}_{\XX_\ell'}
        +
        \eta_\ell(u_\ell^{k+1})
        \\
        &\eqreff*{eq:convergence:linearization_contraction}\leq \,
        \norm{F - \AA u_\ell^{k+1}}_{\XX_\ell'}
        +
        \eta_\ell(u_\ell^{k}) + \CCstab \alpha^{-1} (1+\rr_0) \,
        \norm{F - \AA u_\ell^{k}}_{\XX_\ell'}
        \\
        &\eqreff*{eq:notMetxFalse}\le \,
        \norm{F - \AA u_\ell^{k+1}}_{\XX_\ell'}
        +
        \big[\lambdalin^{-1} + \CCstab \alpha^{-1} (1+\rr_0) \big] \,
        \norm{F - \AA u_\ell^{k}}_{\XX_\ell'}
        \\
        &\eqreff*{eq:damping:guaranteedContraction}\leq \,
        \big[ 1 + \lambdalin^{-1} \rr_0^{-1} + \CCstab \alpha^{-1} (\rr_0^{-1} + 1) \big]\, \rr_0^{k + 1 - k'} \,
        \norm{F - \AA u_\ell^{k'}}_{\XX_\ell'}.
    \end{align*}
    This immediately verifies the quasi-contraction
    \begin{equation}
        \label{eq:quasierror:quasicontraction:case1}
        \Eta_\ell^{k+1}
        \lesssim
        \rr_0^{k+1-k'} \, \norm{F- \AA u_\ell^{k'}}_{\XX_\ell'}
        \leq
        \rr_0^{k+1-k'} \, \Eta_\ell^{k'}
        \quad \text{ for all } 0 < \kmin \leq k' \leq k < \kmax[\ell]
    \end{equation}
    and further applications of the guaranteed reduction~\eqref{eq:damping:guaranteedContraction} result in
    \begin{equation}
        \label{eq:quasierror:quasicontraction:case2}
        \Eta_\ell^{k+1}
        \lesssim
        \rr_0^{k+1} \, \norm{F- \AA u_\ell^0}_{\XX_\ell'}
        \leq
        \rr_0^{k+1} \, \Eta_\ell^0
        \quad \text{ for all } 0 < k_{\textup{min}} \leq k < \kmax[\ell].
    \end{equation}

    \emph{Step 6 (tail-summability with respect to \(k\)).}
    For a fixed mesh level \(\ell' \in \N_0\) with \((\ell', 0) \in \QQ\),
    recall that the stopping criterion~\eqref{algorithm:AILFEM:stopping} guarantees \(\kk[\ell'] \geq \kmin\).
    The argumentation for the summability in \(k\) differs depending on the relation of \(\kmin\)
    and the smallest index \(0 \leq k' \leq \kk[\ell']\) in the sum.
    If \(k' \geq \kmin\), the estimate~\eqref{eq:quasierror:quasicontraction:case1}
    and the convergence of the geometric series yield
    \begin{equation*}
        \sum_{k = k' + 1}^{\kk[\ell']} \Eta_{\ell'}^{k}
        =
        \sum_{k = k'}^{\kk[\ell']} \Eta_{\ell'}^{k+1}
        \eqreff{eq:quasierror:quasicontraction:case1}\lesssim
        \bigg(
            \sum_{k = k'}^{\kk[\ell']} \rr_0^{k-k'}
        \bigg) \Eta_{\ell'}^{k'}
        \lesssim
        \Eta_{\ell'}^{k'}.
    \end{equation*}
    If \(k' < \kmin\), the application
    of~\eqref{eq:quasierror:quasicontraction:case1} for the indices \(\kmin < k < \infty\) and
    of~\eqref{eq:quasierror:contraction_in_linearization} for the indices \(0 \leq k \leq \kmin\)
    result in
    \begin{align*}
        \sum_{k = k' + 1}^{\kk[\ell']} \Eta_{\ell'}^{k}
        =
        \!\!
        \sum_{k = k' + 1}^{\kmin}
        \!\!
        \Eta_{\ell'}^{k}
        +
        \!\!\!
        \sum_{k = \kmin + 1}^{\kk[\ell']}
        \!\!\!
        \Eta_{\ell'}^{k}
        &\eqreff*{eq:quasierror:quasicontraction:case1}\lesssim
        \sum_{k = k' + 1}^{\kmin}
            \!\!
        \Eta_{\ell'}^{k}
        +
        \bigg(
            \sum_{k = \kmin + 1}^{\kk[\ell']} \!\!\!\!\!
            \rr_0^{k-\kmin}
        \bigg)
        \Eta_{\ell'}^{\kmin}
        \lesssim
        \sum_{k = k' + 1}^{\kmin} \Eta_{\ell'}^{k}
        \eqreff{eq:quasierror:contraction_in_linearization}\lesssim
        \Eta_{\ell'}^{k'}.
    \end{align*}
    The two previous displayed formulas hold for finite and infinite \(\kk[\ell']\) and thus establish
    \begin{equation}
        \label{eq:tailsummability:k}
        \sum_{k = k' + 1}^{\kk[\ell']} \Eta_{\ell'}^{k}
        \lesssim
        \Eta_{\ell'}^{k'}
        \quad \text{for all } (\ell', k') \in \QQ.
    \end{equation}

    \emph{Step~7 (stability of mesh-level change).}
    For $\ell>0$, note that Dörfler contraction~\eqref{eq:Doerfler_contraction} and the dual norm increase~\eqref{eq:dualNormIncrease}
    (and nested iteration $u_\ell^0 = u_{\ell-1}^\kmax$) yield
    \begin{equation*}
        \Eta_{\ell}^0
        =
        \norm{F-\AA u_\ell^0}_{\XX_\ell'} + \eta_\ell(u_\ell^0)
        \eqreff{eq:Doerfler_contraction}\lesssim
        \norm{F-\AA u_\ell^0}_{\XX_\ell'}
        +
        q_\theta \, \eta_{\ell-1}(u_{\ell-1}^{\kmax})
        \eqreff{eq:dualNormIncrease}\lesssim
        \Eta_{\ell-1}^\kmax
        +
        \enorm{u_{\ell}^\exact - u_{\ell-1}^\exact}.
    \end{equation*}
    Reliability~\eqref{axiom:reliability}, quasi-monotonicity~\eqref{eq:quasi-monotonicity}
    of the error estimator, stability~\eqref{axiom:stability} with constant \(\CCstab\) from~\eqref{eq:uniform_constants},
    and the linearization error estimate~\eqref{eq:linearization_estimator} prove
    \begin{align*}
        \enorm{u_\ell^\exact - u_{\ell-1}^\exact} \,
        &\eqreff*{axiom:reliability}\leq \,
        \Crel\,
        \big[
            \eta_\ell(u_\ell^\exact)
            +
            \eta_{\ell-1}(u_{\ell-1}^\exact)
        \big]
        \eqreff{eq:quasi-monotonicity}\leq
        \Crel (1 + \Cmon)\,
        \eta_{\ell-1}(u_{\ell-1}^\exact)
        \\
        &\eqreff*{axiom:stability}\leq \,
        \Crel (1 + \Cmon)\,
        \big[
            \eta_{\ell-1}(u_{\ell-1}^\kmax)
            +
            \CCstab \, \enorm{u_{\ell-1}^\exact - u_{\ell-1}^\kmax}
        \big]
        \\
        &\eqreff*{eq:linearization_estimator}\leq \,
        \Crel (1 + \Cmon)\,
        \big[
            \eta_{\ell-1}(u_{\ell-1}^\kmax)
            +
            \CCstab \alpha^{-1} \,
            \norm{F - \AA u_{\ell-1}^\kmax}_{\XX_{\ell-1}'}
        \big]
        \lesssim
        \Eta_{\ell-1}^{\kmax}.
    \end{align*}
    Combining the two previous displayed formulas yields the stability of the quasi-error under mesh refinement
    \begin{equation}
        \label{eq:quasierror:stability_in_refinement}
        \Eta_\ell^0
        \lesssim
        \Eta_{\ell-1}^{\kmax}
        \quad\text{for all } \ell \in \QQ \text{ with } \ell > 0.
    \end{equation}

    \emph{Step~8 (tail-summability with respect to \(\ell\) and \(k\)).}
  The conclusion of the proof of tail-summability (and thus also full R-linear convergence)
        \[
        \sum_{\substack{(\ell, k) \in \QQ \\ \abs{\ell', k'} < \abs{\ell,k}}}
        \Eta_{\ell}^{k}
        \lesssim
        \Eta_{\ell'}^{k'},
    \]
follows the argumentation found in~\cite[Theorem~2]{bfmps2025}. It employs the tail-summabilities~\eqref{eq:contractionFinalIterates} with respect to $\ell$ and~\eqref{eq:tailsummability:k} with respect to $k$ as well as the stability~\eqref{eq:quasierror:contraction_in_linearization} and the stability~\eqref{eq:quasierror:stability_in_refinement} of the mesh-level change. The reasoning needs to respect that the quasi-contraction~\eqref{eq:quasierror:quasicontraction:case2} is restricted to $k > \kmin$. It thus slightly differs depending on different cases for $\ell'$ and $\elll$, but contains only the mentioned arguments. Further details are omitted here, but a full argument is provided in Appendix~\ref{appendix:linearConvergence}.

\end{proof}

\subsection{Optimal convergence rates}\label{subsection:optimalConvergenceRates}
The formal statement of optimal convergence rates of the output sequence of Algorithm~\ref{algorithm:AILFEM}
employs the notion of nonlinear approximation classes as introduced in the context of adaptive finite element methods
by~\cite{bdd2004}.
For \(N \in \N_0\), we write \(\TT_\fine \in \T_N(\TT_\coarse)\)
if \(\TT_\fine \in \T(\TT_\coarse)\) and \(\# \TT_\fine - \# \TT_\coarse \le N\).
Let $u_{\rm opt}^\exact$ denote the exact discrete solution associated
with with respect to the unavailable optimal triangulation $\TT_{\rm opt} \in \T_N (\TT_0)$.
For any rate \(r > 0\), we say that $u^\exact$ belongs to the approximation class of rate $r>0$ if
\begin{equation*}
    \norm{u^\exact}_{\mathbb{A}_r}
    \coloneqq
    \sup_{N \in \N_0}
    \bigl[
        (N + 1)^r \min_{\TT_{\rm opt} \in \T_N }
        \eta_{\rm opt}(u^\exact_{\rm opt})
    \bigr] < \infty.
\end{equation*}
This means that the error estimator values on the optimally chosen meshes decrease at least with rate $- r$
with respect to the number \(N\) of additional triangles.
In order to take the cumulative nature of adaptive algorithms into account,
the following theorem replaces the number of elements \(\# \TT_\ell\) at one level \(\ell \in \N_0\) by the cumulative cost function
\begin{equation}
    \label{eq:cost}
    \mathtt{cost}(\ell, k)
    \coloneqq
    \sum_{\substack{(\ell', k') \in \QQ\\ |\ell', k'| \le |\ell, k| }} \# \TT_{\ell'}.
\end{equation}
\begin{theorem}[optimal complexity]
    \label{theorem:optimalRates}
    Suppose~\eqref{assump:stronglyMonotone}--\eqref{assump:derivativeLocallyLipschitz}
    as well as the axioms~\eqref{axiom:stability}--\eqref{axiom:discreteReliability}
    and quasi-orthogonality~\eqref{axiom:quasiorthogonality}.
    Let the minimal number \(\kmin \in \N\) of Newton iterations be chosen such that~\eqref{eq:Newton:condition_k0} is satisfied and $\varepsilon >0$ as in Theorem~\ref{theorem:fullRLinearConvergence}.
    Let \(0 < \theta < 1\), $1 \le \Cmark <\infty$, and
    \(0 < \lambdalin < \lambdalin^\exact \coloneqq \min \{1, \alpha / \overline{C}_{\textup{stab}}\}\) with $\CCstab$ from\eqref{eq:uniform_constants} such that
    \begin{equation}
        \label{eq:thetamark}
        0 <
        \thetamark
        \coloneqq
        \frac{(\theta^{1/2}+ \, \lambdalin / \lambdalin^\exact)^2}{(1-\lambdalin / \lambdalin^\exact)^2}
        <
        \theta^\exact
        \coloneqq
        \frac1{1 + \overline{C}_{\textup{stab}}^2 \, \Crel^2}
        <
        1.
    \end{equation}
    Then, for all \(r > 0\), it holds that
    \begin{equation}
        \label{eq:optimal_complexity}
        \copt \, \norm{u^\exact}_{\mathbb{A}_r}
        \le
        \!\sup_{(\ell,k) \in \QQ} (\# \TT_\ell - \# \TT_0 +1)^r \, \Eta_\ell^k
        \le
\sup_{(\ell,k) \in \QQ}
        \! \mathtt{cost}(\ell, k)^r \,
        \Eta_\ell^{k}
        \le
        \Copt \,
        \max\{\norm{u^\exact}_{\mathbb{A}_r}, \, \Eta_0^0\}.
    \end{equation}
    The constant $\copt>0$ depends only on $\overline{C}_{\textup{stab}}$, $r$, and the use of NVB refinement,
    whereas \(\Copt>0\) depends on $q_0$, $\qred$, $\alpha$, \(\overline{C}_{\textup{stab}}\), \(\Crel\),
    \(\Cdrel\), \(\Cmark\), $\Clin$, $\qlin$, \(\# \TT_{0}\), \(r\).
\end{theorem}
The statement can be interpreted as follows: Given that $\norm{u^\exact}_{\mathbb{A}_r}< \infty$,
then Algorithm~\ref{algorithm:AILFEM} converges with the rate $-r$ in the sense that
\(\Eta_\ell^k \lesssim \mathtt{cost}(\ell, k)^{-r}\).
The proof of Theorem~\ref{theorem:optimalRates} follows standard perturbation arguments
and is therefore postponed to Appendix~\ref{appendix:optimalRates}.

\begin{remark}\label{remark:linearSolver}
    In general, the convergence analysis with rates of adaptive mesh-refining algorithms
    allows the inexact solution of the involved linear systems.
    In Algorithm~\ref{algorithm:AILFEM}, this concerns the linearized Newton system~\ref{algorithm:NAIL:NewtonUpdate} as well as
    of the Riesz problem~\eqref{eq:residual_Riesz} for the evaluation of the linearization error estimator.
    Following \cite{bfmps2025}, this results in optimal convergence rates with respect to the computational time.
    Since the proof of Lemma~\ref{lemma:adaptiveDamping} reveals that
    the number of trials in each Newton iteration is uniformly bounded,
    the computational cost from~\eqref{eq:cost} is a valid measure for the computational effort in each iteration
    under the assumption of linear costs for the linear solver.
    The inclusion of this notion into Theorem~\ref{theorem:optimalRates} is a first step into the direction
    of the analysis of the practical computational complexity of Algorithm~\ref{algorithm:AILFEM}.
    However, the explicit analysis of suitable contractive algebraic solvers
    with discerning stopping criteria is beyond the scope of this paper.
\end{remark}



\section{Application to semilinear elliptic PDEs}\label{section:applications}

This section applies the analysis from the Sections~\ref{section:linearization}--\ref{section:mainResults} to semilinear elliptic PDEs.
To this end, the key assumptions~\eqref{assump:stronglyMonotone}--\eqref{assump:derivativeLocallyLipschitz} are verified.

\subsection{Semilinear model problem}
Given a polyhedral bounded Lipschitz domain $\Omega \subset \R^d$ of dimension $d \in \{1, 2,3\}$ with conforming initial triangulation $\TT_0$,
consider the right-hand sides \(f \in L^2(\Omega)\) and piecewise Lipschitz continuous $\boldsymbol{f} \in [W^{1,\infty}(\TT_0)]^d$.
Suppose that the diffusion coefficient $\boldsymbol{A} \in [W^{1,\infty}(\TT_0)]^{d \times d}_{\textup{sym}}$
satisfies uniform two-sided eigenvalue bounds
\begin{equation}\label{eq:semilinear:ellipticity}
    0
    <
    \alpha_{\textup{min}}
    \le
    \boldsymbol{A}(x) \xi \cdot \xi
    \le
    \alpha_{\textup{max}}
    \quad \text{ for all } \xi \in \R^d \text{ with } \vert \xi \vert = 1
    \text{ and a.e.\ } x \in \Omega,
\end{equation}
where $|\,\cdot\,|$ is understood as the Euclidean norm. We equip the space \(\XX \coloneqq H^1_0(\Omega)\) with the norm $\enorm{u}^2 \coloneqq \langle \A \nabla u, \nabla u \rangle \coloneqq \langle \A \nabla u, \nabla u \rangle_{L^2(\Omega)}$ for \(u \in \XX\).
For simplicity, we assume that the convection coefficient $\boldsymbol{b} \in [W^{1,\infty}(\Omega)]^d$ satisfies \(- \div \boldsymbol{b} \geq 0\)
ensuring the ellipticity
\begin{equation}
    \label{eq:semilinear:ellipticity_linear_part}
    \enorm{u-v}^2
    \leq
    \dual{\boldsymbol{A} \nabla (u - v)}{\nabla (u - v)}
    +
    \dual{\boldsymbol{b}\cdot \nabla (u - v)}{u-v}
    \quad \text{ for all } u, v \in \XX.
\end{equation}
The boundedness of the coefficients $\boldsymbol{A}$ and $\boldsymbol{b}$ implies the estimate
\begin{equation}
    \label{eq:semilinear:boundedness_linear_part}
    \dual{\boldsymbol{A} \nabla u}{\nabla v}
    +
    \dual{\boldsymbol{b}\cdot \nabla u}{v}
    \lesssim
    \enorm{u} \, \enorm{v}
    \quad \text{ for all } u, v \in \XX.
\end{equation}
The assumptions on the nonlinear reaction term $c \colon \Omega \times \R \to \R$ are discussed below.
The semilinear elliptic PDE seeks \(u^\exact \in \XX\) such that
\begin{equation}
    \label{eq:semilinearStrongform}
    -\div(\A \nabla u^\exact) + \boldsymbol{b}\cdot \nabla u^\exact +  c(u^\exact) = f - \div \f
    \text{ \ in } \Omega
    \quad \text{ and } \quad
    u^\exact = 0 \text{ on }  \partial \Omega.
\end{equation}
Throughout this section, the abbreviation $c(u^\exact) \equiv c(\cdot, u^\exact(\cdot)) \colon \Omega \to \R$ applies.
The weak form of the semilinear model problem involves the operator $\AA \colon \XX \to \XX'$ with
\begin{equation}
    \label{eq:semilinear:operator}
    \AA u
    \coloneqq
    \dual{\A \nabla u}{\nabla \, \cdot \,} + \dual{\boldsymbol{b}\cdot \nabla u + c(u)}{\cdot \,}
    \quad \text{ for all } u \in \XX.
\end{equation}
Hence, the weak solution $u^\exact \in \XX$ to~\eqref{eq:semilinearStrongform} solves
\begin{equation}
    \label{eq:semilinearWeakform}
    \dual{\mathcal{A} u^\exact}{v}
    =
    \dual{f}{v}
    +
    \dual{\boldsymbol{f}}{v}
    \quad \text{ for all } v \in \XX.
\end{equation}
Following~\cite[Chapter III, (12)]{fk1980} and~\cite[(A1)--(A4)]{bhsz2011},
the nonlinearity $c \in C^n(\Omega \times \R)$  for all $n \in \N_0$ is supposed to satisfy the following three assumptions:
\begin{description}
    \item[\bfseries (CAR) Carathéodory function]
        \labeltext{CAR}{assump:Caratheodory}
		The $n$-th derivative $c^{(n)} \coloneqq \partial_{\xi}^n c$ with respect to the second argument satisfies that
		the function $x \mapsto \partial_{\xi}^n \, c(x,\xi)$ is measurable on $\Omega$ for all $\xi \in \R$
        and that $\xi \mapsto \partial_{\xi}^n \, c(x,\xi)$ is continuous for all $x \in \Omega$.
    \item[\bfseries (MON) Monotonicity of the reaction $c$]
        \labeltext{MON}{assump:monotonicity_c}
        Let $c^{(1)}(x,\xi) \ge 0$ for all $x \in \Omega$ and $\xi \in \R$.
    \item[\bfseries (GC) Growth condition]
        \labeltext{GC}{assump:growth_c}
        There exist $R > 0$ and $N \in \N$ such that
        \begin{equation*}
            |c^{(N)}(x,\xi)|
            \le
            R
            \quad \text{ for a.e. } x \in \Omega \text{ and all } \xi \in \R.
        \end{equation*}
        If \(d = 3\), then assume that this holds for some \(N \in \{2, 3\}\).
\end{description}
The assumption~\eqref{assump:monotonicity_c} ensures the monotonicity of \(c\) in the second argument, i.e.,
\begin{equation*}
    \big[ c(x, \xi_2) - c(x, \xi_1) \big] (\xi_2-\xi_1) \ge 0
    \quad \text{ for all } x \in \Omega \text{ and all } \xi_1, \xi_2 \in \R.
\end{equation*}
This and the ellipticity of the linear part from~\eqref{eq:semilinear:ellipticity}
verify the strong monotonicity assumption~\eqref{assump:stronglyMonotone} for the nonlinear operator $\AA$ from~\eqref{eq:semilinear:operator}.
Recall from \cite[Lemma~20]{bbimp2022cost} that the reaction term \(c\) satisfying~\eqref{assump:Caratheodory}
and~\eqref{assump:growth_c} is locally Lipschitz continuous with respect to $\xi$, i.e.,
for \(\vartheta > 0\), there exists a constant \(\widetilde L[\vartheta] > 0\) such that any
\(u, v, w \in \XX\) with \(\max\{\enorm{u}, \enorm{u - v}\} \le \vartheta\) satisfy
\[
    \langle c(u) - c(v), w \rangle
    \leq
    \widetilde L[\vartheta]\,
    \enorm{u - v} \enorm{w}.
\]
This and the boundedness of the linear part from~\eqref{eq:semilinear:boundedness_linear_part}
prove the local Lipschitz continuity~\eqref{assump:locallyLipschitz} for $\AA$.
Hence, the assumptions from subsection~\ref{section:nonlinearModelProblem} are satisfied and
the Browder--Minty theorem guarantees well-posedness of the weak formulation~\eqref{eq:semilinearWeakform}.
The same applies to the discrete problem using the conforming finite-element ansatz space of fixed polynomial degree $p \in \N$
\[
    \XX_\coarse
    \coloneqq
    \mathcal{S}_0^p(\TT_\coarse)
    \coloneqq
    \set{v_\coarse \in H_0^1(\Omega) : \forall \,T \in \TT_\coarse\colon \, v_\coarse|_T \text{ is a polynomial of degree} \le p}
    \subset
    \XX.
\]

The growth condition~\eqref{assump:growth_c} allows to establish
the local Lipschitz continuity~\eqref{assump:derivativeLocallyLipschitz} of the Fréchet derivative $\d\AA\colon \XX \to \mathcal{L}(\XX, \XX'), w \mapsto \d\AA[w]$ with
\begin{equation}\label{eq:semilinearDerivative}
    \d\AA[w] \colon \XX \to \XX',
    \quad
    v \mapsto \dual{\A \nabla v}{\nabla \cdot}_\Omega + \dual{\boldsymbol{b}\,\cdot\nabla v + c'(w)\, v}{\cdot \,}_{\Omega}.
\end{equation}

\begin{lemma}\label{lemma:difference}
    Under the assumptions on the semilinear model problem~\eqref{eq:semilinear:ellipticity},~\eqref{assump:Caratheodory},~\eqref{assump:monotonicity_c}, and~\eqref{assump:growth_c} and for $0< \vartheta$, there exists \(\textup{M}[\vartheta] > 0\)
    such that all \(v, w, z \in \XX\) with \(\max\{\enorm{v}, \enorm{v - w}\} \le \vartheta\) satisfy
    \begin{equation}
        \label{eq:dualCriticalLipschitz}
        \norm{(\d{\AA}[v] - \d{\AA}[w]) z}_{\XX'}
        \le
        \mathrm{M}[\vartheta] \enorm{v - w} \enorm{z}.
    \end{equation}
    The local Lipschitz constant $\mathrm{M}[\vartheta]$ depends only on
    $\vartheta, |\Omega|, d, p, N, R, \alpha_{\textup{min}}$, and \(\alpha_{\textup{max}}\).
\end{lemma}

\begin{proof}
    The proof is subdivided into five steps.

    \emph{Step~1.}
    For arbitrary $2 < t < \infty$, choose $t'' \coloneqq t / (t-2) > 0$ such that $1 = 1/{t''} + 2/t$
    and recall the generalized H\"{o}lder inequality, e.g., from \cite[Section~2.2]{kof1977},
    \begin{equation}
        \label{eq:generalizedHolder}
        \dual{\gamma \varphi}{\psi}
        \leq
        \norm{\gamma}_{L^{t''}(\Omega)} \,
        \norm{\varphi}_{L^t(\Omega)} \,
        \norm{\psi}_{L^t(\Omega)}
        \quad \text{ for all } \gamma \in L^{t''}(\Omega) \text{ and } \varphi, \psi \in L^t(\Omega).
    \end{equation}

    \emph{Step~2.}
    The Sobolev embedding $H^1_0(\Omega) \hook L^r(\Omega)$ is continuous for arbitrary \(1 \leq r < \infty\)
    if \(d \in \{1, 2\}\) and for \(1 \leq r \leq 6\) if \(d = 3\); see, e.g., \cite[Theorem~16.6]{fk1980}.
    In the case $d \in \{1,2\}$, let $2 < t < \infty$ and $N \in \N$ with \(N \geq 2\) be arbitrary.
    In the case $d=3$, Assumption~\eqref{assump:growth_c} asserts that \(N \in \{2,3\}\)
    and the choice $t = 6$ leads to $t'' = t/(t-2) = 3/2$ ensuring \((N-1)t'' \leq 6\).
    Hence, for all cases \(d \in \{1, 2, 3\}\), the Sobolev embedding gives
    \begin{equation}
        \label{eq:semilinear:sobolev}
        \norm{v}_{L^{(N-1)t''}(\Omega)}
        \lesssim
        \norm{v}_{H_0^1(\Omega)}
        \eqsim
        \enorm{v}
        \quad\text{and}\quad
        \norm{v}_{L^{t}(\Omega)}
        \lesssim
        \norm{v}_{H_0^1(\Omega)}
        \eqsim
        \enorm{v}
        \quad \text{ for all } v \in \XX.
    \end{equation}

    \emph{Step~3.}
    Recall from~\cite[Remark~19{\textup{(ii)}}]{bbimp2022cost} that
    Assumptions~\eqref{assump:Caratheodory} and~\eqref{assump:growth_c} imply
    \(c'(v) \in L^{t''}(\Omega)\) for all \(v \in \XX\) with
    \begin{equation}
        \label{eq:estimateDerivatives}
        \norm{c'(v)}_{L^{t''}(\Omega)}
        \le
        (N-1)\, \big(1 + \norm{v}_{L^{(N-1)t''}(\Omega)}^{N-1} \big)
        \eqreff{eq:semilinear:sobolev}\lesssim
        (N-1) \, ( 1 + \enorm{v}^{N-1}).
    \end{equation}
    This allows to apply the generalized Hölder inequality~\eqref{eq:generalizedHolder} and
    the Sobolev embedding~\eqref{eq:semilinear:sobolev} to establish
    \begin{equation}
        \label{eq:Lipschitz}
        \begin{split}
            \norm{(\d{\AA}[v] - \d{\AA}[w]) z}_{\XX'}
            &=
            \sup_{\varphi \in \XX, \enorm{\varphi}=1}
            \dual{[c'(v)- c'(w)] z}{\varphi}
            \\
            &\eqreff*{eq:generalizedHolder}\lesssim
            \sup_{\varphi \in \XX, \enorm{\varphi}=1}
            \norm{c'(v)- c'(w)}_{L^{t''}(\Omega)}\,
            \norm{z}_{L^t(\Omega)} \,
            \norm{\varphi}_{L^t(\Omega)}
            \\
            &\eqreff*{eq:semilinear:sobolev}\lesssim
            \norm{c'(v)- c'(w)}_{L^{t''}(\Omega)}\,
            \enorm{z}.
        \end{split}
    \end{equation}

    \emph{Step~4.}
    Given \(1 \leq t'' < \infty\) and \(N \in \N\) with \(N \geq 2\) as chosen in Step~2,
    for any $0 \le n \le N$, let \(1 \leq t_1, t_2 < \infty\) satisfy
    \begin{equation*}
        \frac{1}{t''}
        =
        \frac{1}{t''} \Big(\frac{N-n}{N-1} + \frac{n-1}{N-1}\Big)
        \eqqcolon
        \frac{1}{t_1} + \frac{1}{t_2}.
    \end{equation*}
    The H\"older inequality for $\varphi, \psi \in \XX$
    and the equalities $t_1 (N-n) = t_2 (n-1) = t''(N-1)$ prove
    \begin{equation}
        \label{eq:semilinear:holder}
        \begin{split}
            \norm{\varphi^{N-n} \, \psi^{n-1}}_{L^{t''}(\Omega)}
            \le
            \norm{\varphi^{N-n}}_{L^{t_1}(\Omega)}\,
            \norm{\psi^{n-1}}_{L^{t_2}(\Omega)}
            &=
            \norm{\varphi^{N-n}}_{L^{t_1}(\Omega)}\,
            \norm{\psi}_{L^{t_2(n-1)}(\Omega)}^{n-1}
            \\
            &=
            \norm{\varphi^{N-n}}_{L^{t_1}(\Omega)}\,
            \norm{\psi}_{L^{t''(N-1)}(\Omega)}^{n-1}
        \end{split}
    \end{equation}

    \emph{Step~5.}
    Let \(v, w \in \XX\) with \(\max\{\enorm{v}, \enorm{v - w}\} \leq \vartheta\).
    The smoothness assumption~\eqref{assump:Caratheodory} and a Taylor expansion in $w \in \XX$ show
    \begin{equation}
        \label{eq2:primal:taylorexp:dual}
        c'(v)
        =
        \sum_{n=1}^{N-1} c^{(n)}(w) \,
        \frac{(v -w)^{n-1}}{(n-1)!}
        +
        \frac{(v - w)^{N-1}}{(N-2)!} \,
        \int_0^1 (1 - s)^{N-2} \, c^{(N)}\big(w + s \, (v - w) \big) \d{s}.
    \end{equation}
    The triangle inequality and the growth condition~\eqref{assump:growth_c} yield
    \begin{equation}
        \label{eq:taylorEstimate}
        \norm{c'(v )-c'(w)}_{L^{t''}(\Omega)}
        \eqreff{assump:growth_c}\lesssim \,
        \sum_{n=2}^{N-1} \norm{c^{(n)}(w)\, (v -w)^{n-1}}_{L^{t''}(\Omega)}
        +
        \norm{(v -w)^{N-1}}_{L^{t''}(\Omega)}.
        \end{equation}
    The application of the H\"older inequality~\eqref{eq:semilinear:holder}
    with $\psi = v - w$ and $\varphi^{N-n} = c^{(n)}(w)$ reveals
    \begin{equation*}
        \norm{c'(v) - c'(w)}_{L^{t''}(\Omega)}
        \eqreff{eq:semilinear:holder}\lesssim
        \sum_{n=2}^{N-1}
        \norm{c^{(n)}(w)}_{L^{t_1}(\Omega)}
        \norm{v - w}_{L^{t''(N-1)}(\Omega)}^{n-1}
        +
        \norm{v -w}_{L^{t''(N-1)}(\Omega)}^{N-1}.
    \end{equation*}
    For all $n \in \{2, \dots, N-1\}$,
    choose $t_1$ such that $t_1 (N-n) = t''(N-1)$. Recall from~\cite[Equation~(67)]{bbimp2022cost} that
    the Sobolev embedding~\eqref{eq:semilinear:sobolev} and the definition of $t_1$ yield
    \begin{equation*}
        \norm{c^{(n)}(w)}_{L^{t_1}(\Omega)}
        \lesssim
        (N-n) \, (1 + \enorm{w}^{N-n})
        \le
        (N-n) (1 + \vartheta^{N-n}).
    \end{equation*}
    The combination of the two previous displayed formulas
    with~\eqref{eq:semilinear:sobolev} reads
    \begin{align*}
        \norm{c'(v) - c'(w)}_{L^{t''}(\Omega)}
        &\lesssim
        \Big(\sum_{n=2}^{N} (N-n)\,(1 + \vartheta^{N-n})\, \vartheta^{n-2} \Big)
        \, \enorm{v - w}
        \eqqcolon
        \mathrm{M}[\vartheta]\, \enorm{v -w}.
    \end{align*}
    Norm equivalence and~\eqref{eq:Lipschitz} finish the proof of local Lipschitz continuity.
\end{proof}

\subsection{A~posteriori error estimator}

The local contributions to the residual error estimator read
\begin{equation}
    \label{eq:estimator:primal}
    \begin{aligned}
        \eta_\coarse(T, v_\coarse)^2
        &\coloneqq
        h_T^2 \,\norm{f + \div(\A \, \nabla v_\coarse - \f) - \boldsymbol{b}\cdot \nabla v_\coarse - c(v_\coarse)}_{L^2(T)}^2
        \\
        &\phantom{{}\coloneqq{}}
        +
        h_T \, \norm{\jump{(\A \, \nabla v_\coarse - \f ) \, \cdot \, \n}}_{L^2(\partial T \cap \Omega)}^2,
    \end{aligned}
\end{equation}
where $\jump{\,\cdot\,}$ denotes the jump across edges (for $d=2$) resp.\ faces (for $d=3$) and
$\n$ denotes the corresponding unit normal vector.
For $d=1$, these jumps vanish, i.e., $\jump{\,\cdot\,} = 0$.
For the case \(\boldsymbol{b} = 0\), it has been proven in~\cite[Proposition~15]{bbimp2022}
that \(\eta_\coarse(\cdot)\) satisfies the axioms~\eqref{axiom:stability}--\eqref{axiom:discreteReliability}
from Section~\ref{subsection:axioms}.
For general convection coefficients \(\boldsymbol{b} \in [W^{1,\infty}(\Omega)]^d\),
the proofs apply verbatim.
It remains to verify the quasi-orthogonality~\eqref{axiom:quasiorthogonality}.
We distinguish the symmetric and the non-symmetric case.

\subsubsection*{Case 1: Energy minimization}

If $\boldsymbol{b} = \boldsymbol{0}$, the semilinear model problem~\eqref{eq:semilinearWeakform}
can be characterized as an energy minimization problem for the energy functional \(\EE\colon \XX \to \R\) with
\begin{equation}\label{eq:semilinearEnergy}
    \EE(v)
    \coloneqq
    \frac{1}{2} \int_\Omega |\boldsymbol{A}^{1/2}\nabla v|^2 \d x
    +
    \int_\Omega \int_0^{v(x)} c(s) \d s \d x
    -
    \int_\Omega fv \d x
    -
    \int_\Omega \boldsymbol{f} \cdot \nabla v \d x.
\end{equation}
This energy is well-defined by the assumptions on the nonlinearity; cf~\cite[Section~3.6]{bbimp2022cost}.
It is well-known (see, e.g.,~\cite[Lemma~5.1]{ghps2018})
that the energy difference can be employed as a measure for the discretization error for nested spaces $\XX_\coarse \subseteq \XX_\fine$ (where also $\XX_\fine = \XX$ is admissible) and $\Bup_0$ from~\eqref{eq:exact:bounded}:
\begin{equation*}
    \frac{\alpha}{2}\,
    \enorm{u_\coarse^\exact - u_\fine^\exact}^2
    \le
    \EE(u_\coarse^\exact) - \EE(u_{\fine}^\exact)
    \le
    \frac{\Lup[\textup{B}_0]}{2}\,
    \enorm{u_\coarse^\exact - u_\fine^\exact}^2.
\end{equation*}
This, the telescoping nature of the energy differences, the straight-forward equality, for any refinement \(\TT_\fine\) of \(\TT_\coarse\),
\begin{equation*}
    \EE(u_\coarse^\exact) - \EE(u^\exact)
    =
    \bigl[ \EE(u_\coarse^\exact) - \EE(u_\fine^\exact) \bigr]
    +
    \bigl[ \EE(u_\fine^\exact) - \EE(u^\exact) \bigr],
\end{equation*}
reliability~\eqref{axiom:reliability}, and estimator monotonicity~\eqref{eq:quasi-monotonicity}
imply quasi-orthogonality~\eqref{eq:quasiorthogonality} with \(\varepsilon = 0\)
and \(C_{\mathrm{orth}} = \Crel^2 \, (1+\Cmon^2) \, \Lup[\Bup_0] / \alpha\).

\subsubsection*{Case 2: With convection}
If $\boldsymbol{b} \neq \boldsymbol{0}$,
the quasi-orthogonality follows from a compact-perturbation argument.
Let $\XX_\infty \coloneqq\mathrm{closure}\big( \bigcup_{\ell=0}^\lmax \XX_\ell \big)$
be the closure with respect to the norm \(\enorm{\,\cdot\,}\).
Let $u_\infty^\exact \in \XX_\infty$ denote the corresponding Galerkin solution.

First, recall that reliability transfers to the discrete limit space $\XX_\infty$ in that
there exists $C_{\textup{rel}}' > 0$ such that
\begin{equation}
    \label{eq:reliabilityLimitSpace}
    \enorm{u_\infty^\exact - u_\ell^\exact}
    \le
    C_{\textup{rel}}' \, \eta_\ell(u_\ell^\exact)
    \quad \text{ for all } 0 \le \ell \le \lmax.
\end{equation}

Second, the arguments from~\cite[Lemma~29]{bbimp2022} applied to the discrete limit space reveal that,
for all $0 < \sigma < 1$, there exists an index $\ell_0 = \ell_0[\sigma] \in \N$ such that
the quasi-Pythagorean identity holds
\begin{equation}
    \label{eq:quasiPythagorean}
    \enorm{u_\infty^\exact - u_{\ell+1}^\exact}^2
    +
    \enorm{u_{\ell+1}^\exact - u_\ell^\exact}^2
    \le
    \frac{1}{1-\sigma} \, \enorm{u_\infty^\exact - u_\ell^\exact}^2
    \quad \text{ for all } \ell_0 \le \ell < \lmax.
\end{equation}
The choice $0 < \sigma \coloneqq \varepsilon / \bigl(\varepsilon + (\Crel')^2\bigr) < 1$
and, hence, $\varepsilon = (C_{\textup{rel}}')^2 \sigma/(1-\sigma)$ together with reliability~\eqref{eq:reliabilityLimitSpace} show
\begin{equation}
    \label{eq:QPImpliesQOStep1}
    \frac{1}{1-\sigma} \,
    \enorm{u_\infty^\exact - u_\ell^\exact}^2
    -
    \varepsilon\, \eta_\ell(u_\ell^\exact)^2
    \eqreff{eq:reliabilityLimitSpace}\le
    \frac{1}{1-\sigma} \,
    \enorm{u_\infty^\exact - u_\ell^\exact}^2
    -
    \frac{\sigma}{1-\sigma} \,
    \enorm{u_\infty^\exact - u_\ell^\exact}^2
    =
    \enorm{u_\infty^\exact - u_\ell^\exact}^2.
\end{equation}
Therefore, a combination of the estimates~\eqref{eq:quasiPythagorean}--\eqref{eq:QPImpliesQOStep1} proves,
for all $\ell_0 \le \ell < \lmax$,
\begin{equation}\label{eq:QPImpliesQOStep2}
    \begin{aligned}
        \enorm{u_{\ell+1}^\exact - u_\ell^\exact}^2 - \varepsilon \eta_\ell(u_\ell^\exact)^2
        \, \,
        &\eqreff*{eq:quasiPythagorean}\le \, \,
        \frac{1}{1-\sigma} \,
        \enorm{u_\infty^\exact - u_\ell^\exact}^2
        -
        \enorm{u_\infty^\exact - u_{\ell+1}^\exact}^2
        -
        \varepsilon\, \eta_\ell(u_\ell^\exact)^2
        \\
        &\eqreff*{eq:QPImpliesQOStep1}\le \, \,
        \enorm{u_\infty^\exact - u_\ell^\exact}^2 - \enorm{u_\infty^\exact - u_{\ell+1}^\exact}^2.
    \end{aligned}
\end{equation}
Hence, the telescoping sum 
concludes the proof of~\eqref{eq:quasiorthogonality} with, for all \(\ell_0 \le \ell' < \elll-1\),
\begin{equation}
    \label{eq:summability_remainder}
    \begin{split}
        \sum_{\ell = \ell'}^{\elll-1}
        \big[ \enorm{u_{\ell+1}^\exact - u_\ell^\exact}^2 - \varepsilon \eta_\ell(u_\ell^\exact)^2 \big] \, \,
        &\eqreff*{eq:QPImpliesQOStep2}
        \le  \, \, \,
        \sum_{\ell = \ell'}^{\elll-1}
        \big[
        \enorm{u_\infty^\exact - u_\ell^\exact}^2 - \enorm{u_\infty^\exact - u_{\ell+1}^\exact}^2
        \big]
        \le
        \enorm{u_\infty^\exact - u_{\ell'}^\exact}^2.
         \end{split}
\end{equation}
Since $\XX_k \subseteq \XX_\infty$, the Céa lemma~\eqref{eq:cea} with $\Bup_0$ from~\eqref{eq:exact:bounded} proves
\[
\enorm{u_\infty^\exact-u_\ell^\exact}
    \eqreff{eq:cea}\le
    (1 + L[\Bup_0]/\alpha) \, \enorm{u^\exact - u_\ell^\exact}
    \eqreff{axiom:reliability}\le
    (1 + L[\Bup_0]^2/\alpha^2) \, \Crel \, \eta_\ell(u_\ell^\exact)
    \quad \text{ for all } \ell_0 \le \ell' < \elll-1.
\]
This concludes the proof of~\eqref{eq:quasiorthogonality} with \(C_{\mathrm{orth}} = (1 + L[\Bup_0]/\alpha)^2 \, \Crel^2\).

\begin{remark}
    In Case~2,
    the quasi-Pythagorean estimate~\eqref{eq:quasiPythagorean} holds only asymptotically for $\ell \ge \ell_0$ sufficiently large.
    However, all main results on full R-linear convergence (Theorem~\ref{theorem:fullRLinearConvergence} above)
    and optimal complexity (Theorem~\ref{theorem:optimalRates} above)
    generalize accordingly and hold for $\ell_0 \le \ell \le \lmax$.
\end{remark}

\subsection{Numerical experiments}
\label{sec:semilinear:numerics}

This section is devoted to the empirical investigation of Algorithm~\ref{algorithm:AILFEM}.
The benchmark problem considers the L-shaped domain $\Omega = (-1,1)^2 \setminus [0,1)^2$ with constant right-hand side $f \equiv 2$,
a constant convection $\boldsymbol{b} \in \R^2$, and
a nonlinear reaction term $c(u) = \sum_{n=0}^N (40u)^n/n! \eqqcolon \exp_N(40u) \approx \exp(40u)$ with $N=11$.
Let the unknown exact solution $u^\exact \in H_0^1(\Omega)$ satisfy
\begin{equation}
    \label{eq:semilinear:experiment:weakform}
    \dual{\nabla u^\exact}{\nabla v}_\Omega
    +
    \dual{\boldsymbol{b}\cdot \nabla u^\exact
    +
    c(u^\exact)}{v}_\Omega
    =
    \dual{2}{v}_\Omega
    \quad \text{ for all } v \in H_0^1(\Omega).
\end{equation}

\subsubsection*{Case 1: Energy minimization}
In the first experiment, the choice of $\boldsymbol{b} = \boldsymbol{0}$ allows to interpret~\eqref{eq:semilinear:experiment:weakform}
as the first-order optimality condition of the minimization of the convex energy functional
\[
    \mathcal{E}(u)
    =
    \frac{1}{2} \int_\Omega |\nabla u|^2 \d x
    +
    \int_\Omega \exp_{N+1}(40u) / 40 \d x
    -
    \int_\Omega 2u \d x.
\]
Figure~\ref{fig:semilinear:mesh_solution} shows the adaptively refined mesh and a piecewise quadratic approximation to the solution \(u^\exact\).
The mesh plot in Figure~\ref{fig:semilinear:mesh} exhibits increased refinement at the reentrant corner
as well as moderate refinement at the boundary layer due to the comparably large gradient in this region.
\begin{figure}
    \centering
    \begin{subfigure}{0.47\textwidth}
        \begin{tikzpicture}
    \begin{axis}[%
        axis equal image,%
        width=8.0cm,%
        xmin=-1.15, xmax=1.15,%
        ymin=-1.15, ymax=1.15,%
        font=\footnotesize%
    ]
        \addplot graphics [xmin=-1, xmax=1, ymin=-1, ymax=1]
        {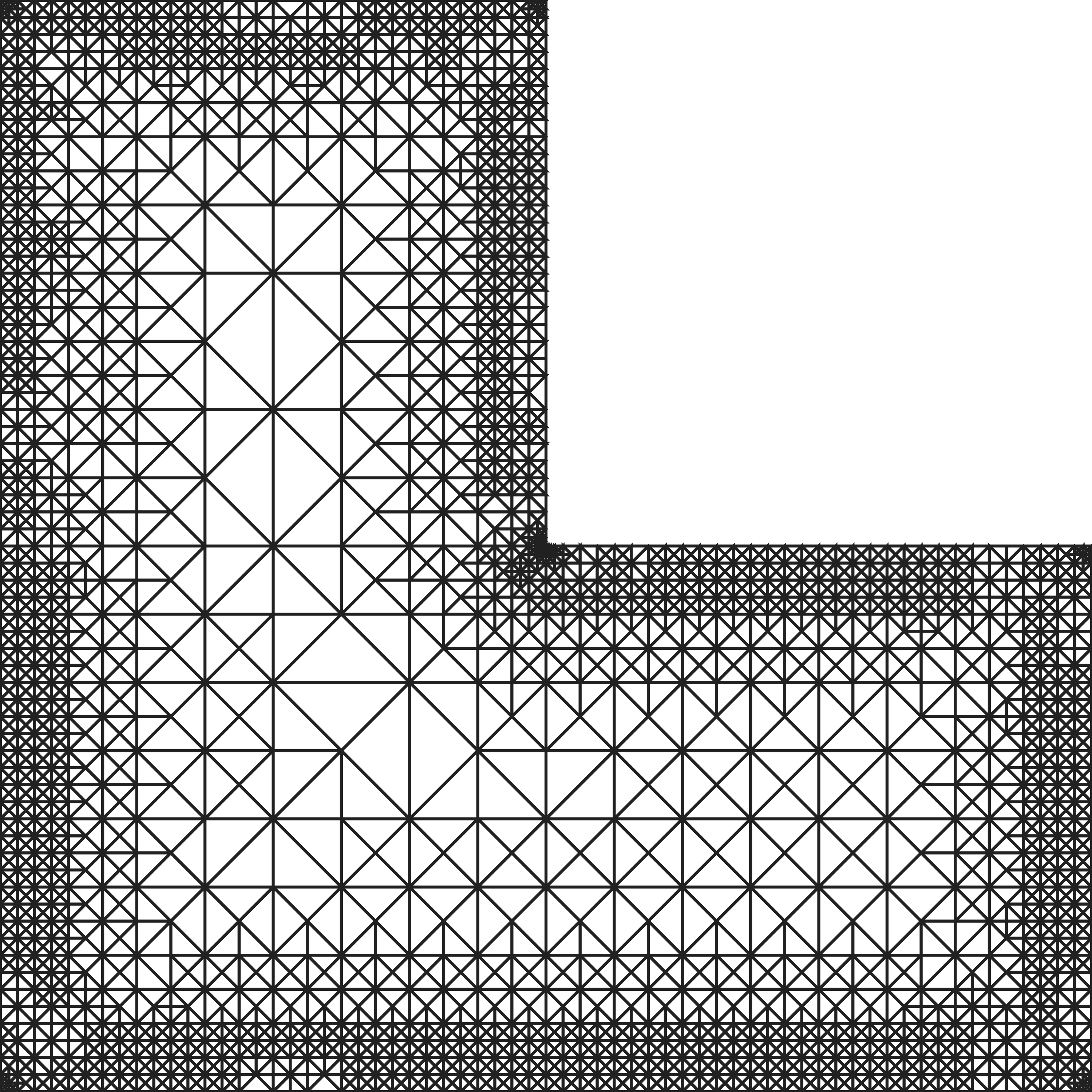};
    \end{axis}
\end{tikzpicture}
        \caption{Mesh \(\TT_{17}\) (4\,898 triangles)}
        \label{fig:semilinear:mesh}
    \end{subfigure}
    \hfill
    \begin{subfigure}{0.47\textwidth}
        \begin{tikzpicture}
    \pgfplotsset{/pgf/number format/fixed}
    \begin{axis}[%
        width=79mm,%
        xmin=-1.1, xmax=1.1,%
        ymin=-1.1, ymax=1.1,%
        zmin=-0.0007, zmax=0.008,%
        font=\footnotesize,%
    ]
        \addplot3 graphics [%
            points={%
                (-1,-1,0)              => (0.1  ,327.4-239.6)
                (1,-1,0)               => (220.1,327.4-327.4)
                (0,1,0)                => (372.2,327.4-209.9)
                (-0.688,0.688,0.00675) => (255.5,327.4-0.0)
            }%
            ]
            {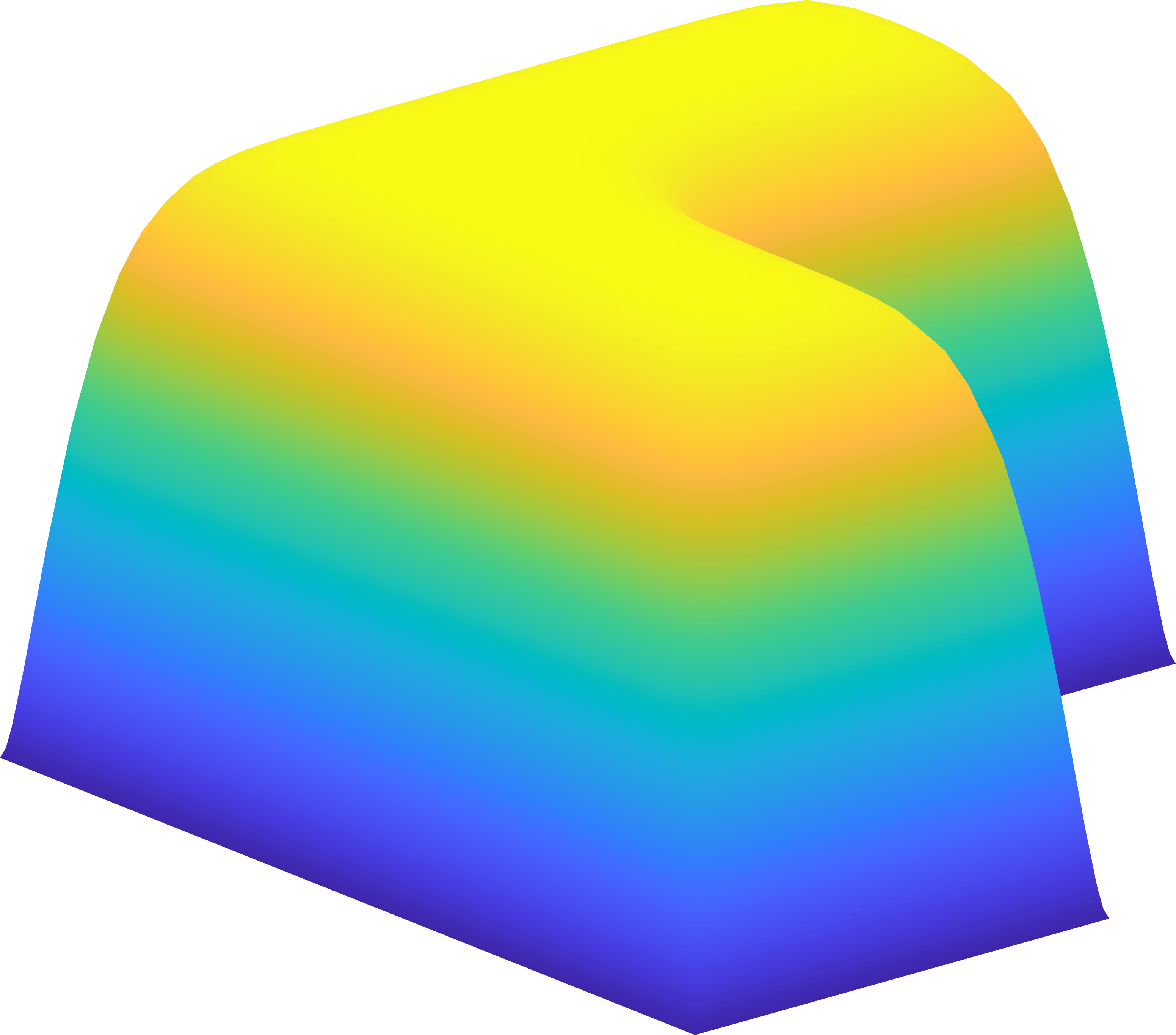};
    \end{axis}
\end{tikzpicture}
        \caption{Solution \(u_\ell^\kk \in \mathcal{S}^p_0(\TT_\ell)\) for \(\ell = 17\)}
        \label{fig:semilinear:solution}
    \end{subfigure}
    \caption{%
        Case 1 from subsection~\ref{sec:semilinear:numerics}:
        Plot of adaptively refined mesh (\textsc{a}) and corresponding discrete solution (\textsc{b})
        generated by Algorithm~\ref{algorithm:AILFEM} with \(p = 2\), \(\theta = 0.3\), \(\protect\lambdalin = 0.1\), and \(\kmin = 1\).
    }
    \label{fig:semilinear:mesh_solution}
\end{figure}
The discretization error estimator \(\eta_\ell(u_\ell^\kk)\) of the final iterates \(u_\ell^\kk\)
converges with the optimal rate for the corresponding polynomial degree as displayed in Figure~\ref{fig:semilinear:convergence:estimator}.
The convergence behavior of the Newton iteration on each level is investigated using the reduction factor
\begin{equation}
    \label{eq:semilinear:reduction_factor}
    r_\ell^k
    \coloneqq
    \norm{F - \AA u_\ell^k}_{\XX_\ell'} \big/ \norm{F - \AA u_\ell^{k-1}}_{\XX_\ell'}
    \quad \text{ for } (\ell, k) \in \QQ \text{ with } k \ge 1.
\end{equation}
Figure~\ref{fig:semilinear:convergence:reduction} shows that the reduction factor is uniformly bounded away from 1.
Due to the strong nonlinearity, there is a relatively large number of iterations on the coarsest level.
As soon as the linearization error is sufficiently reduced, the algorithm continues with mesh refinement and one Newton iteration per level.
The steadily decreasing reduction factors confirm the quadratic convergence of the Newton iteration
in the final phase of Algorithm~\ref{algorithm:AILFEM}.
\begin{figure}
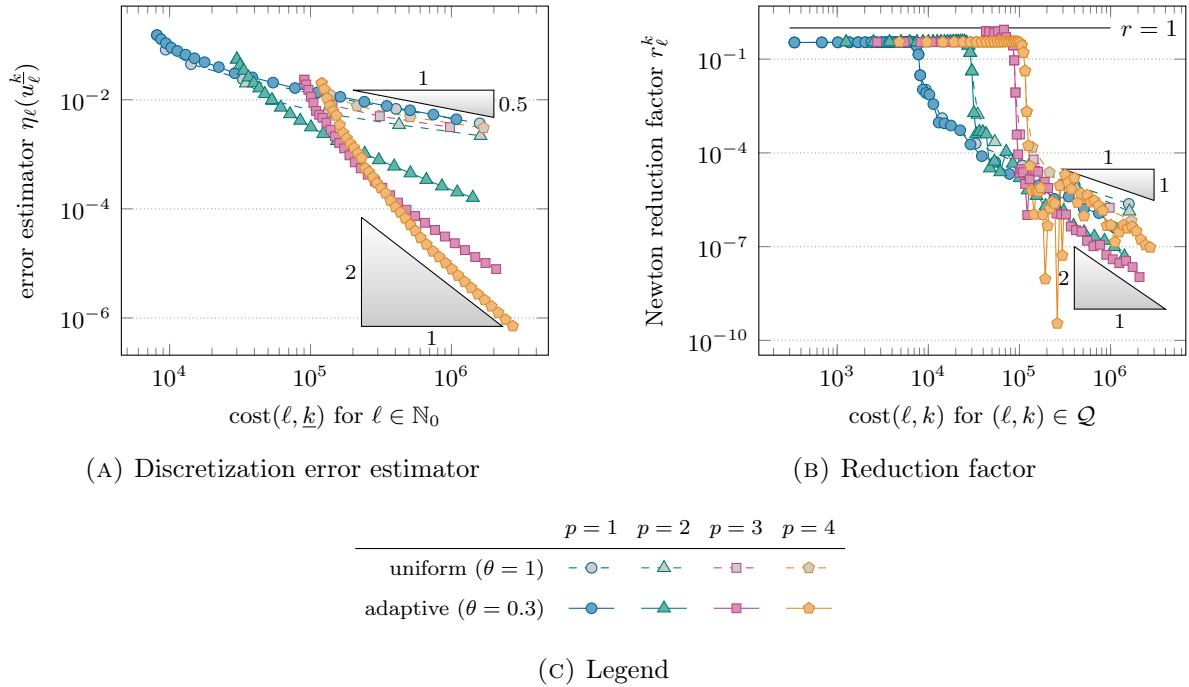

    \centering
    \begin{subfigure}{0.47\textwidth}
        \input{figures/Fig2a_Semilinear_estimator.tex}
        \caption{Discretization error estimator}
        \label{fig:semilinear:convergence:estimator}
    \end{subfigure}
    \hfill
    \begin{subfigure}{0.47\textwidth}
        \input{figures/Fig2b_Semilinear_reduction.tex}
        \caption{Reduction factor}
        \label{fig:semilinear:convergence:reduction}
    \end{subfigure}
    \medskip

    \begin{subfigure}{0.47\textwidth}
        \centering
        \begin{tikzpicture}[>=stealth]
    %
    %
    \colorlet{col1}{TUblue}
    \colorlet{col2}{TUgreen}
    \colorlet{col3}{TUmagenta}
    \colorlet{col4}{TUyellow}
    \colorlet{col5}{purple}
    \colorlet{col6}{green}
    \pgfplotsset{%
        linedefault/.style = {%
            mark = *,%
            mark size = 2pt,%
            every mark/.append style = {solid},%
            gray,%
            every mark/.append style = {fill = gray!60!white}%
        },%
        line1/.style = {%
            linedefault,%
            col1,%
            every mark/.append style = {fill = col1!60!white}%
        },%
        line2/.style = {%
            linedefault,%
            mark = triangle*,%
            mark size = 2.75pt,%
            col2,%
            every mark/.append style = {fill = col2!60!white}%
        },%
        line3/.style = {%
            linedefault,%
            mark = square*,%
            mark size = 1.66pt,%
            col3,%
            every mark/.append style = {fill = col3!60!white}%
        },%
        line4/.style = {%
            linedefault,%
            mark = pentagon*,%
            mark size = 2.2pt,%
            col4,%
            every mark/.append style = {fill = col4!60!white}%
        },%
        line5/.style = {%
            linedefault,%
            mark = diamond*,%
            mark size = 2.75pt,%
            col5,%
            every mark/.append style = {fill = col5!60!white}%
        },%
        line6/.style = {%
            linedefault,%
            mark = halfsquare*,%
            mark size = 1.66pt,%
            col6,%
            every mark/.append style = {fill = col6!60!white}%
        },%
        minorline/.style = {%
            dashed,%
            every mark/.append style = {fill = black!20!white}%
        },%
        majorline/.style = {%
            solid%
        }%
    }

    \matrix [
        matrix of nodes,
        anchor = south,
        font = \scriptsize,
        column 1/.style={anchor=base east},
    ] at (0,0) {
        & \(p=1\)
        & \(p=2\)
        & \(p=3\)
        & \(p=4\)
        \\
        \hline
        \\
        uniform (\(\theta = 1\))
        & \ref*{leg:est:unif:1}
        & \ref*{leg:est:unif:2}
        & \ref*{leg:est:unif:3}
        & \ref*{leg:est:unif:4}
        \\
        adaptive (\(\theta = 0.3\))
        & \ref*{leg:est:adap:1}
        & \ref*{leg:est:adap:2}
        & \ref*{leg:est:adap:3}
        & \ref*{leg:est:adap:4}
        \\
    };
\end{tikzpicture}
        \caption{Legend}
        \label{fig:semilinear:legend}
    \end{subfigure}
    \caption{%
        Case 1 from subsection~\protect\ref{sec:semilinear:numerics}:
        Convergence history plots for the discretization error estimator (\textsc{a}) from~\protect\eqref{eq:estimator:primal}
        and the reduction factor (\textsc{b}) from~\protect\eqref{eq:semilinear:reduction_factor}
        for adaptive and uniform mesh refinements with various polynomial degrees \(p\).
        The remaining parameters read \(\protect\lambdalin = 0.1\) and \(\kmin = 1\).
        The legend (\textsc{c}) applies to both plots.
    }
    \label{fig:semilinear:convergence}
\end{figure}

\subsubsection*{Case 2: With convection}
The second experiment investigates the strong constant convection \(\boldsymbol{b} = [-50, 0]^\top\).
In this case, there exists no corresponding energy functional.
The adaptively refined mesh and the piecewise quadratic approximation to the solution \(u^\exact\)
are displayed in Figure~\ref{fig:semilinear_convective:mesh_solution}.
In addition to the increased refinement areas from Figure~\ref{fig:semilinear:mesh},
the mesh in Figure~\ref{fig:semilinear_convective:mesh} exhibits further refinement at the left boundary layer caused by the strong convection.
\begin{figure}
    \centering
    \begin{subfigure}{0.47\textwidth}
        \begin{tikzpicture}
    \begin{axis}[%
        axis equal image,%
        width=8.5cm,%
        xmin=-1.15, xmax=1.15,%
        ymin=-1.15, ymax=1.15,%
        font=\footnotesize%
    ]
        \addplot graphics [xmin=-1, xmax=1, ymin=-1, ymax=1]
        {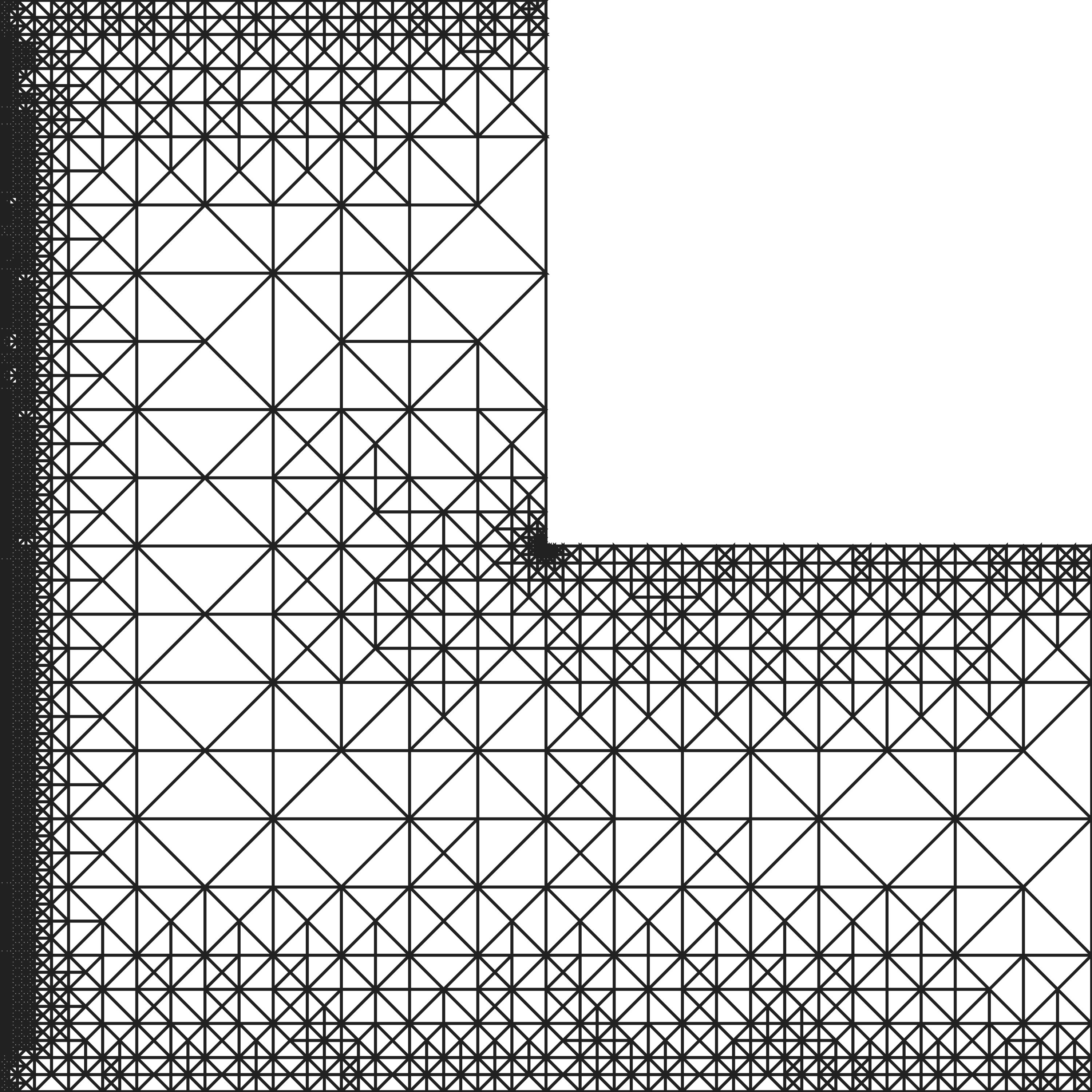};
    \end{axis}
\end{tikzpicture}
        \caption{Mesh \(\TT_{20}\) (5\,248 triangles)}
        \label{fig:semilinear_convective:mesh}
    \end{subfigure}
    \hfill
    \begin{subfigure}{0.47\textwidth}
        \begin{tikzpicture}
    \pgfplotsset{/pgf/number format/fixed}
    \begin{axis}[%
        width=77mm,%
        xmin=-1.1, xmax=1.1,%
        ymin=-1.1, ymax=1.1,%
        zmin=-0.0007, zmax=0.008,%
        font=\footnotesize,%
    ]
        \addplot3 graphics [%
            points={%
                (1,0,0)                 => (0    ,323.2-250.6)
                (-1,1,0)                => (373.7,323.2-323.2)
                (-1,-1,0)               => (446  ,323.2-215.8)
                (-0.906,-0.688,0.00679) => (416.6,323.2-5.4)
            }%
            ]
            {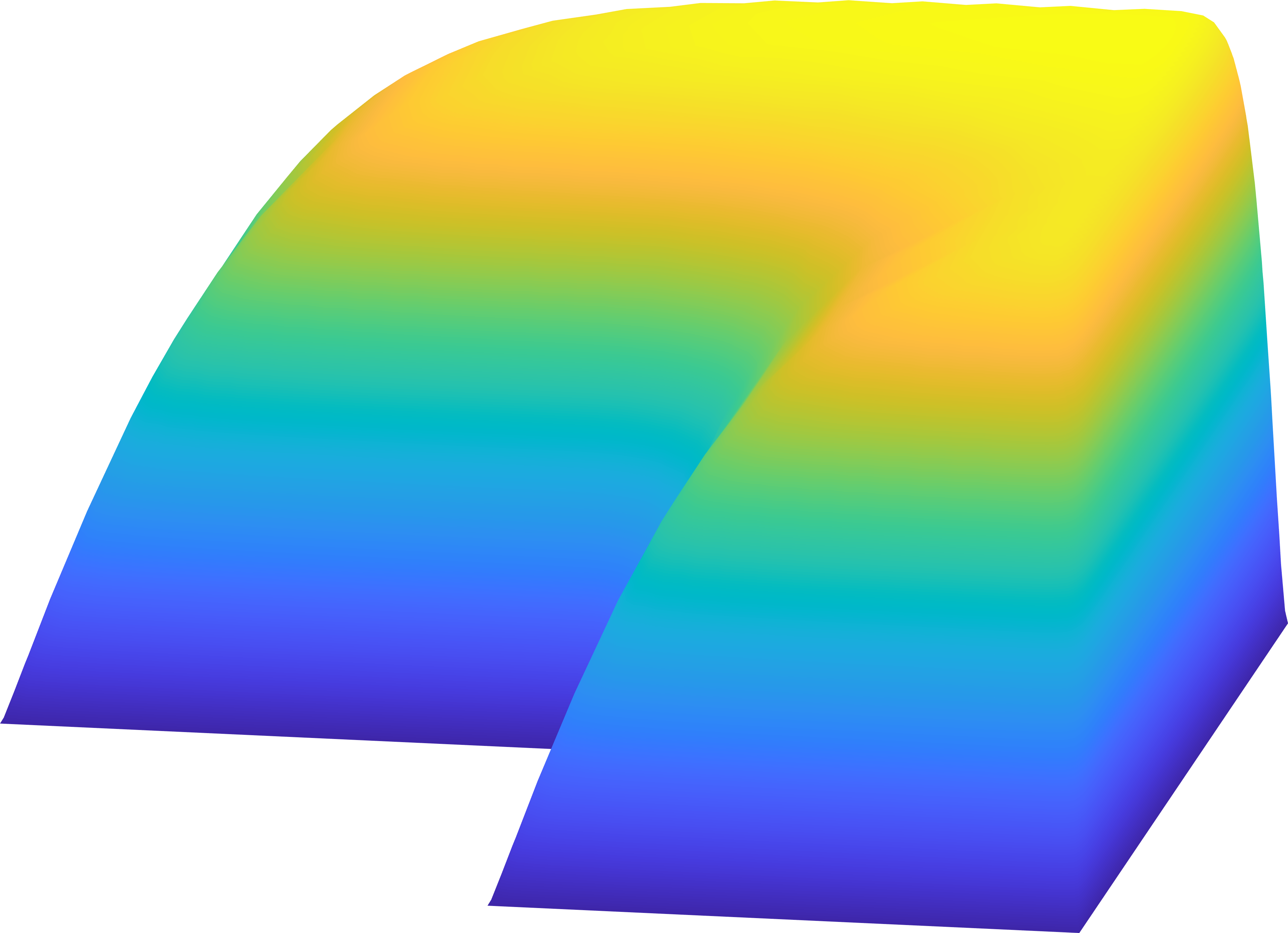};
    \end{axis}
\end{tikzpicture}
        \caption{Solution \(u_\ell^\kk \in \mathcal{S}^p_0(\TT_\ell)\) for \(\ell = 20\)}
        \label{fig:semilinear_convective:solution}
    \end{subfigure}
    \caption{%
        Case 2 from subsection~\ref{sec:semilinear:numerics}:
        Plot of adaptively refined mesh (\textsc{a}) and corresponding discrete solution (\textsc{b})
        generated by Algorithm~\ref{algorithm:AILFEM} with \(p = 2\), \(\theta = 0.3\), \(\protect\lambdalin = 0.1\), and \(\kmin = 1\).
    }
    \label{fig:semilinear_convective:mesh_solution}
\end{figure}
The convergence plots in Figure~\ref{fig:semilinear_convective:convergence} confirm the results from Case~1
and illustrate that the Algorithm~\ref{algorithm:AILFEM} is capable of solving nonsymmetric problems as well.
\begin{figure}
    \centering
    \begin{subfigure}{0.47\textwidth}
        \input{figures/Fig5a_SemilinearConvective_estimator.tex}
        \caption{Discretization error estimator}
    \end{subfigure}
    \hfill
    \begin{subfigure}{0.47\textwidth}
        \input{figures/Fig5b_SemilinearConvective_reduction.tex}
        \caption{Reduction factor}
    \end{subfigure}
    \caption{%
        Case 2 from subsection~\ref{sec:semilinear:numerics}:
        Convergence history plots for the discretization error estimator (\textsc{a}) from~\eqref{eq:estimator:primal}
        and the reduction factor (\textsc{b}) from~\eqref{eq:semilinear:reduction_factor}
        for adaptive and uniform mesh refinements with various polynomial degrees \(p\).
        The remaining parameters read \(\protect\lambdalin = 0.1\) and \(\protect\kmin = 1\).
        The legend from Figure~\ref{fig:semilinear:legend} applies to both plots.
    }
    \label{fig:semilinear_convective:convergence}
\end{figure}



{
    \renewcommand*{\bibfont}{\footnotesize}
    \sloppy
    \printbibliography

@article{bfmps2025,
    AUTHOR    = {Bringmann, Philipp and Feischl, Michael and Mira\c{c}i, Ani and Praetorius, Dirk and Streitberger, Julian},
    TITLE     = {On full linear convergence and optimal complexity of adaptive {FEM} with inexact solver},
    JOURNAL   = {Comput. Math. Appl.},
    FJOURNAL  = {Computers \& Mathematics with Applications. An International Journal},
    VOLUME    = {180},
    YEAR      = {2025},
    PAGES     = {102--129},
    ISSN      = {0898-1221,1873-7668},
    DOI       = {10.1016/j.camwa.2024.12.013},
}

@article {mps2024,
    AUTHOR     = {Mira\c{c}i, Ani and Praetorius, Dirk and Streitberger, Julian},
    TITLE      = {Unconditional full linear convergence and optimal complexity of adaptive iteratively linearized FEM for nonlinear PDEs},
    YEAR       = {2025},
    JOURNAL    = {Math. Comp., \textup{published online first}},
    DOI        = {10.1090/mcom/4135},
}

@article {cw2017,
    AUTHOR = {Congreve, Scott and Wihler, Thomas P.},
     TITLE = {Iterative {G}alerkin discretizations for strongly monotone
              problems},
   JOURNAL = {J. Comput. Appl. Math.},
  FJOURNAL = {Journal of Computational and Applied Mathematics},
    VOLUME = {311},
      YEAR = {2017},
     PAGES = {457--472},
      ISSN = {0377-0427,1879-1778},
   MRCLASS = {65N30 (65N50)},
  MRNUMBER = {3552717},
MRREVIEWER = {C.\ Ilioi},
       DOI = {10.1016/j.cam.2016.08.014},
       URL = {https://doi.org/10.1016/j.cam.2016.08.014},
}

@article {dk2008,
    AUTHOR = {Diening, Lars and Kreuzer, Christian},
     TITLE = {Linear convergence of an adaptive finite element method for
              the {$p$}-{L}aplacian equation},
   JOURNAL = {SIAM J. Numer. Anal.},
  FJOURNAL = {SIAM Journal on Numerical Analysis},
    VOLUME = {46},
      YEAR = {2008},
    NUMBER = {2},
     PAGES = {614--638},
      ISSN = {0036-1429,1095-7170},
   MRCLASS = {65N30 (35J60 35J70)},
  MRNUMBER = {2383205},
MRREVIEWER = {Olivier\ Besson},
       DOI = {10.1137/070681508},
       URL = {https://doi.org/10.1137/070681508},
}

@article {v2002,
    AUTHOR = {Veeser, Andreas},
     TITLE = {Convergent adaptive finite elements for the nonlinear
              {L}aplacian},
   JOURNAL = {Numer. Math.},
  FJOURNAL = {Numerische Mathematik},
    VOLUME = {92},
      YEAR = {2002},
    NUMBER = {4},
     PAGES = {743--770},
      ISSN = {0029-599X,0945-3245},
   MRCLASS = {65N30 (65N12)},
  MRNUMBER = {1935808},
       DOI = {10.1007/s002110100377},
       URL = {https://doi.org/10.1007/s002110100377},
}

@article {bc2008,
    AUTHOR = {Bartels, S\"oren and Carstensen, Carsten},
     TITLE = {A convergent adaptive finite element method for an optimal
              design problem},
   JOURNAL = {Numer. Math.},
  FJOURNAL = {Numerische Mathematik},
    VOLUME = {108},
      YEAR = {2008},
    NUMBER = {3},
     PAGES = {359--385},
      ISSN = {0029-599X,0945-3245},
   MRCLASS = {65N30 (74S05)},
  MRNUMBER = {2365822},
MRREVIEWER = {Christoph\ Ortner},
       DOI = {10.1007/s00211-007-0122-x},
       URL = {https://doi.org/10.1007/s00211-007-0122-x},
}

@article {c2008,
    AUTHOR = {Carstensen, Carsten},
     TITLE = {Convergence of an adaptive {FEM} for a class of degenerate
              convex minimization problems},
   JOURNAL = {IMA J. Numer. Anal.},
  FJOURNAL = {IMA Journal of Numerical Analysis},
    VOLUME = {28},
      YEAR = {2008},
    NUMBER = {3},
     PAGES = {423--439},
      ISSN = {0272-4979,1464-3642},
   MRCLASS = {65N30 (65N12 65N50)},
  MRNUMBER = {2433207},
MRREVIEWER = {S\"oren\ Bartels},
       DOI = {10.1093/imanum/drm034},
       URL = {https://doi.org/10.1093/imanum/drm034},
}

@article{hw2020:ailfem,
	AUTHOR = {Heid, Pascal and Wihler, Thomas P.},
	TITLE = {On the convergence of adaptive iterative linearized {G}alerkin
	methods},
	JOURNAL = {Calcolo},
	FJOURNAL = {Calcolo. A Quarterly on Numerical Analysis and Theory of
	Computation},
	VOLUME = {57},
	YEAR = {2020},
	NUMBER = {3},
%	PAGES = {\#24},
	ISSN = {0008-0624},
	MRCLASS = {65N30 (35J62 47H05 47H10 47J25 49M15 65N12 65N50)},
	MRNUMBER = {4131951},
	MRREVIEWER = {Riccardo Sacco},
	DOI = {10.1007/s10092-020-00368-4},
%	URL = {https://doi.org/10.1007/s10092-020-00368-4},
}

@article {hw2020:convergence,
	AUTHOR = {Heid, Pascal and Wihler, Thomas P.},
	TITLE = {Adaptive iterative linearization {G}alerkin methods for
	nonlinear problems},
	JOURNAL = {Math. Comp.},
	FJOURNAL = {Mathematics of Computation},
	VOLUME = {89},
	YEAR = {2020},
	NUMBER = {326},
	PAGES = {2707--2734},
	ISSN = {0025-5718},
	MRCLASS = {65N30 (47H05 47H10 47J25 49M15 65J15 65N12 65N50)},
	MRNUMBER = {4136544},
	MRREVIEWER = {Huai Zhang},
	DOI = {10.1090/mcom/3545},
%	URL = {https://doi.org/10.1090/mcom/3545},
}

@article {amw2017,
    AUTHOR = {Amrein, Mario and Melenk, Jens Markus and Wihler, Thomas P.},
     TITLE = {An {$hp$}-adaptive {N}ewton-{G}alerkin finite element
              procedure for semilinear boundary value problems},
   JOURNAL = {Math. Methods Appl. Sci.},
  FJOURNAL = {Mathematical Methods in the Applied Sciences},
    VOLUME = {40},
      YEAR = {2017},
    NUMBER = {6},
     PAGES = {1973--1985},
      ISSN = {0170-4214},
   MRCLASS = {65L60 (65J15 65L50)},
  MRNUMBER = {3624074},
MRREVIEWER = {S. Hitotumatu},
}

@article {aw2015,
    AUTHOR = {Amrein, Mario and Wihler, Thomas P.},
     TITLE = {Fully adaptive {N}ewton-{G}alerkin methods for semilinear
              elliptic partial differential equations},
   JOURNAL = {SIAM J. Sci. Comput.},
  FJOURNAL = {SIAM Journal on Scientific Computing},
    VOLUME = {37},
      YEAR = {2015},
    NUMBER = {4},
     PAGES = {A1637--A1657},
      ISSN = {1064-8275},
   MRCLASS = {65N30 (49M15 58C15 65N50)},
  MRNUMBER = {3365566},
MRREVIEWER = {H. P. Dikshit},
       DOI = {10.1137/140983537},
       URL = {https://doi.org/10.1137/140983537},
}

@article{s2008,
	Author = {Stevenson, Rob},
	Doi = {10.1090/S0025-5718-07-01959-X},
	Fjournal = {Mathematics of Computation},
	Issn = {0025-5718},
	Journal = {Math. Comp.},
	Mrclass = {65N50},
	Mrnumber = {2353951},
	Number = {261},
	Pages = {227--241},
	Title = {The completion of locally refined simplicial partitions created by bisection},
	Url = {https://doi.org/10.1090/S0025-5718-07-01959-X},
	Volume = {77},
	Year = {2008}}

@article{kpp2013,
	Author = {Karkulik, Michael and Pavlicek, David and Praetorius, Dirk},
	Doi = {10.1007/s00365-013-9192-4},
	Fjournal = {Constructive Approximation. An International Journal for Approximations and Expansions},
	Issn = {0176-4276},
	Journal = {Constr. Approx.},
	Mrclass = {65N50 (65N30 65Y20)},
	Mrnumber = {3097045},
	Number = {2},
	Pages = {213--234},
	Title = {On 2{D} newest vertex bisection: optimality of mesh-closure and {$H^1$}-stability of {$L_2$}-projection},
	Url = {https://doi.org/10.1007/s00365-013-9192-4},
	Volume = {38},
	Year = {2013}}

@article{bhp2017,
	Author = {Bespalov, Alex and Haberl, Alexander and Praetorius, Dirk},
	Doi = {10.1016/j.cma.2016.12.014},
	Fjournal = {Computer Methods in Applied Mechanics and Engineering},
	Issn = {0045-7825},
	Journal = {Comput. Methods Appl. Mech. Engrg.},
	Mrclass = {65N30 (35J05 35J25 65N12 65N22 65N50)},
	Mrnumber = {3612759},
	Mrreviewer = {Hans-Peter Helfrich},
	Pages = {318--340},
	Title = {Adaptive {FEM} with coarse initial mesh guarantees optimal convergence rates for compactly perturbed elliptic problems},
	Url = {https://doi.org/10.1016/j.cma.2016.12.014},
	Volume = {317},
	Year = {2017}}

@article{gmz2012,
	Author = {Garau, Eduardo M. and Morin, Pedro and Zuppa, Carlos},
	Doi = {10.4208/nmtma.2012.m1023},
	Fjournal = {Numerical Mathematics. Theory, Methods and Applications},
	Issn = {1004-8979},
	Journal = {Numer. Math. Theory Methods Appl.},
	Mrclass = {65N30 (35J25 35J62 65N12 65N50)},
	Mrnumber = {2911871},
	Mrreviewer = {Meng Zhao Qin},
	Number = {2},
	Pages = {131--156},
	Title = {Quasi-optimal convergence rate of an {AFEM} for quasi-linear problems of monotone type},
	Url = {https://doi.org/10.4208/nmtma.2012.m1023},
	Volume = {5},
	Year = {2012},
%	Bdsk-Url-1 = {https://doi.org/10.4208/nmtma.2012.m1023}
}

@article{gmz2011,
	Author = {Garau, Eduardo M. and Morin, Pedro and Zuppa, Carlos},
	Doi = {10.1016/j.apnum.2010.12.001},
	Fjournal = {Applied Numerical Mathematics. An IMACS Journal},
	Issn = {0168-9274},
	Journal = {Appl. Numer. Math.},
	Mrclass = {65N30 (65N12 65N50)},
	Mrnumber = {2754575},
	Number = {4},
	Pages = {512--529},
	Title = {Convergence of an adaptive {K}a\v{c}anov {FEM} for quasi-linear problems},
	Url = {https://doi.org/10.1016/j.apnum.2010.12.001},
	Volume = {61},
	Year = {2011}}

@article{bdk2012,
	Author = {Belenki, Liudmila and Diening, Lars and Kreuzer, Christian},
	Doi = {10.1093/imanum/drr016},
	Fjournal = {IMA Journal of Numerical Analysis},
	Issn = {0272-4979},
	Journal = {IMA J. Numer. Anal.},
	Mrclass = {65N30 (65N15)},
	Mrnumber = {2911397},
	Mrreviewer = {Snorre H. Christiansen},
	Number = {2},
	Pages = {484--510},
	Title = {Optimality of an adaptive finite element method for the {$p$}-{L}aplacian equation},
	Url = {https://doi.org/10.1093/imanum/drr016},
	Volume = {32},
	Year = {2012}}

@article{bdd2004,
	Author = {Binev, Peter and Dahmen, Wolfgang and DeVore, Ron},
	Doi = {10.1007/s00211-003-0492-7},
	Fjournal = {Numerische Mathematik},
	Issn = {0029-599X},
	Journal = {Numer. Math.},
	Mrclass = {65N50 (65N12 65N30 65Y20 68W25 68W40)},
	Mrnumber = {2050077},
	Mrreviewer = {Thomas Apel},
	Number = {2},
	Pages = {219--268},
	Title = {Adaptive finite element methods with convergence rates},
	Url = {https://doi.org/10.1007/s00211-003-0492-7},
	Volume = {97},
	Year = {2004}}

@article{axioms,
	Author = {Carstensen, Carsten and Feischl, Michael and Page, Marcus and Praetorius, Dirk},
	Doi = {10.1016/j.camwa.2013.12.003},
	Fjournal = {Computers \& Mathematics with Applications. An International Journal},
	Issn = {0898-1221},
	Journal = {Comput. Math. Appl.},
	Mrclass = {65N50 (65N12 65N22 65N30 65N38)},
	Mrnumber = {3170325},
	Mrreviewer = {Tsu-Fen Chen},
	Number = {6},
	Pages = {1195--1253},
	Title = {Axioms of adaptivity},
	Url = {http://dx.doi.org/10.1016/j.camwa.2013.12.003},
	Volume = {67},
	Year = {2014},
	%shorthand = {CFPP14},
	Bdsk-Url-1 = {http://dx.doi.org/10.1016/j.camwa.2013.12.003}}

@book {zeidler,
	AUTHOR = {Zeidler, Eberhard},
	TITLE = {Nonlinear functional analysis and its applications. {P}art {II}/{B}},
	% NOTE = {Nonlinear monotone operators, Translated from the German by the author and Leo F. Boron},
	PUBLISHER = {Springer-Verlag, New York},
	YEAR = {1990},
%	PAGES = {i--xvi and 469--1202},
	ISBN = {0-387-97167-X},
	MRCLASS = {47-02 (35-01 35J60 47Hxx 58-01 65Jxx)},
	MRNUMBER = {1033498},
	MRREVIEWER = {Jean Mawhin},
	DOI = {10.1007/978-1-4612-0985-0},
%	URL = {https://doi.org/10.1007/978-1-4612-0985-0},
}

@article{bhsz2011,
  title={Finite element error estimates for critical growth semilinear problems without angle conditions},
  author={Bank, Randolph E. and Holst, Michael and Szypowski, Ryan and Zhu, Yunrong},
%  journal={Preprint arXiv:1108.3661},
  year={2011},
  eprint={1108.3661},
  archivePrefix={arXiv},
%  primaryClass={math.NA},
%shorthand = {BHSZ11},
}

@book{kof1977,
%series = {Monographs and textbooks on mechanics of solids and fluids: Mechanics, Analysis; 3},
publisher = {Noordhoff, Academia},
isbn = {9028600159},
year = {1977},
title = {Function spaces},
language = {eng},
address = {Leyden, Prague},
author = {Kufner, Alois and John, Oldrich and Fučík, Svatopluk},
keywords = {Funktionenraum},
}

@book {fk1980,
	AUTHOR = {Fu\v{c}\'{i}k, Svatopluk and Kufner, Alois},
	TITLE = {Nonlinear differential equations},
	SERIES = {Studies in Applied Mechanics},
	VOLUME = {2},
	PUBLISHER = {Elsevier, Amsterdam},
	YEAR = {1980},
	PAGES = {359},
	ISBN = {0-444-99771-7},
	MRCLASS = {35-02 (34-02 47-02 49-02 58-XX 90C25)},
	MRNUMBER = {558764},
	MRREVIEWER = {Jean Mawhin},
}

@article {ghps2021,
    AUTHOR = {Gantner, Gregor and Haberl, Alexander and Praetorius, Dirk and
              Schimanko, Stefan},
     TITLE = {Rate optimality of adaptive finite element methods with
              respect to overall computational costs},
   JOURNAL = {Math. Comp.},
  FJOURNAL = {Mathematics of Computation},
    VOLUME = {90},
      YEAR = {2021},
    NUMBER = {331},
     PAGES = {2011--2040},
      ISSN = {0025-5718},
   MRCLASS = {65N30 (65N22 65N50 65Y20)},
  MRNUMBER = {4280291},
       DOI = {10.1090/mcom/3654},
       URL = {https://doi.org/10.1090/mcom/3654},
}

@article {bbimp2022,
	AUTHOR = {Becker, Roland and Brunner, Maximilian and Innerberger,
	Michael and Melenk, Jens Markus and Praetorius, Dirk},
	TITLE = {Rate-optimal goal-oriented adaptive {FEM} for semilinear
	elliptic {PDE}s},
	JOURNAL = {Comput. Math. Appl.},
	FJOURNAL = {Computers \& Mathematics with Applications. An International
	Journal},
	VOLUME = {118},
	YEAR = {2022},
	PAGES = {18--35},
	ISSN = {0898-1221},
	MRCLASS = {65N30 (65N12 65N15 65N50)},
	MRNUMBER = {4432100},
	DOI = {10.1016/j.camwa.2022.05.008},
	URL = {https://doi.org/10.1016/j.camwa.2022.05.008},
%	     %shorthand = {BBIMP22},
}

@article {ghps2018,
	AUTHOR = {Gantner, Gregor and Haberl, Alexander and Praetorius, Dirk and
	Stiftner, Bernhard},
	TITLE = {Rate optimal adaptive {FEM} with inexact solver for nonlinear
	operators},
	JOURNAL = {IMA J. Numer. Anal.},
	FJOURNAL = {IMA Journal of Numerical Analysis},
	VOLUME = {38},
	YEAR = {2018},
	NUMBER = {4},
	PAGES = {1797--1831},
	ISSN = {0272-4979},
	MRCLASS = {65N30 (65N12 65N22 65N50)},
	MRNUMBER = {3867383},
	MRREVIEWER = {Bj\"{o}rn Stinner},
	DOI = {10.1093/imanum/drx050},
	URL = {https://doi.org/10.1093/imanum/drx050},
}

@article {hpw2021,
	AUTHOR = {Heid, Pascal and Praetorius, Dirk and Wihler, Thomas P.},
	TITLE = {Energy contraction and optimal convergence of adaptive
	iterative linearized finite element methods},
	JOURNAL = {Comput. Methods Appl. Math.},
	FJOURNAL = {Computational Methods in Applied Mathematics},
	VOLUME = {21},
	YEAR = {2021},
	NUMBER = {2},
	PAGES = {407--422},
	ISSN = {1609-4840},
	MRCLASS = {65N30 (35J62 47H05 47J25 65N12 65Y20)},
	MRNUMBER = {4235817},
	DOI = {10.1515/cmam-2021-0025},
	URL = {https://doi.org/10.1515/cmam-2021-0025},
}

@article {affkp2013,
	AUTHOR = {Aurada, Markus and Feischl, Michael and F\"{u}hrer, Thomas and
	Karkulik, Michael and Praetorius, Dirk},
	TITLE = {Efficiency and optimality of some weighted-residual error
	estimator for adaptive 2{D} boundary element methods},
	JOURNAL = {Comput. Methods Appl. Math.},
	FJOURNAL = {Computational Methods in Applied Mathematics},
	VOLUME = {13},
	YEAR = {2013},
	NUMBER = {3},
	PAGES = {305--332},
	ISSN = {1609-4840},
	MRCLASS = {65N38 (65N15 65N50)},
	MRNUMBER = {3094620},
	MRREVIEWER = {Paul Andrew Martin},
	DOI = {10.1515/cmam-2013-0010},
	URL = {https://doi.org/10.1515/cmam-2013-0010},
	%shorthand = {AFFKP13},
}

@article {hpsv2021,
	AUTHOR = {Haberl, Alexander and Praetorius, Dirk and Schimanko, Stefan
	and Vohral{\'i}k, Martin},
	TITLE = {Convergence and quasi-optimal cost of adaptive algorithms for
	nonlinear operators including iterative linearization and
	algebraic solver},
	JOURNAL = {Numer. Math.},
	FJOURNAL = {Numerische Mathematik},
	VOLUME = {147},
	YEAR = {2021},
	NUMBER = {3},
	PAGES = {679--725},
	ISSN = {0029-599X},
	MRCLASS = {65N30 (35J15 65N12 65N15 65N50 68Q25)},
	MRNUMBER = {4224933},
	MRREVIEWER = {Mohammad Asadzadeh},
	DOI = {10.1007/s00211-021-01176-w},
%	URL = {https://doi.org/10.1007/s00211-021-01176-w},
}

@article {ev2013,
	AUTHOR = {Ern, Alexandre and Vohral{\'i}k, Martin},
	TITLE = {Adaptive inexact {N}ewton methods with a posteriori stopping
	criteria for nonlinear diffusion {PDE}s},
	JOURNAL = {SIAM J. Sci. Comput.},
	FJOURNAL = {SIAM Journal on Scientific Computing},
	VOLUME = {35},
	YEAR = {2013},
	NUMBER = {4},
	PAGES = {A1761--A1791},
	ISSN = {1064-8275},
	MRCLASS = {65H10 (65M08 65M22 65M50 65M60)},
	MRNUMBER = {3072765},
	MRREVIEWER = {B\"{u}lent Karas\"{o}zen},
	DOI = {10.1137/120896918},
	URL = {https://doi.org/10.1137/120896918},
}

@article {aev2011,
	AUTHOR = {El Alaoui, Linda and Ern, Alexandre and Vohral{\'i}k, Martin},
	TITLE = {Guaranteed and robust a posteriori error estimates and
	balancing discretization and linearization errors for monotone
	nonlinear problems},
	JOURNAL = {Comput. Methods Appl. Mech. Engrg.},
	FJOURNAL = {Computer Methods in Applied Mechanics and Engineering},
	VOLUME = {200},
	YEAR = {2011},
	NUMBER = {37-40},
	PAGES = {2782--2795},
	ISSN = {0045-7825},
	MRCLASS = {65N30 (65N15)},
	MRNUMBER = {2811915},
	DOI = {10.1016/j.cma.2010.03.024},
	URL = {https://doi.org/10.1016/j.cma.2010.03.024},
}

@article {bbimp2022cost,
	AUTHOR = {Becker, Roland and Brunner, Maximilian and Innerberger,
	Michael and Melenk, Jens Markus and Praetorius, Dirk},
	TITLE = {Cost-optimal adaptive iterative linearized {FEM} for
	semilinear elliptic {PDE}s},
	JOURNAL = {ESAIM Math. Model. Numer. Anal.},
	FJOURNAL = {ESAIM. Mathematical Modelling and Numerical Analysis},
	VOLUME = {57},
	YEAR = {2023},
	NUMBER = {4},
	PAGES = {2193--2225},
	ISSN = {2822-7840,2804-7214},
	MRCLASS = {65N30 (65N15 65N50 65Y20)},
	MRNUMBER = {4609880},
	DOI = {10.1051/m2an/2023036},
	URL = {https://doi.org/10.1051/m2an/2023036},
}

@article{dgs2023,
	title={Adaptive mesh refinement for arbitrary initial triangulations}, 
	author={Lars Diening and Lukas Gehring and Johannes Storn},
	year={2025},
	Journal= {Found. Comput. Math., \textup{published online first}},
	DOI={10.1007/s10208-025-09698-7},
}

@book {d2004,
	AUTHOR = {Deuflhard, Peter},
	TITLE = {Newton methods for nonlinear problems},
	SERIES = {Springer Series in Computational Mathematics},
	VOLUME = {35},
	% NOTE = {Affine invariance and adaptive algorithms, corrected printing of 2004 version},
	PUBLISHER = {Springer, Heidelberg},
	YEAR = {2006},
	DOI = {10.1007/978-3-642-23899-4},
	URL = {https://doi.org/10.1007/978-3-642-23899-4},
}

@article {bps2024,
	AUTHOR = {Brunner, Maximilian and Praetorius, Dirk and Streitberger,
	Julian},
	TITLE = {Cost-optimal adaptive {FEM} with linearization and algebraic
	solver for semilinear elliptic {PDE}s},
	JOURNAL = {Numer. Math.},
	FJOURNAL = {Numerische Mathematik},
	VOLUME = {157},
	YEAR = {2025},
	NUMBER = {2},
	PAGES = {409--445},
	ISSN = {0029-599X,0945-3245},
	MRCLASS = {65N30 (65N15 65N50 65Y20)},
	MRNUMBER = {4883774},
	DOI = {10.1007/s00211-025-01455-w},
}

@article {al1996,
	AUTHOR = {Axelsson, O. and Layton, W.},
	TITLE = {A two-level discretization of nonlinear boundary value
	problems},
	JOURNAL = {SIAM J. Numer. Anal.},
	FJOURNAL = {SIAM Journal on Numerical Analysis},
	VOLUME = {33},
	YEAR = {1996},
	NUMBER = {6},
	PAGES = {2359--2374},
	ISSN = {0036-1429},
	MRCLASS = {65N30 (65J15)},
	MRNUMBER = {1427468},
	MRREVIEWER = {A.\ V.\ Dzhishkariani},
	DOI = {10.1137/S0036142993247104},
	URL = {https://doi.org/10.1137/S0036142993247104},
}

@article {dp1992,
	AUTHOR = {Deuflhard, Peter and Potra, Florian A.},
	TITLE = {Asymptotic mesh independence of {N}ewton-{G}alerkin methods
	via a refined {M}ysovski\u i\ theorem},
	JOURNAL = {SIAM J. Numer. Anal.},
	FJOURNAL = {SIAM Journal on Numerical Analysis},
	VOLUME = {29},
	YEAR = {1992},
	NUMBER = {5},
	PAGES = {1395--1412},
%	ISSN = {0036-1429},
%	MRCLASS = {65J15 (65N30 65N35)},
%	MRNUMBER = {1182736},
%	MRREVIEWER = {A.\ V.\ Dzhishkariani},
	DOI = {10.1137/0729080},
%	URL = {https://doi.org/10.1137/0729080},
}

@article {wsd2005,
	AUTHOR = {Weiser, Martin and Schiela, Anton and Deuflhard, Peter},
	TITLE = {Asymptotic mesh independence of {N}ewton's method revisited},
	JOURNAL = {SIAM J. Numer. Anal.},
	FJOURNAL = {SIAM Journal on Numerical Analysis},
	VOLUME = {42},
	YEAR = {2005},
	NUMBER = {5},
	PAGES = {1830--1845},
%	ISSN = {0036-1429,1095-7170},
%	MRCLASS = {65H10 (65N22 65N30)},
%	MRNUMBER = {2139225},
%	MRREVIEWER = {R.\ P.\ Tewarson},
	DOI = {10.1137/S0036142903434047},
%	URL = {https://doi.org/10.1137/S0036142903434047},
}

@article {abpr1986,
	AUTHOR = {Allgower, E. L. and B\"ohmer, K. and Potra, F. A. and
	Rheinboldt, W. C.},
	TITLE = {A mesh-independence principle for operator equations and their
	discretizations},
	JOURNAL = {SIAM J. Numer. Anal.},
	FJOURNAL = {SIAM Journal on Numerical Analysis},
	VOLUME = {23},
	YEAR = {1986},
	NUMBER = {1},
	PAGES = {160--169},
	DOI = {10.1137/0723011},
}

@incollection {lm1954,
	AUTHOR = {Lax, P. D. and Milgram, A. N.},
	TITLE = {Parabolic equations},
	BOOKTITLE = {Contributions to the theory of partial differential equations},
	SERIES = {Ann. of Math. Stud.},
	VOLUME = {no. 33},
	PAGES = {167--190},
	PUBLISHER = {Princeton Univ. Press, Princeton, NJ},
	YEAR = {1954},
	MRCLASS = {35.0X},
	MRNUMBER = {67317},
	MRREVIEWER = {L.\ G\aa rding},
}

@article {h2023,
	AUTHOR = {Heid, Pascal},
	TITLE = {A short note on an adaptive damped {N}ewton method for
	strongly monotone and {L}ipschitz continuous operator
	equations},
	JOURNAL = {Arch. Math. (Basel)},
	FJOURNAL = {Archiv der Mathematik},
	VOLUME = {121},
	YEAR = {2023},
	NUMBER = {1},
	PAGES = {55--65},
	MRNUMBER = {4609626},
	DOI = {10.1007/s00013-023-01858-x},
}
}

\appendix

\section{Proof of full R-linear convergence}\label{appendix:linearConvergence}

We provide full details for Step~8 of the proof of Theorem~\ref{theorem:fullRLinearConvergence}. 

\begin{proof}[\textbf{Proof of Step~8 (tail-summability with respect to \(\ell\) and \(k\)) in Theorem~\ref{theorem:fullRLinearConvergence}}]
    Given a double index \((\ell', k') \in \QQ\), tail-summability considers the infinite sum
    \begin{equation}
        \label{eq:tailsummability:split}
        \sum_{\substack{(\ell, k) \in \QQ \\ \abs{\ell', k'} < \abs{\ell,k}}}
        \Eta_{\ell}^{k}
        =
        \sum_{k = k' + 1}^{\kmax[\ell']} \Eta_{\ell'}^{k}
        +
        \sum_{\ell = \ell' + 1}^{\lmax}
        \sum_{k = 0}^{\kmax[\ell]}
        \Eta_{\ell}^{k}.
    \end{equation}
    In order to bound this sum, we distinguish three cases depending on \(\ell'\) and \(\elll\).

	In case 1, suppose that \(\ell' = \elll < \infty\).
	Then, $\kmax[\ell'] = \infty$ and the sum from \(\ell = \ell'+1\) to \(\elll\)
    on the right-hand side of~\eqref{eq:tailsummability:split} vanishes.
    The tail-summability follows from the summability~\eqref{eq:tailsummability:k} in \(k\)
    \begin{equation*}
        \sum_{\substack{(\ell, k) \in \QQ \\ \abs{\ell', k'} < \abs{\ell,k}}}
        \Eta_{\ell}^{k}
        =
        \sum_{k = k' + 1}^{\infty} \Eta_{\ell'}^{k}
        \eqreff{eq:tailsummability:k}\lesssim
        \Eta_{\ell'}^{k'}.
    \end{equation*}

	In case 2, suppose that \(\ell' < \elll < \infty\).
    Then, \(\kmin < \kk[\ell] < \infty\) for all \(\ell' \leq \ell < \elll\) and \(\kk[\elll] = \infty\).
    The tail-summability~\eqref{eq:tailsummability:k} in \(k\) on the levels \(\ell'\) and \(\elll\) provides
    \[
        \sum_{k=k'+1}^{\kk[\ell']} \Eta_{\ell'}^k
        +
        \sum_{k = 0}^{\infty}
        \Eta_{\elll}^{k}
        \eqreff{eq:tailsummability:k}\lesssim
        \Eta_{\ell'}^{k'}
        +
        \Eta_{\elll}^{0}
    .
    \]
    The quasi-contraction~\eqref{eq:quasierror:quasicontraction:case2} and the geometric sum yield
    \begin{equation}
        \label{eq:tailsummability:case2:ell}
        \sum_{\ell=\ell'+1}^{\elll-1}
        \sum_{k=0}^{\kk[\ell]}
        \Eta_\ell^k
        \eqreff{eq:quasierror:quasicontraction:case2}\lesssim
        \sum_{\ell=\ell'+1}^{\elll-1}
        \bigg(\sum_{k=0}^{\kk[\ell]} \rr_0^k \bigg)
        \Eta_\ell^0
        \lesssim
        \sum_{\ell=\ell'+1}^{\elll-1}
        \Eta_\ell^0.
    \end{equation}
    The combination of the two previous displayed formulas with the split~\eqref{eq:tailsummability:split},
    the stability estimates~\eqref{eq:quasierror:stability_in_refinement}
    and~\eqref{eq:quasierror:contraction_in_linearization}, and
    the tail-summability~\eqref{eq:contractionFinalIterates} with respect to \(\ell\) proves
    \begin{align*}
        \sum_{\substack{(\ell, k) \in \QQ \\ \abs{\ell', k'} < \abs{\ell,k}}}
        \Eta_{\ell}^{k} \;
        &\eqreff*{eq:tailsummability:split}{=}
        \sum_{k = k' + 1}^{\kmax[\ell']} \Eta_{\ell'}^{k}
        +
        \sum_{\ell = \ell' + 1}^{\elll-1}
        \sum_{k = 0}^{\kmax[\ell]}
        \Eta_{\ell}^{k}
        +
        \sum_{k = 0}^{\infty}
        \Eta_{\elll}^{k}
        \\
        &\lesssim
        \Eta_{\ell'}^{k'}
        +
        \sum_{\ell=\ell'+1}^{\elll}
        \Eta_\ell^0
        \eqreff{eq:quasierror:stability_in_refinement}\lesssim
        \Eta_{\ell'}^{k'}
        +
        \sum_{\ell=\ell'}^{\elll-1}
        \Eta_\ell^{\kk}
        \eqreff{eq:contractionFinalIterates}\lesssim
        \Eta_{\ell'}^{k'}
        +
        \Eta_{\ell'}^{\kk}
        \eqreff{eq:quasierror:contraction_in_linearization}\lesssim
        \Eta_{\ell'}^{k'}.
    \end{align*}

	In the remaining case 3, suppose that \(\ell' < \elll = \infty\).
    Then, $\kmin \leq \kmax[\ell] < \infty$ for all $(\ell, k) \in \QQ$ with \(\abs{\ell, k} > \abs{\ell', k'}\).
    The argumentation for~\eqref{eq:tailsummability:case2:ell} in case 2 applies to the infinite sum as well and gives
    \[
        \sum_{\ell=\ell'+1}^{\infty}
        \sum_{k=0}^{\kk[\ell]}
        \Eta_\ell^k
        \eqreff{eq:quasierror:quasicontraction:case2}\lesssim
        \sum_{\ell=\ell'+1}^{\infty}
        \bigg(\sum_{k=0}^{\kk[\ell]} \rr_0^k \bigg)
        \Eta_\ell^0
        \lesssim
        \sum_{\ell=\ell'+1}^{\infty}
        \Eta_\ell^0
        \eqreff{eq:quasierror:stability_in_refinement}\lesssim
        \sum_{\ell=\ell'}^{\infty}
        \Eta_\ell^{\kk}
        \eqreff{eq:contractionFinalIterates}\lesssim
        \Eta_{\ell'}^{\kk}
        \eqreff{eq:quasierror:contraction_in_linearization}\lesssim
        \Eta_{\ell'}^{k'}.
    \]
    This, the tail-summability~\eqref{eq:tailsummability:k} in \(k\), and the split~\eqref{eq:tailsummability:split} show
    \[
        \sum_{\substack{(\ell, k) \in \QQ \\ \abs{\ell', k'} < \abs{\ell,k}}}
        \Eta_{\ell}^{k}
        \eqreff{eq:tailsummability:split}=
        \sum_{k = k' + 1}^{\kmax[\ell']} \Eta_{\ell'}^{k}
        +
        \sum_{\ell = \ell' + 1}^{\infty}
        \sum_{k = 0}^{\kmax[\ell]}
        \Eta_{\ell}^{k}
        \eqreff{eq:tailsummability:k}\lesssim
        \Eta_{\ell'}^{k'}.
    \]

    In all three cases, we have verified the estimate
    \[
        \sum_{\substack{(\ell, k) \in \QQ \\ \abs{\ell', k'} < \abs{\ell,k}}}
        \Eta_{\ell}^{k}
        \lesssim
        \Eta_{\ell'}^{k'}.
    \]
    The tail-summability equivalence from~\cite[Lemma~2]{bfmps2025}
    concludes the proof of Step~8.
\end{proof}

\section{Proof of optimal convergence rates}\label{appendix:optimalRates}
This section is devoted to the proof of Theorem~\ref{theorem:optimalRates}.
Following the state-of-the-art convergence analysis of adaptive algorithms with inexact solution
\cite{bps2024,bfmps2025}, it departs with a perturbation argument for the discretization error estimator for
the final iterates.
\begin{lemma}[estimator equivalence]\label{lemma:estimatorEquivalence}
	Suppose the assumptions of Theorem~\ref{theorem:optimalRates}, in particular,
	let $0<\lambdalin< \lambdalin^\exact \coloneqq \min\{1, \alpha/\overline{C}_{\textup{stab}}\}$ with $\CCstab$ from~\eqref{eq:uniform_constants}.
	Then, it holds that
	\begin{equation}\label{eq:estimatorEquivalence}
		\bigg( 1 - \frac{\lambdalin}{\lambdalin^\exact} \bigg) \,
		\eta_\ell(u_\ell^{\kmax})
		\le
		\eta_\ell(u_\ell^\exact)
		\le
		\bigg( 1 + \frac{\lambdalin}{\lambdalin^\exact} \bigg) \,
		\eta_\ell(u_\ell^{\kmax})
		\quad\text{for all } (\ell, \kmax) \in \QQ.
	\end{equation}
	Moreover, for $0 < \theta < \theta^\exact$ with $\theta^\exact$ and the definition of $\thetamark$ given in~\eqref{eq:thetamark},
	the Dörfler criterion for \(u_\ell^\exact\) with parameter \(\thetamark\) from~\eqref{eq:thetamark} implies
	the Dörfler criterion for \(u_\ell^\kmax\) with parameter \(\theta\) satisfying~\eqref{eq:thetamark},
	i.e., for any \(\RR_{\ell} \subseteq \TT_{\ell}\), there holds the implication
	\begin{equation}\label{eq:equivalence_Doerfler}
		\thetamark \, \eta_{\ell}(u_\ell^\exact)^2
		\le
		\eta_{\ell}(\RR_{\ell}; u_\ell^\exact)^2
		\quad
		\implies
		\quad
		\theta \, \eta_{\ell}(u_\ell^{\kmax})^2
		\le
		\eta_{\ell}(\RR_{\ell}; u_\ell^{\kmax})^2.
	\end{equation}
\end{lemma}
\begin{proof}
	The proof consists of two steps.

	\emph{Step 1.}
	The linearization error estimate~\eqref{eq:linearization_estimator} and
	the stopping criterion~\eqref{algorithm:AILFEM:stopping} prove
	\[
		\enorm{u^\exact_\ell - u_\ell^\kmax}
		\eqreff*{eq:linearization_estimator}\le
		\alpha^{-1} \, \Vert F - \AA u_\ell^\kmax \Vert_{\XX_\ell'}
		\eqreff*{algorithm:AILFEM:stopping}\le
		\lambdalin \alpha^{-1} \, \eta_\ell(u_\ell^\kmax).
	\]
	This and the stability~\eqref{axiom:stability} show, for any subset $\UU_{\ell} \subseteq \TT_\ell$, that
	\begin{align}
		\label{eq:estimatorEquivalenceStep2}
		\eta_\ell(\UU_\ell, u_\ell^\exact)
		\,&\eqreff*{axiom:stability}\le\,
		\eta_\ell(\UU_\ell, u_\ell^\kmax)
		+
		\overline{C}_{\textup{stab}} \,
		\enorm{u^\exact_\ell - u_\ell^\kmax}
		\le
		\eta_\ell(\UU_\ell, u_\ell^\kmax)
		+
		(\lambdalin/\lambdalin^\exact) \,
		\eta_\ell(u_\ell^\kmax),
		\\
		\intertext{and similarly,}
		\label{eq:estimatorEquivalenceStep}
		\eta_\ell(\UU_\ell, u_\ell^\kmax)
		&\le
		\eta_\ell(\UU_\ell, u_\ell^\exact)
		+
		(\lambdalin/\lambdalin^\exact) \,
		\eta_\ell(u_\ell^\kmax).
	\end{align}
	The choice $\UU_\ell = \TT_\ell$ and the choice $0 < \lambdalin < \lambdalin^\exact$ allow
	to verify~\eqref{eq:estimatorEquivalence}.

	\emph{Step 2.}
	For \(\UU_\ell = \RR_{\ell} \subseteq \TT_\ell\) with
	\(\thetamark^{1/2} \, \eta_{\ell}(u_\ell^\exact) \le \eta_{\ell}(\RR_{\ell}, u_\ell^\exact)\),
	the estimates~\eqref{eq:estimatorEquivalenceStep2}--\eqref{eq:estimatorEquivalenceStep}, and~\eqref{eq:thetamark} imply
	\begin{align*}
		[1 - (\lambdalin/\lambdalin^\exact)]
		\, \thetamark^{1/2} \,
		\eta_\ell(u_\ell^\kmax)
		\, &\eqreff*{eq:estimatorEquivalenceStep}\le \,
		\thetamark^{1/2} \, \eta_\ell(u_\ell^\exact)
		\le
		\eta_\ell(\RR_\ell, u_\ell^\exact)
		\\
		&\eqreff*{eq:estimatorEquivalenceStep2}\le \,
		\eta_\ell(\RR_\ell, u_\ell^\kmax)
		+
		(\lambdalin/\lambdalin^\exact) \,
		\eta_\ell(u_\ell^\kmax)
		\\
		&\eqreff*{eq:thetamark}=\, \,
		\eta_\ell(\RR_\ell, u_\ell^\kmax)
		+
		\big[\thetamark^{1/2} \, (1 - (\lambdalin/\lambdalin^\exact)) - \theta^{1/2}] \,
		\eta_\ell(u_\ell^\kmax).
	\end{align*}
	This yields \(\theta^{1/2} \, \eta_\ell(u_\ell^\kmax) \le \eta_\ell(\RR_\ell, u_\ell^\kmax)\) and concludes the proof.
\end{proof}

\begin{proof}[\textbf{Proof of Theorem~\ref{theorem:optimalRates}}]
	Without loss of generality assume that \(\norm{u^\exact}_{\mathbb{A}_r} < \infty\).
	Since full R-linear convergence from Theorem~\ref{theorem:fullRLinearConvergence} implies
	the equivalence of convergence rates with respect to the number of degrees of freedom and
	to the computational cost \cite[Corollary~11]{bfmps2025}, it suffices to prove
	\begin{equation}\label{eq:optimalRates}
		\copt \, \norm{u^\exact}_{\mathbb{A}_r} \le  \sup_{(\ell, k) \in \QQ}
		(\# \TT_\ell)^r \,
		\Eta_\ell^{k}
		\le
		\Copt \,
		\max\{\norm{u^\exact}_{\mathbb{A}_r}, \, \Eta_0^0\}
	\end{equation}
	The proof is subdivided into four steps.

	\emph{Step 1.}
	Let \(0 < \thetamark < \theta^\exact\) as defined in~\eqref{eq:thetamark}
	and fix any \(0 \le \ell' < \lmax\).
	Stability~\eqref{axiom:stability} with constant \(\overline{C}_{\textup{stab}}\) from~\eqref{eq:uniform_constants}and
	discrete reliability~\eqref{axiom:discreteReliability} enable
	the application of~\cite[Lemma~4.14]{axioms} asserting
	the existence of a set \(\RR_{\ell^\prime} \subseteq \TT_{\ell^\prime}\) such that
	\begin{equation*}
		\# \RR_{\ell^\prime}
		\lesssim
		\norm{u^\exact}_{\mathbb{A}_r}^{1/r} \,
		[\eta_{\ell^\prime}(u_{\ell^\prime}^\exact)]^{-1/r}
		\quad \text{and} \quad
		\thetamark \, \eta_{\ell^\prime}(u_{\ell^\prime}^\exact)
		\le
		\eta_{\ell^\prime}(\RR_{\ell^\prime}, u_{\ell^\prime}^\exact).
	\end{equation*}
	Since~\eqref{eq:equivalence_Doerfler} implies that $\RR_{\ell'}$ satisfies also the Dörfler criterion
	$\theta \, \eta_{\ell'}(u_{\ell'}^\kmax)^2 \le \eta_{\ell'}(\RR_{\ell'}; u_{\ell'}^\kmax)^2$
	for the final iterate $u_{\ell'}^{\kmax}$ with $0 < \theta < \theta^\exact$,
	this proves
	\begin{equation*}
		\# \MM_{\ell'}
		\le
		\Cmark \,
		\# \RR_{\ell'}
		\lesssim
		\norm{u^\exact}_{\mathbb{A}_r}^{1/r} \,
		[
		\eta_{\ell^\prime}(u_{\ell^\prime}^{\exact})
		]^{-1/r}.
	\end{equation*}
	Full R-linear convergence~\eqref{eq:fullRLinearConvergence},
	the stopping criterion~\eqref{algorithm:AILFEM:stopping},
	and the equivalence~\eqref{eq:estimatorEquivalence} prove
	\begin{equation*}
		\Eta_{\ell'+1}^0
		\eqreff{eq:fullRLinearConvergence}\lesssim
		\Eta_{\ell'}^\kmax
		=
		\norm{F - \AA u_{\ell'}^\kmax}_{\XX_{\ell'}'}
		+
		\eta_{\ell'}(u_{\ell'}^{\kmax})
		\eqreff{algorithm:AILFEM:stopping}\lesssim
		\eta_{\ell'}(u_{\ell'}^{\kmax})
		\eqreff{eq:estimatorEquivalence}\lesssim
		\eta_{\ell'}(u_{\ell'}^\exact).
	\end{equation*}
	Consequently, the two previous displayed formulas establish
	\begin{equation}
		\label{eq:R_Delta}
		\# \MM_{\ell'}
		\lesssim
		\norm{u^\exact}_{\mathbb{A}_r}^{1/r} \,
		[ \eta_{\ell'}(u_{\ell'}^\exact) 	]^{-1/r}
		\lesssim
		\norm{u^\exact}_{\mathbb{A}_r}^{1/r} \,
		(\Eta_{\ell^\prime+1}^0)^{-1/r}.
	\end{equation}

	\emph{Step 2.}
	For $(\ell, k) \in \QQ$, full R-linear convergence \eqref{eq:fullRLinearConvergence} and the geometric series prove
	\begin{align}\label{eq:lin_cv_sum}
		\begin{split}
			\sum_{\substack{ (\ell',k') \in \QQ \\ |\ell',k'| \le |\ell,k|}}
			(\Eta_{\ell'}^{k'})^{-1/r}
			&\eqreff{eq:fullRLinearConvergence}
			\lesssim
			(\Eta_{\ell}^{k})^{-1/r}
			\sum_{\substack{(\ell',k') \in \QQ \\ |\ell',k'| \le |\ell,k|}}
			(\qlin^{1/r})^{|\ell, k| - |\ell', k'|}
			\lesssim
			(\Eta_{\ell}^{k})^{-1/r}.
		\end{split}
	\end{align}
	The mesh-closure estimate from~\cite{bdd2004, s2008, kpp2013, dgs2023} reads
	\begin{equation}\label{eq:meshClosure}
		\# \TT_\ell - \# \TT_0
		\lesssim
		\sum_{\ell' = 0}^{\ell-1} \# \MM_{\ell'}.
	\end{equation}
	This and the estimates~\eqref{eq:R_Delta}--\eqref{eq:lin_cv_sum} verify
	\begin{equation}
		\label{eq:optimality:final}
		\begin{split}
			\# \TT_\ell - \# \TT_0 \, \,
			&\eqreff*{eq:meshClosure}\lesssim \, \,
			\sum_{\ell' = 0}^{\ell-1} \# \MM_{\ell'}
			\eqreff{eq:R_Delta}\le
			\norm{u^\exact}_{\mathbb{A}_r}^{1/r} \,
			\sum_{\ell' = 0}^{\ell-1}
			(\Eta_{\ell'+1}^{0})^{-1/r}
			\\
			&\le
			\norm{u^\exact}_{\mathbb{A}_r}^{1/r}
			\sum_{\substack{(\ell',k') \in \QQ \\ |\ell',k'| \le |\ell,k|}}
			(\Eta_{\ell'}^{k'})^{-1/r}
			\eqreff{eq:lin_cv_sum}
			\lesssim
			\norm{u^\exact}_{\mathbb{A}_s}^{1/r}
			(\Eta_{\ell}^{k})^{-1/r}
			\quad \text{ for all } (\ell, k) \in \QQ.
		\end{split}
	\end{equation}
	For \(\ell \geq 1\), note that \(1 \le \# \TT_{\ell} - \# \TT_{0}\) yields
	$\# \TT_{\ell} - \# \TT_{0} +1 \le 2 \, (\#\TT_{\ell} - \# \TT_{0})$.
	For all $\TT_{\ell} \in \T$, an elementary calculation~\cite[Lemma~22]{bhp2017} shows that
	\[
		\# \TT_{\ell}
		\le
		\# \TT_{0} \, (\# \TT_{\ell} - \# \TT_{0} +1)
		\leq
		2 \, \# \TT_{0} \, (\# \TT_{\ell} - \# \TT_{0}).
	\]
	Hence, it follows from~\eqref{eq:optimality:final} that
	\[
		(\# \TT_{\ell})^r \, \Eta_\ell^k
		\lesssim
		\norm{u^\exact}_{\mathbb{A}_r}
		\quad \text{ for all } (\ell, k) \in \QQ \text{ with } \ell \ge 1.
	\]
	For $\ell=0$, full R-linear convergence~\eqref{eq:fullRLinearConvergence} proves that
	\[
		(\# \TT_0)^r \,
		\Eta_0^k
		\eqreff{eq:fullRLinearConvergence}\lesssim
		\Eta_0^0
		\quad \text{ for all } (0, k) \in \QQ.
	\]
	The combination of the two previous formulas results in the upper bound in~\eqref{eq:optimalRates}.

	\emph{Step 3.}
	In the first case \(\lmax = \infty\) with \(\# \TT_{\ell} \to \infty\) as \(\ell \to \infty\),
	for given \(N \in \N\),
	choose the maximal index \(\ell' \in \N_0\) with \(\# \TT_{\ell'} - \# \TT_{0} \le N\).
	The NVB algorithm ensures
	\begin{equation*}
		N + 1 < \# \TT_{\ell'+1} - \# \TT_{0} + 1
		\le
		\# \TT_{\ell'+1}
		\lesssim
		\# \TT_{\ell'}.
	\end{equation*}
	Since \(\TT_{\ell'} \in \T\), stability~\eqref{axiom:stability} and
	the a posteriori estimate~\eqref{eq:linearization_estimator} prove that
	\begin{equation*}
		\min_{\TT_{\rm opt} \in \T_N} \eta_{\rm opt}(u_{\rm opt}^\exact)
		\le
		\eta_{\ell'}(u_{\ell'}^\exact)
		\eqreff{axiom:stability}\lesssim
		\eta_{\ell'}(u_{\ell'}^{k'})
		+
		\norm{F- u_{\ell'}^{k'}}_{\XX_{\ell'}'}
		\\
		\eqreff{eq:linearization_estimator}\lesssim
		\Eta_{\ell'}^{k'} \quad \text{ for }k' \in \N_0 \text{ with } (\ell', k') \in \QQ.
	\end{equation*}
	The two previous displayed estimates result in
	\begin{equation}\label{eq:optimality:combined}
		(N + 1)^s \,
		\min_{\TT_{\rm opt} \in \T_N} \eta_{\rm opt}(u_{\rm opt}^\exact)
		\lesssim
		(\# \TT_{\ell'})^s \, \Eta_{\ell'}^{k'}
		\le
		\sup_{(\ell, k) \in \QQ} (\# \TT_{\ell})^s \, \Eta_{\ell}^{k}.
	\end{equation}
	The supremum over all \(N \in \N\) shows the lower bound in~\eqref{eq:optimalRates} for \(\lmax = \infty\).

	\emph{Step~4.}
	In the remaining case (\(\lmax < \infty\)),
	Lemma~\ref{lemma:kLoop} yields \(\eta_\elll(u_\elll^\exact) = 0\) and \(u_\elll^\exact = u^\exact\).
	If \(\lmax = 0\), then \(\| u^\exact \|_{\mathbb{A}_r} = 0\) and there is nothing to prove.
	If \(\lmax > 0\), let \(0 \le N < \# \TT_{\lmax} - \# \TT_{0}\).
	As in Step~3, the maximal index \(0 \le \ell' < \lmax\) with \(\# \TT_{\ell'} - \# \TT_{0} \le N\) satisfies
	\begin{equation*}
		\norm{u^\exact}_{\mathbb{A}_r}
		=
		\sup_{0 \le N < \# \TT_{\lmax} - \# \TT_{0}}
		\bigl(
			(N+1)^s
			\min_{\TT_{\textup{opt}} \in \T_N} \eta_{\textup{opt}}(u_{\mathrm{opt}}^\exact)
		\bigr)
		\eqreff{eq:optimality:combined}\lesssim
		\sup_{(\ell, k) \in \QQ} (\# \TT_{\ell})^r \, \Eta_{\ell}^{k}.
	\end{equation*}
	This proves the lower bound in~\eqref{eq:optimalRates} also for \(\lmax < \infty\) and the proof is complete.
\end{proof}



\end{document}